\documentclass[10pt]{amsart}
\usepackage{amsmath,amssymb}
\usepackage[hyperindex]{hyperref}
\usepackage[nobysame]{amsrefs}

\def\codim{\mathalpha{\rm codim}}

\numberwithin{equation}{section}

\newtheorem{theorem}[equation]{Theorem}
\newtheorem{lemma}[equation]{Lemma}
\newtheorem{corollary}[equation]{Corollary}
\newtheorem{prop}[equation]{Proposition}
\newtheorem{question}[equation]{Question}

\theoremstyle{definition}
\newtheorem{definition}[equation]{Definition}
\newtheorem{example}[equation]{Example}

\theoremstyle{remark}
\newtheorem{remark}[equation]{Remark}

\begin{document}

\title{On the capability of finite groups of class two and prime exponent}
\author{Arturo Magidin}
\address{Mathematics Dept. University of Louisiana at Lafayette,
  217 Maxim Doucet Hall, P.O.~Box 41010, Lafayette LA 70504-1010}
\email{magidin@member.ams.org}


\begin{abstract}
We consider the capability of $p$-groups of class two and odd prime
exponent. The question of capability is shown to be equivalent to a
statement about vector spaces and linear transformations, and using
the equivalence we give proofs of some old results and several new
ones. In particular, we establish a number of new necessary and new sufficient 
conditions for capability, including a sufficient condition based only on the
ranks of $G/Z(G)$ and $[G,G]$. Finally, we characterise the capable
groups among the $5$-generated groups in this class. 
\end{abstract}

\maketitle

\section{Introduction.}\label{sec:intro}

In his landmark paper~\cite{hallpgroups} on the classification of
finite
$p$-groups, P.~Hall remarked:
\begin{quote}
The question of what conditions a group $G$ must fulfill in order that
it may be the central quotient group of another group $H$, $G\cong
H/Z(H)$, is an interesting one. But while it is easy to write down a
number of necessary conditions, it is not so easy to be sure that they
are sufficient.
\end{quote}
Following 
\cite{hallsenior}, we make the following
definition:
\begin{definition} A group $G$ is said to be \textit{capable} if and
  only if
there exists a group $H$ such that $G\cong H/Z(H)$.
\end{definition}

Capability of groups was first studied 
in~\cite{baer},
where, as a corollary of deeper investigations, he characterised the
capable groups that are direct sums of cyclic groups.  Capability of
groups has received renewed attention in recent years, thanks to
results 
in \cite{beyl} characterising the
capability of a group in terms of its epicenter; and more recently to
work of 
\cite{ellis} that describes the epicenter in
terms of the nonabelian tensor square of the group.

We will consider here the special case of nilpotent groups of class
two and exponent an odd prime~$p$. This case was studied
in~\cite{heinnikolova}, and also addressed elsewhere (e.g., Prop.~9
in~\cite{ellis}). As noted in the final paragraphs
of~\cite{baconkappe}, currently available techniques seem insufficient
for a characterisation of the capable finite $p$-groups of class~$2$,
but a characterisation of the capable finite groups of class~$2$ and
exponent~$p$ seems a more modest and possibly attainable
goal. The present work is a contribution towards achieving that goal.
We began to study this situation in \cite{capablep}; here we will
introduce what I believe is clearer notation as well as a general
setting to frame the discussion. We will also be able to use our
methods to extend the necessary condition from~\cite{heinnikolova} to
include groups that do not satisfy $Z(G)=[G,G]$,
and to provide a short new proof of the sufficient condition
from~\cite{ellis}. We will also prove a sufficient condition which is
closer in flavor to the necessary condition of Heineken and~Nikolova.

In the remainder of this section we will give basic definitions and
our notational conventions. In Section~\ref{sec:witness} we will
obtain a necessary and sufficient condition for the capability of a
given group $G$ of class at most two and exponent~$p$ in terms of a
``canonical witness.'' In Section~\ref{sec:linalg} we discuss the
general setting in which we will work from the point of view of Linear
Algebra, and the specific instance of that general setting that occurs
in this work is introduced. We proceed in Section~\ref{sec:basics} to
obtain several easy consequences of this set-up, and their equivalent
statements in terms of capability. In Section~\ref{sec:dimcounting} we
use a counting argument to give a sufficient condition for the
capability of~$G$ that depends only on the ranks of $G/Z(G)$ and
$[G,G]$.  Next, in Section~\ref{sec:neccond}, we prove a slight
strengthening of the necessary condition first proven
in~\cite{heinnikolova}, which also depends only on the ranks of $G/Z(G)$
and $[G,G]$.

In Section~\ref{sec:applications} we characterise the capable groups
among the $5$-generated $p$-groups of prime exponent and class at most
two. We also give an alternative geometric proof for a key part of the
classification in the $4$-generated case, since it highlights the way
in which the set-up using linear algebra allows us to invoke other tools
(in this case, algebraic geometry) to study our problem.  We should mention that
the approach using linear algebra and geometry has been used before in
the study of groups of class two and exponent~$p$; in particular, the
work of Brahana~\cites{brahanalines,brahanaplucker} exploits geometry
in a very striking fashion to classify certain groups of class two
and exponent~$p$ in terms of points, lines, planes, and spaces in a
projective space over~$\mathbb{F}_p$. This
classification, found in~\cite{brahanalines}, will also play a role in
our classification in the $5$-generated case, allowing us to deal with
certain groups of order $p^{8}$ and $p^{9}$.

Finally, in Section~\ref{sec:finalsec} we discuss some of the limits
of our results so far, and state some questions. 

Throughout the paper $p$ will be an odd prime, and $\mathbb{F}_p$ will
denote the field with $p$ elements. All groups will be
written multiplicatively, and the identity element will be denoted by
$e$; if there is danger of ambiguity or confusion, we will use $e_G$
to denote the identity of the group~$G$. The center of $G$ is denoted
by $Z(G)$. Recall that if $G$ is a group, and $x,y\in G$, the
commutator of $x$ and $y$ is defined to be $[x,y]=x^{-1}y^{-1}xy$; we use $x^y$ to
denote the conjugate $y^{-1}xy$. We write commutators left-normed, so
that $[x,y,z] = [[x,y],z]$.  Given subsets $A$ and $B$ of~$G$ we
define $[A,B]$ to be the subgroup of $G$ generated by all elements of
the form $[a,b]$ with $a\in A$, $b\in B$.  The terms of the lower
central series of $G$ are defined recursively by letting $G_1=G$, and
$G_{n+1}=[G_n,G]$.  A group is \textit{nilpotent of class at most~$k$}
if and only if $G_{k+1}=\{e\}$, if and only if $G_k\subset Z(G)$. We
usually drop the ``at most'' clause, it being understood. The class of
all nilpotent groups of class at most~$k$ is denoted
by~$\mathfrak{N}_k$. Though we will sometimes use indices to denote
elements of a family of groups, it will be clear from context that we
are not refering to the terms of the lower central series in those
cases. 

The following commutator identities are well known, and may be
verified by direct calculation:
\begin{prop} Let $G$ be any group. Then for all $x,y,z\in G$,
\begin{itemize}
\item[(a)] $[xy,z] = [x,z][x,z,y][y,z]$.
\item[(b)] $[x,yz] = [x,z][z,[y,x]][x,y]$.
\item[(c)] $[x,y,z][y,z,x][z,x,y] \equiv e \pmod{G_4}$.
\item[(d)] $[x^r,y^s] \equiv
  [x,y]^{rs}[x,y,x]^{s\binom{r}{2}}[x,y,y]^{r\binom{s}{2}}
  \pmod{G_4}$.
\item[(e)] $[y^r,x^s] \equiv
  [x,y]^{-rs}[x,y,x]^{-r\binom{s}{2}}[x,y,y]^{-s\binom{r}{2}}
  \pmod{G_4}$.
\end{itemize}
Here, $\binom{n}{2} = \frac{n(n-1)}{2}$ for all integers~$n$.
\label{prop:commident}
\end{prop}

As in~\cite{capable}, our starting tool will be the nilpotent product of
groups, specifically the $2$-nilpotent and $3$-nilpotent product of
cyclic groups. We restrict Golovin's original definition
\cite{golovinnilprods} to the situation we will consider:

\begin{definition} Let $A_1,\ldots,A_n$ be nilpotent groups of class
  at most~$k$. The
  $k$-nilpotent product of $A_1,\ldots,A_n$, denoted by
  $A_1\amalg^{\germ N_k}\cdots \amalg^{\germ N_k} A_n$, is defined to
  be the group $G=F/F_{k+1}$, where $F$ is the free product of the
  $A_i$, $F=A_1*\cdots*A_n$, and $F_{k+1}$ is the $(k+1)$-st term of
  the lower central series of~$F$.
\end{definition}

From the definition it is clear that the $k$-nilpotent product is the
coproduct in the variety~${\germ N}_k$, so it will have the usual
universal property. Note that if the $A_i$ lie in ${\germ N}_k$, and
$G$ is the $(k+1)$-nilpotent product of the $A_i$, then $G\in{\germ
N}_{k+1}$ and $G/G_{k+1}$ is the $k$-nilpotent product of the $A_i$.

When we take the $k$-nilpotent product of cyclic $p$-groups, with
$p\geq k$, we may write each element uniquely as a product of basic
commutators of weight at most~$k$ on the generators, as shown in
in~\cite{struikone}*{Theorem~3}; see~\cite{hall}*{\S 12.3} for the
definition of basic commutators which we will use. In our
applications, where each cyclic group is of order~$p$, the order of
each basic commutator is likewise equal to~$p$.

Finally, when we say that a group is \textit{$k$-generated} we mean
that it can be generated by $k$ elements, but may in fact need
less. If we want to say that it can be generated by $k$ elements, but
not by $m$ elements for some $m<k$, we will say that it is
\textit{minimally $k$-generated}, or \textit{minimally generated by
$k$ elements}.

\section{A canonical witness.}\label{sec:witness}

The idea behind our development is the following: given a group~$G$, we
attempt to construct a witness for the capability of~$G$; meaning a
group $H$ such that $H/Z(H)\cong G$. The relations among the elements
of~$G$ force in turn relations among the elements of~$H$. When $G$ is
not capable, this will manifest itself as undesired relations among
the elements of~$H$, forcing certain elements whose image should not
be trivial in~$G$ to be central in~$H$. 

When $G$ is a group of class two, this can be achieved by starting
from the relatively free group of class three in an adequate number of
generators. However, any further reductions that can be done in the
starting potential witness group~$H$ will yield dividends of
simplicity later on; this is the main goal of the following result; 
the argument for condition (ii) appears \textit{en passant} in the proof of
\cite{heinnikolova}*{Theorem~1}. 

\begin{theorem} Let $G$ be a group, generated by
  $g_1,\ldots,g_n$. If $G$ is capable, then there exists a group
  $H$, such that $H/Z(H)\cong G$, and elements $h_1,\ldots,h_n\in H$
  which map onto $g_1,\ldots,g_n$, respectively, under the isomorphism such that:
\begin{itemize}
\item[(i)] $H=\langle h_1,\ldots,h_n\rangle$, and
\item[(ii)] The order of $h_i$ is the same as the order of $g_i$,
  $i=1,\ldots,n$.
\end{itemize}
Moreover, if $G$ is finite, then $H$ can be chosen to be finite as well.
\label{th:specialwitness}
\end{theorem}

\begin{proof} If $G$ is capable, then there exists a group $K$ such
  that $K/Z(K)\cong G$; if $G$ is finite, then by~\cite{isaacs}*{Lemma~2.1}
we may choose $K$ to be finite. 

Pick $k_1,\ldots,k_n\in K$ mapping to
$g_1,\ldots,g_n$, respectively, and let $M$ be the subgroup of $K$
generated by $k_1,\ldots,k_n$. Since $MZ(K)=K$, it follows that
$Z(M)=M\cap Z(K)$, hence $M/Z(M)\cong K/Z(K)\cong G$. Thus, replacing
$K$ by $M$ if necessary, we may assume that $K$ is generated by
$k_1,\ldots,k_n$, mapping onto $g_1,\ldots,g_n$, respectively.

Fix $i_0\in\{1,\ldots,n\}$; we show that we can replace $K$ with a
group $H$ with generators $h_1,\ldots,h_n$, such that $H/Z(H)\cong
G$, where $h_i$ maps to $g_i$ for each $i$, the order of $h_{i_0}$
is the same as the order of $g_{i_0}$, and for all $i\neq i_0$, the
order of $h_{i}$ is the same as the order of $k_i$.  Repeating the
construction for $i_0=1,\ldots,n$ will yield the desired group~$H$.

Let $C=\langle x\rangle$ be a cyclic group, with $x$ of the same order
as $k_{i_0}$, and consider $K\times C$. Let $m$ be the order of
$g_{i_0}$ (set $m=0$ if $g_{i_0}$ is not torsion), and consider the
group $M= (K\times C)/\langle (k_{i_0}^m,x^{-m})\rangle$. Since the
intersection of the subgroup generated by
$(k_{i_0}^m,x^{-m})$ with the commutator subgroup of $K\times C$ is
trivial, it follows that if $(k,x^a)$ maps to the center of $M$, then
$[(k,x^a),K\times C]$ must be trivial, so $k\in Z(K)$. 
That is, $Z(M)$ is the image of $Z(K)\times C$.
Therefore,
$M/Z(M) \cong (K\times C)/(Z(K)\times C) \cong K/Z(K)\cong G$. Note
that the isomorphism identifies the image of $(k_{j},x^a)$ with $g_j$
for all $j$ and all integers~$a$.

For $i\neq i_0$, let $h_i$ be the image of $(k_i,e)$ in $M$; and let
$h_{i_0}$ be the image of $(k_{i_0},x^{-1})$ in $M$. Finally, let $H$
be the subgroup of $M$ generated by $h_1,\ldots,h_n$. Then $HZ(M)=M$,
so once again we have $H/Z(H)\cong M/Z(M)\cong G$, and the map $H\to
H/Z(H)\cong G$ sends $h_i$ to $g_i$. Moreover, the order of $h_{i_0}$
is equal to the order of $g_{i_0}$. This finishes the construction.
\end{proof}

This result now allows us to give a very specific ``canonical witness''
to the capability of~$G$.

\begin{theorem} Let $G$ be a finite noncyclic group of class at most
  two and exponent an odd prime~$p$. Let $g_1,\ldots,g_n$ be elements
  of $G$ that project onto a basis for $G^{\rm ab}$, and let $F$ be
  the $3$-nilpotent product of $n$ cyclic groups of order~$p$
  generated by $x_1,\ldots,x_n$, respectively. Let $N$ be the kernel
  of the morphism $\psi\colon F\to G$ induced by mapping $x_i\mapsto g_i$,
  $i=1,\ldots,n$. Then $G$ is capable if and only if
\[ G \cong \left(F/[N,F]\right)\bigm/ Z\left(F/[N,F]\right).\]
\label{th:canonicalwitness}
\end{theorem}

\begin{proof}  Sufficiency is immediate. For the necessity, assume
  that $G$ is capable, and let $H$ be the group guaranteed by
  Theorem~\ref{th:specialwitness} such that $G\cong H/Z(H)$. Note that
  $H$ is of class at most three. Let $\theta\colon H/Z(H)\to G$ be an
  isomorphism that maps $h_iZ(H)$ to $g_i$.

Since $h_1,\ldots,h_n$ are of order~$p$, there exists a (unique
surjective) morphism $\varphi\colon F\to H$ induced by mapping $x_i$
to~$h_i$, $i=1,\ldots,n$. If $\pi\colon H\to H/Z(H)$ is the
canonical projection, then we must have $\theta\pi\varphi=\psi$ by the
universal property of the coproduct. Thus,
$\varphi(N)=\ker(\pi)=Z(H)$, so~$[N,F]\subset\ker(\varphi)$,
and~$\varphi$ factors through $F/[N,F]$; 
surjectivity of $\varphi$ implies that $\varphi(Z(F/[N,F]))\subset
Z(H)$, hence $G\cong H/Z(H)$ is a quotient of
$(F/[N,F])\bigm/Z(F/[N,F])$.

On the other hand, $N[N,F]\subseteq Z(F/[N,F])$, so $G\cong
F/N=F/N[N,F]$ has $(F/[N,F])\bigm/Z(F/[N,F])$ as a quotient.

Thus we have that $G$ has $(F/[N,F])\bigm/ Z(F/[N,F])$ as a quotient,
which in turn has $G$ as a quotient.  Since $G$ is finite, the only
possibility is that the central quotient of $F/[N,F]$ is isomorphic
to~$G$, as claimed.
\end{proof}

\begin{corollary} Let $G$ be a finite noncyclic group of class at most
  two and exponent an odd prime~$p$. Let $g_1,\ldots,g_n$ be elements
  of~$G$ that project onto a basis for $G^{\rm ab}$, and let $F$ be
  the $3$-nilpotent product of $n$ cyclic groups of order~$p$
  generated by $x_1,\ldots,x_n$, respectively. Let $\psi\colon F\to G$
  be the map induced by sending $x_i$ to $g_i$, $i=1,\ldots,n$.
Finally, let $C$ be the subgroup
  of~$F$ generated by the commutators $[x_j,x_i]$, $1\leq i<j\leq n$.
If $X$ is the subgroup of~$C$ such that ${\rm ker}(\psi)=X\oplus F_3$,
  then $G$ is capable if and only if
$\bigl\{c \in C\,|\, [c,F]\subset [X,F]\bigr\} = X$.
\label{cor:simplecondition}
\end{corollary}

\begin{proof} Let $N={\rm ker}(\psi)$. By
  Theorem~\ref{th:canonicalwitness}, $G$ is capable if and only if $G$
 is isomorphic to the central quotient of $F/[N,F]$. Thus, $G$ is
 capable if and only if the center of $F/[N,F]$ is~$N/[N,F]$, and no
 larger . 

An element $h[N,F]\in F/[N,F]$ lies in $Z(F/[N,F])$ if and only if
$[h,F]\subseteq [N,F]$. Since $G$ is of exponent~$p$, $F_3\subseteq
N\subseteq F_2$ and so $[N,F]=[X,F]\subseteq F_3$. In particular, we deduce
that if $h[N,F]$ is central, then $h$ must lie in $F_2$. Write $h=cf$,
with $c\in C$ and $f\in F_3$. Then $[h,F]=[c,F]$, so $h[N,F]$ is
central if and only if $[c,F]\subset [X,F]$. 

If $\bigl\{c \in C\,\bigm|\, [c,F]\subset [X,F]\bigr\}=X$, then it
follows that $h[N,F]$ is central if and only if $h=cf$ with $c\in X$
and $f\in F_3$, which means that $h[N,F]$ is central if and only if
$h\in N$. Hence, the center of $F/[N,F]$ is $N/[N,F]$, and $G$ is
capable.

Conversely, assume that $G$ is capable. Then the center of $F/[N,F]$
is equal to $N/[N,F]$. Therefore, $X\subseteq\bigl\{c\in C\,\bigm|\,
[c,f]\subset [X,F]\bigr\}\subseteq N\cap C = X$,
giving equality and
establishing the corollary.
\end{proof}

One advantage of the description just given is the following: both
$F_2$ and $F_3$ are vector spaces over $\mathbb{F}_p$, and the maps
$[-,f]\colon F_2\to F_3$ are linear transformations for each $f\in F$;
hence, the condition just described can be restated in terms of vector
spaces, subspaces, and linear transformations.
While all the work can still be done at the level of groups and
commutators, the author, at any rate, found it easier to think in
terms of linear algebra. In addition, once the problem has been cast
into linear algebra terms, there is a host of tools (such as
geometric arguments) that can be brought to bear on the issue.

We will discuss this translation and more results on capability below,
after a brief abstract interlude on linear algebra.

\section{Some linear algebra.}\label{sec:linalg}

We set aside groups and capability temporarily to describe the general
construction that we will use in our analysis.

\begin{definition} Let $V$ and $W$ be vector spaces over the same field, and let
$\{\ell_i\}_{i\in I}$ be a nonempty family of linear transformations
  from $V$ to~$W$.  Given a subspace $X$ of~$V$, let $X^*$ be the
  subspace of $W$ defined by:
\[ X^* = {\rm span}\bigl(\ell_i(X)\,|\, i\in I\bigr).\]
Given a subspace $Y$ of $W$, let $Y^*$ be the subspace of $V$ defined by:
\[ Y^* = \bigcap_{i\in I}\ell_i^{-1}(Y).\]
It will be clear from context whether we are talking about subspaces
of $V$ or~$W$.
\end{definition}
It is clear that $X\subset X'\Rightarrow X^*\subset X'^*$ for all
subspaces $X$ and~$X'$ of~$V$, and likewise
$Y\subset Y' \Rightarrow Y^*\subset Y'^*$ for all subspaces $Y,Y'$ of~$W$. 

\begin{theorem}
Let $V$ and $W$ be vector spaces over the same field and let
  $\{\ell_i\}_{i\in I}$ be a nonempty family of linear transformations
  from $V$ to~$W$. The operator on
  subspaces of $V$ defined by $X\mapsto X^{**}$ is a closure operator;
  that is, it is increasing, isotone, and idempotent. Moreover,
  $(X^{**})^* = (X^*)^{**}=X^*$ for all subspaces $X$ of~$V$.
\label{th:closureop}
\end{theorem}

\begin{proof}
Since $\ell_i(X)\subseteq X^*$ for all $i$, it follows that $X\subset
X^{**}$, so the operator is increasing. If $X\subset X'$, then
$X^*\subset X'^*$, hence $X^{**}\subset X'^{**}$, and the operator is
isotone. The equality of $(X^{**})^*$ and $(X^*)^{**}$ is immediate. 
Since $X\subset X^{**}$, we have
$X^*\subset (X^{**})^*$. And by construction $\ell_i(X^{**})\subset
X^*$ for each $i$, so $(X^{**})^*\subset X^*$ giving equality. 

Thus, $(X^{**})^{**} = (X^{***})^* = (X^{*})^{*} = X^{**}$, so the
operator is idempotent, finishing the proof. 
\end{proof}

It may be worth noting that while this closure operator is algebraic
(the closure of a subspace $X$ is the union of the closures of all
finitely generated subspaces $X'$ contained in~$X$), it is not
topological (in general, the closure of the subspace generated by $X$
and $X'$ is not equal to the subspace generated by $X^{**}$ and
$X'^{**}$). 

The dual result holds for subspaces of~$W$:

\begin{theorem}
Let $V$ and $W$ be vector spaces over the same field, and let $\{\ell_i\}_{i\in
  I}$ be a nonempty family of linear transformations from $V$ to~$W$. The operator on
  subspaces of $W$ defined by $Y\mapsto Y^{**}$ is an interior
  operator; that is, it is decreasing, isotone, and
  idempotent. Moreover, 
  $(Y^{**})^*=(Y^*)^{**}=Y^*$ for all subspaces $Y$ of~$W$. 
\label{th:interiorop}
\end{theorem} 

\begin{proof}
That the operator is isotone follows as it did in the previous
  theorem. Since $\ell_i(Y^{*})\subset Y$ for each $i$, it follows
  that $Y^{**}\subset Y$, showing the operator is decreasing. Set
  $Z=Y^{**}$; by construction, $Y^*\subset \ell_i^{-1}(Z)$ for each
  $i$, so $Y^*\subset Z^*$. Therefore, $Z=Y^{**}\subset Z^{**}\subset Z$. Thus
  $Z=Z^{**}$, proving the operator is idempotent.

Again, the equality of $(Y^{**})^*$ and $(Y^*)^{**}$ is immediate. To
finish we only need to show that $Y^*$ is a closed subspace of
$V$. From Theorem~\ref{th:closureop} we know that $Y^*\subset
(Y^*)^{**}$; since $Y^{**}\subset Y$, it follows that
$(Y^*)^{**}=(Y^{**})^*\subset Y^{*}$, giving equality.
\end{proof}

As above, the interior operator is algebraic but in general not
topological.  However, we do have the following result:

\begin{lemma}
Let $V$ and~$W$ be vector spaces over the same field, and let
$\{\ell_i\}_{i\in I}$ be a nonempty family of linear transformations from~$V$ to~$W$. If
$A$ and~$B$ are subspaces of~$V$, then $(A+B)^*=A^*+B^*$. 
\label{lemma:sumsforstar}
\end{lemma}

\begin{proof} Since $A$ and $B$ are contained in $A+B$, we have
  $A^*,B^*\subseteq (A+B)^*$, and therefore $A^*+B^*\subseteq
  (A+B)^*$. Conversely, if $\mathbf{w}\in(A+B)^*$, then we can express
  $\mathbf{w}$ as a linear combination
$\mathbf{w} = \ell_{i_1}(a_1+b_1) + \cdots +
  \ell_{i_k}(a_k+b_k)$, with $a_i\in A$, $b_i\in B$. This gives
$\mathbf{w} = \Bigl( \ell_{i_1}(a_1) + \cdots +
  \ell_{i_k}(a_k)\Bigr) + \Bigl( \ell_{i_1}(b_1) + \cdots +
  \ell_{i_k}(b_k)\Bigr)\in A^*+B^*$,
proving the equality.
\end{proof}
The lemma implies that $(A\oplus B)^* = A^*+B^*$; however, in
general we cannot replace the sum on the right hand side with a direct
sum. 

Given a family of linear transformations
$\{\ell_i\colon V\to W\}_{i\in I}$, we will say a subspace $X$ of $V$
is \textit{$\{\ell_i\}_{i\in I}$-closed} (or simply \textit{closed} if
the family is understood from context) if and only if
$X=X^{**}$. Likewise, we will say a subspace $Y$ of $W$ is
\textit{$\{\ell_i\}_{i\in I}$-open} (or simply \textit{open}) if and only if $Y=Y^{**}$.

It is easy to verify that the closure and interior operators
determined by a nonempty family $\{\ell_i\}_{i\in I}$ of linear
transformations is the same as the closure operator determined by the
subspace of $\mathcal{L}(V,W)$ (the space of all linear
transformations from $V$ to~$W$) spanned by the $\ell_i$. 
Likewise, the following observation is straightforward:

\begin{prop}
Let $V$ and $W$ be vector spaces, and $X$ be a subspace of~$V$. Let
$\{\ell_i\}_{i\in I}$ be a nonempty family of linear
transformations from $V$ to~$W$, and let $\psi\in {\rm Aut}(V)$. If we use ${}^{**}$
to denote the $\{\ell_i\}_{i\in I}$ closure operator, then the
$\{\ell_i\psi^{-1}\}_{i\in I}$-closure of $\psi(X)$ is $\psi(X^{**})$.
In particular, $X$ is $\{\ell_i\}$-closed if and only if $\psi(X)$ is
$\{\ell_i\psi^{-1}\}$-closed. If $\{\ell_i\}$ and
$\{\ell_i\psi^{-1}\}$ span the same subspace of $\mathcal{L}(V,W)$,
then $X$ is closed if and only if $\psi(X)$ is closed.
\label{prop:symmetry}
\end{prop}

\subsection*{Back to capability}\label{subsec:backtocap}

To tie the construction above back to the problem of capability, we introduce specific vector
spaces and linear transformations based on
Corollary~\ref{cor:simplecondition}. We fix an odd prime $p$ throughout.

\begin{definition} Let $n>1$. We let $U(n)$ denote a vector space over
  $\mathbb{F}_p$ of dimension~$n$. We let $V(n)$ denote the vector
  space $U(n)\wedge U(n)$ of dimension $\binom{n}{2}$. Finally, we let
  $W(n)$ be the quotient $(V(n)\otimes U(n))/J$, where $J$ is the
  subspace spanned by all elements of the form
\[ (\mathbf{a}\wedge\mathbf{b})\otimes\mathbf{c} +
  (\mathbf{b}\wedge\mathbf{c})\otimes\mathbf{a} +
  (\mathbf{c}\wedge\mathbf{a})\otimes \mathbf{b},\] with
  $\mathbf{a},\mathbf{b},\mathbf{c}\in U$.
  The vector space $W(n)$ has dimension $2\binom{n+1}{3}$. 
  If there is no danger
  of ambiguity and $n$ is understood from context, we will simply
  write $U$, $V$, and~$W$ to refer to these vector spaces.
\label{defn:defofUVW}
\end{definition}

The following notation will be used only in the context where there is
a single specified basis for $U$, to avoid any possibility of ambiguity:

\begin{definition} Let $n>1$, and let $U$, $V$, and~$W$ be as
  above. If $u_1,\ldots,u_n$ is a given basis for $U$, and $i$, $j$,
  and~$k$ are integers, $1\leq i,j,k\leq n$, then we let $v_{ji}$
  denote the vector $u_j\wedge u_i$ of $V$, and $w_{jik}$ denote
  vector of~$W$ which is the image of $v_{ji}\otimes u_k$.
  The ``prefered basis'' for $V$ (relative to $u_1,\ldots,u_n$) will
  consist of the vectors $v_{ji}$ with $1\leq i<j\leq n$. The
  ``prefered basis'' for $W$ will consist of the vectors $w_{jik}$
  with $1\leq i<j\leq n$ and $i\leq k\leq n$. 
\end{definition}

To specify our closure and interior operators on $V$ and~$W$, we
define the following family of linear transformations:

\begin{definition} Let $n>1$. We embed $U$ into $\mathcal{L}(V,W)$
  as follows:
  given $\mathbf{u}\in U$ and $\mathbf{v}\in V$, we let
  $\varphi_{\mathbf{u}}(\mathbf{v})=\overline{\mathbf{v}\otimes
  \mathbf{u}}$, where $\overline{\mathbf{x}}$ denotes the image in $W$
  of a vector $\mathbf{x}\in V\otimes U$. If $u_1,\ldots,u_n$
  is a given basis for $U$ and $i$ is an integer, $1\leq i\leq n$,
  then we will use $\varphi_i$ to denote the linear transformation
  $\varphi_{u_i}$. 
\end{definition}

The closure operator we will consider is determined by the
family $\{\varphi_{\mathbf{u}}\,|\, \mathbf{u}\in U\}$. As noted
above, if $u_1,\ldots,u_n$ is a basis for $U$, then this closure
operator is also determined by the family
$\{\varphi_1,\ldots,\varphi_n\}$.

Going back to the problem of capability,
let $F$ be the $3$-nilpotent product of cyclic groups of
order~$p$ generated by $x_1,\ldots,x_n$. We can identify $F_2$ with
$V\oplus W$ by identifying $v_{ji}$ with $[x_j,x_i]$ and $w_{jik}$
with $[x_j,x_i,x_k]$; this also identifies $W$ with~$F_3$. 

Let $G$ be a noncyclic group of class at most two and exponent~$p$,
and let $g_1,\ldots,g_n$ be elements of $G$ that project onto a basis
for $G^{\rm ab}$. If we let $\psi\colon F\to G$ be the map induced by
mapping $x_i\mapsto g_i$ and $N={\rm ker}(\psi)$, then as above we can
write $N=X\oplus F_3$, where $X$ is a subgroup of $C=\langle
[x_j,x_i]\,\bigm|\, 1\leq i<j\leq n\rangle$. Thus, we can identify $X$
with a subspace of $V$ by identifying the latter with the subgroup~$C$; abusing notation
somewhat, we call this subspace $X$ as well.

\begin{theorem}
Let $G$, $F$, $C$, and~$X$ be as in the preceding two paragraphs. Then $G$
is capable if and only if $X$ is $\{\varphi_{\mathbf{u}}\,|\,
\mathbf{u}\in U\}$-closed.
\label{th:capasclosure}
\end{theorem}

\begin{proof} We know that $G$ is capable if and only if 
$\bigl\{c \in C\,|\, [c,F]\subset [X,F]\bigr\} = X$.
Identifying $C$ with~$V$ and $F_3$ with $W$, note that
$\varphi_i$ is a map from $C$ to~$F_3$,
corresponding to $[-,x_i]$. Thus, $X^*\subseteq W$ corresponds to
$[X,F]\subseteq F_3$, and
$X^{**}$ corresponds to the set
$\bigl\{c \in C\,|\, [c,F]\subset [X,F]\bigr\}$.
Therefore, $G$ is capable if and only if 
\[X = \bigl\{ \mathbf{v}\in V\,\bigm|\,
\varphi_{\mathbf{u}}(\mathbf{v})\in X^*\mbox{\ for all $\mathbf{u}\in
  U$}\bigr\} = X^{**},\] as claimed. 
\end{proof}

In other words, the closure operator codifies exactly the condition we
want to check to test the capability of $G$.
Thus the question
\textit{``What $n$-generated $p$-groups of class two and exponent~$p$
  are capable?''} is equivalent to the question \textit{``What
  subspaces of $V(n)$ are $\{\varphi_{\mathbf{u}}\,|\,\mathbf{u}\in
  U\}$-closed?''}  

Of course, different subspaces may yield isomorphic groups. In
particular, if we let ${\rm GL}(n,p)$ act on $U$, then this action
induces an action of ${\rm GL}(n,p)$ on $V=U\wedge U$; if $X$ and $X'$
are on the same orbit relative to this action, then the groups $G$
and~$H$ that correspond to $X$ and~$X'$, respectively, are
isomorphic. By Proposition~\ref{prop:symmetry} the closures of $X$ and
$X'$ will also be in the same orbit under the action and $G$ will
be capable if and only if $H$ is capable.

Also of interest is the description of the closure of $X$ when $G$ is
not capable. It is clear that the quotient of $G$ determined by
$X^{**}$ is the largest quotient of~$G$ that is capable. That is,
$X^{**}/X$ is isomorphic to the \textit{epicenter} of~$G$, the
smallest normal subgroup $N\triangleleft G$ such that $G/N$ is
capable.  In most cases where a subspace $X$ is not closed, therefore,
we will attempt to give an explicit description of $X^{**}$ rather
than simply prove $X$ is not closed.

The following explicit descriptions of the linear transformations
$\varphi_{\mathbf{u}}$, relative to a given basis, will also be useful
and are straightforward:

\begin{lemma}
Fix $n>1$, let $u_1,\ldots,u_n$ be a basis for $U$, and let
$v_{ji}$, $w_{jik}$ be the corresponding bases for $V$ and~$W$.
For all integers $i$, $j$, and~$k$, $1\leq i<j\leq n$, $1\leq k\leq
n$, the image of $v_{ji}$ under $\varphi_k$ in terms of the prefered
basis of $W$ is:
\[
\varphi_k(v_{ji}) = \left\{\begin{array}{ll}
w_{jik} & \mbox{if $k\geq i$,}\\
w_{jki} - w_{ikj} & \mbox{if $k<i$.}
\end{array}\right.\]
\label{lemma:explicitformulas}
\end{lemma}

\section{Basic applications.}\label{sec:basics}

In this section, we obtain some consequences of our set-up so far. We
assume throughout that we have a specified ``preferred basis''
$\{u_i\}$ for $U$, from which we obtain the corresponding basis
$\{v_{ji}\,|\, 1\leq i<j\leq n\}$ for~$V$, and likewise the basis $\{w_{jik}\,|\, 1\leq i<j\leq n,
i\leq k < n\}$ for~$W$. 

The following observations follow immediately from the definitions:

\begin{lemma} Fix $n>1$, and let $k$ be an integer, $1\leq k\leq n$. 
\begin{itemize}
\item[(i)] $\varphi_k$ is one-to-one, and $W = \langle
  \varphi_1(V),\ldots,\varphi_n(V)\rangle$. 
\item[(ii)] The trivial and total subspaces of~$V$ are closed.
\item[(iii)] The trivial and total subspaces of~$W$ are open.
\end{itemize}
\end{lemma}

\begin{definition} Let $i,j,k$ be integers, $1\leq i<j\leq n$, 
$i\leq k\leq n$. We let
  $\pi_{ji}\colon V\to \langle v_{ji}\rangle$ and
  $\pi_{jik}\colon W\to\langle w_{jik}\rangle$ be
  the canonical projections.
\end{definition}

\begin{lemma} Let $\mathbf{w}\in\varphi_k(V)$. If
  $\pi_{rst}(\mathbf{w})\neq\mathbf{0}$, with $1\leq s<r\leq n$,
  $s\leq t\leq n$, then $s\leq k\leq t$, and at most one of the
  inequalities is strict. 
\label{lemma:indexinequalities}
\end{lemma}

\begin{proof} It is enough to prove the result for $\mathbf{w}$ an
  element of a basis of $\varphi_k(V)$.
 Such a basis is given by the vectors
  $w_{jik}$ with $1\leq i<j\leq n$, $i\leq k\leq n$, and the vectors
  $w_{jki}-w_{ikj}$ with $1\leq i<j\leq n$ and $1\leq
  k<i$. Considering these basis vectors, we see that the first class
 has $r=j$, $s=i$, $t=k$, so $s\leq k=t$.
 The second class of vectors will
 yield either $r=j$, $s=k$, $t=i$, with $s=k<t$; or else $r=i$, $s=k$,
 $t=j$, with $s=k<t$. This proves the lemma.
\end{proof}

\begin{lemma} Let $i,j$ be integers, $1\leq i<j\leq n$, and $r$ an
  integer such that $1\leq r\leq n$. For $\mathbf{v}\in V$,
  $\pi_{jij}(\varphi_r(\mathbf{v}))\neq\mathbf{0}$ if and only if
  $\pi_{ji}(\mathbf{v})\neq \mathbf{0}$ and $r=j$. Likewise,
  $\pi_{jii}(\varphi_r(\mathbf{v}))\neq\mathbf{0}$ if and only if
  $\pi_{ji}(\mathbf{v})\neq\mathbf{0}$ and $r=i$.
\label{lemma:doubleindex}
\end{lemma}
\begin{proof} The vectors $w_{jij}$ occurs in the image of a
  $\varphi_r$ exactly when $r=j$ and it is applied a vector with
  nontrivial $\pi_{ji}$ projection.
   Thus, if $\pi_{jij}(\mathbf{v})\neq\mathbf{0}$ then
  $\pi_{ji}(\mathbf{v})\neq\mathbf{0}$. The converse is immediate, and
  the case of $\pi_{jii}$ is settled in the same manner. 
\end{proof}

\begin{lemma} Fix $i,j$, $1\leq i<j\leq n$. If
  $\pi_{ji}(X)=\{\mathbf{0}\}$, then
  $\pi_{ji}(X^{**})=\{\mathbf{0}\}$.
\label{lemma:trivialproj}
\end{lemma}

\begin{proof} Since $\pi_{ji}(X)=\{\mathbf{0}\}$, it follows that
  $\pi_{jii}(X^*)=\{\mathbf{0}\}$ by
  Lemma~\ref{lemma:indexinequalities}. Therefore, if $\mathbf{v}\in V$
  has $\pi_{ji}(\mathbf{v})\neq \mathbf{0}$ then
  $\varphi_{i}(\mathbf{v})\notin X^*$, hence $\mathbf{v}\notin
  X^{**}$. Thus, $\pi_{ji}(X^{**})=\{\mathbf{0}\}$, as claimed.
\end{proof}

These lemmas suffice to establish a result of
Ellis~\cite{ellis}*{Prop.~9}, which appears as
Corollary~\ref{cor:elliscor} below.

\begin{theorem} If $X$ is a coordinate subspace relative to a basis
  for $U$ (that is, there is a basis $u_1,\ldots,u_n$ such that $X$ is
  generated by a subset of $\{v_{ji} \,|\, 1\leq i<j\leq n\}$), then
  $X$ is closed.
\label{thm:coordclosed}
\end{theorem}

\begin{proof} Suppose $S\subseteq\{ v_{ji} \,|\, 1\leq i<j\leq n\}$ is
  such that $X=\langle S\rangle$. By the previous Lemma, we have that
  $X^{**}\subseteq \langle S\rangle$; therefore, 
$\langle S\rangle = X \subseteq X^{**}\subseteq \langle S\rangle =
  X$,
and so $X=X^{**}$. 
\end{proof}

\begin{corollary}[\cite{ellis}*{Prop.~9}] Let $G$ be a group of class two
  and exponent~$p$, and let $x_1,\ldots,x_n$ be elements of $G$ that
  project onto a basis for $G/Z(G)$. If the nontrivial commutators of
  the form $[x_j,x_i]$, $1\leq i<j\leq n$, are distinct and form a
  basis for $[G,G]$, then $G$ is capable.
\label{cor:elliscor}
\end{corollary}

\begin{proof} Such a $G$ corresponds to an $X$ that is a coordinate
  subspace of~$V$, so capability follows from
  Theorem~\ref{thm:coordclosed}. 
\end{proof} 

\subsection*{The big, the small, and the mixed.}\label{subsect:bigandsmall}

The following definition and proposition will be needed below.

\begin{definition}
Let $n$ be an integer greater than~$1$, and $i$ an integer,
$1\leq i\leq n$. We define
$\Pi_i\colon V\to \langle v_{i,1},\ldots,
v_{i,i-1},v_{i+1,i},\ldots,v_{n,i}\rangle$ to be the canonical
projection.
\end{definition}

\begin{prop} Let $n>1$ and $i$ be an integer, $1\leq i\leq n$. Let
  $W_i$ be the subspace of $W$ spanned by the basis vectors $w_{rst}$,
  $1\leq s<r\leq n$, $s\leq t\leq n$, such that exactly one of $r$,
  $s$, and~$t$ is equal to~$i$. If $X$ is a subspace of $V$ such that
  $\Pi_i(X)=\{\mathbf{0}\}$, then $X^*\cap W_i=\varphi_i(X)$ and $X$
  is closed.
\label{prop:oneindexabsent}
\end{prop}

\begin{proof} That $\varphi_i(X)$ is contained in $W_i$ follows
  because $\Pi_i(X)$ is trivial. 
  Since the subspace
  $\langle \varphi_j(X)\,|\,j\neq
  i\rangle$ is contained in the subspace spanned by basis vectors
  $w_{rst}$ in which none of $r,s,t$ are equal to~$i$, we have
$X^* = \varphi_i(X) \oplus \langle \varphi_j(X)\,|\, j\neq
  i\rangle$
and the equality of intersection follows.
  To show $X$ is closed, let
  $\mathbf{v}\in X^{**}$. By Lemma~\ref{lemma:trivialproj}, we know
  that $\Pi_i(\mathbf{v})=\mathbf{0}$, and so $\varphi_i(\mathbf{v})$
  lies in $X^*\cap W_i=\varphi_i(X)$. Since $\varphi_i$ is
  one-to-one, we deduce that $\mathbf{v}\in X$. Thus, $X$ is closed.
\end{proof}

Fix a basis $u_1,\ldots,u_n$ for~$U$. Given $r$, $1\leq r<n$, we
can divide these basis vectors into ``small'' and ``large'', according
to whether their indices are less than or equal to $r$, or strictly larger
than~$r$, respectively. From this, we obtain a similar partition of the
corresponding basis vectors $v_{ji}$, $1\leq i<j\leq n$ of~$V$, and
$w_{jik}$, $1\leq i<j\leq n$, $i\leq k\leq n$ for~$W$. Namely, we
write $V = V_{s} \oplus V_{m} \oplus V_{\ell}$, $W=W_s\oplus W_{ms}
\oplus W_{m\ell} \oplus W_{\ell}$, where:
\begin{eqnarray*}
V_s & = & \Bigl\langle v_{ji}\,\Bigm|\, 1\leq i<j\leq r\Bigr\rangle,\\
V_m & = & \Bigl\langle v_{ji}\,\Bigm|\, 1\leq i\leq r< j\leq n
\Bigr\rangle,\\
V_{\ell} & = & \Bigl \langle v_{ji}\,\Bigm|\, r<i<j\leq
n\Bigr\rangle,\\
W_s & = & \Bigl \langle w_{jik}\,\Bigm|\, 1\leq i<j\leq r, i\leq
k\leq r\Bigr\rangle,\\
W_{ms} & = & \Bigl\langle w_{jik}\,\Bigm|\, 1\leq i<j\leq n, i\leq
k\leq n, \mbox{\small\ either $j\leq r$ or $k\leq r$, but not both}\Bigr\rangle,\\
W_{m\ell} & = & \Bigl\langle w_{jik}\,\Bigm|\, 
1\leq i\leq r < j,k \leq n\Bigr\rangle,\\
W_{\ell} & = & \Bigl\langle w_{jik}\,\Bigm|\, r<i<j\leq n, i\leq
k\leq n\Bigr\rangle.
\end{eqnarray*}
We refer informally to $V_s$ as the ``small part'' of $V$, and its
elements as ``small vectors;'' $V_{\ell}$ is the
``large part'' and contains the ``large vectors;'' and $V_{m}$ will
be called the ``mixed part'' while its elements will be refered to as
``mixed vectors.'' A similar informal convention will be followed
with~$W$, calling $W_s$ the ``small part,'' $W_{\ell}$ the ``large
part,'' $W_{ms}$ the ``mixed-small part,'' and $W_{m\ell}$ the
``mixed-large part'' of~$W$.

\begin{lemma} Notation as in the previous paragraph. If $n>1$ and $r$
  is an integer, $1\leq r < n$, then:
\begin{itemize}
\item[(i)] $V_s^* \subseteq W_{s}\oplus W_{ms}$.
\item[(ii)] $V_{\ell}^* \subseteq W_{m\ell}\oplus W_{\ell}$.
\item[(iii)] $V_m^* = W_{ms}\oplus W_{m\ell}$.
\end{itemize}
\label{lemma:imsofmixed}
\end{lemma}
\begin{proof}
Note that the indices involved in the image of $\varphi_k(v_{ji})$ are
$i$, $j$, and~$k$. Thus, if both $i$ and~$j$ are small (resp.~large),
then all images are either small or mixed-small (resp.~mixed large or
large); and if $i$ is small and $j$ is large, then all images are
mixed. This proves (i) and~(ii), and also proves that $V_m^*$ is
contained in $W_{ms}\oplus W_{m\ell}$. To finish the proof of (iii),
suppose that $w_{jik}$ is one of the generators of $W_{ms}\oplus
W_{m\ell}$, as described above. Note that we must have $i\leq r$ in
either case. Then $w_{jik}=\varphi_k(v_{ji})$. If
$j>r$, then $v_{ji}\in V_{m}$, so $w_{jik}\in V_m^*$. If, on the other
hand, $j\leq r$, then we must have $k>r$ since $w_{jik}$ is either
mixed-small or mixed-large. Then we know that $w_{kij}\in V_m^*$ by
the immediately preceding argument. Also, $v_{kj}\in V_m$, hence
$\varphi_i(v_{kj})=w_{kij}-w_{jik}\in V_m^*$. Since $w_{kij}\in
V_m^*$, we deduce that $w_{jik}\in V_m^*$ as well, and this finishes
the proof of~(iii). 
\end{proof}

In the following theorem, ${\rm cl}_s(X_s)$ is meant to stand for the
``small closure of~$X_s$''; that is, the
$\{\varphi_{i}\}_{i=1}^r$-closure of~$X_s$; likewise, ${\rm
    cl}_{\ell}(X_{\ell})$ is the ``large closure of~$X_{\ell}$.''

\begin{theorem} Let $n>1$, and let $r$ be an integer, $1\leq r< n$,
  as above. Suppose that $X_s$ is a subspace of $V_s$, and $X_{\ell}$
  is a subspace of $V_{\ell}$. Then:
\begin{itemize}
\item[(i)] $(X_s\oplus X_{\ell})^* = X_s^* \oplus X_{\ell}^*$.
\item[(ii)] $(X_s\oplus V_m \oplus X_{\ell})^* = \langle
  \varphi_i(X_s)\,|\, 1\leq i\leq r\rangle \oplus W_{ms}\oplus
  W_{m\ell}\oplus \langle \varphi_i(X_{\ell})\,|\, r<i\leq n\rangle$. 
\item[(iii)] $X_s\oplus X_{\ell}$ is closed.
\item[(iv)] If ${\rm cl}_s(X_s)$ is the
  $\{\varphi_i\}_{i=1}^r$-closure of $X_s$ and ${\rm
  cl}_{\ell}(X_{\ell})$ is the $\{\varphi_i\}_{i=r+1}^n$-closure of
  $X_{\ell}$, then
$(X_s\oplus V_m \oplus X_{\ell})^{**} = {\rm cl}_s(X_s) \oplus V_m
  \oplus {\rm cl}_{\ell}(X_{\ell})$.
In particular, the subspace $X_s\oplus V_m\oplus X_{\ell}$ is closed if and only if
  $X_s$ is $\{\varphi_i\}_{i=1}^r$-closed and $X_{\ell}$ is
  $\{\varphi_i\}_{i=r+1}^n$-closed.
\end{itemize}
\label{th:bigandsmall}
\end{theorem}

\begin{proof} Part (i) follows from Lemma~\ref{lemma:sumsforstar} and
  from Lemma~\ref{lemma:imsofmixed}(i) and~(ii). 

To prove (ii), note that by Lemmas~\ref{lemma:sumsforstar}
  and~\ref{lemma:imsofmixed}, we have:
\begin{eqnarray*}
(X_s\oplus V_m\oplus X_{\ell})^* &=& X_s^* + V_m^* + X_{\ell}^*\\
& = & \langle \varphi_i(X_s)\,|\, 1\leq i\leq n\rangle + W_{ms} +
  W_{m\ell} + \langle \varphi_i(X_{\ell})\,|\, 1\leq i\leq n\rangle\\
& = & \langle \varphi_i(X_s)\,|\, 1\leq i\leq r\rangle 
  + W_{ms} + W_{m\ell} + \langle
  \varphi_i(X_{\ell})\,|\, 1 \leq i \leq r\rangle.
\end{eqnarray*}
Now simply observe that the first summand is contained in $W_s$ and
the last in $W_{\ell}$ to deduce that the sum is direct.

Moving on to~(iii), by Lemma~\ref{lemma:trivialproj}, we know that
  $(X_s\oplus X_{\ell})^{**}\subseteq V_s\oplus V_{\ell}$. Let
  $\mathbf{v}_s+\mathbf{v}_{\ell}$ be an element of $(X_s\oplus
  X_{\ell})^{**}$, with $\mathbf{v}_s$ a small vector, and
  $\mathbf{v}_{\ell}$ a large vector. Then
  for each $i$, $\varphi_i(\mathbf{v}_s + \mathbf{v}_{\ell})\in
  X_s^*\oplus X_{\ell}^*$. Thus, we must have
  $\varphi_i(\mathbf{v}_s)\in X_s^*$ and
  $\varphi_i(\mathbf{v}_{\ell})\in X_{\ell}^*$ for each $i$, so
  $\mathbf{v}_s\in X_s^{**}$ and $\mathbf{v}_{\ell}\in
  X_{\ell}^{**}$. Thus, $(X_s\oplus X_{\ell})^{**} \subseteq
  X_{s}^{**}\oplus X_{\ell}^{**}$, and the reverse inclusion follows
  because the closure operator is isotonic. It is then enough to show
  that each of $X_s$ and $X_{\ell}$ are closed, and since
  $\Pi_1(X_{\ell})=\Pi_n(X_{s})=\{\mathbf{0}\}$, this follows from
  Proposition~\ref{prop:oneindexabsent}. 

Finally, for (iv), note that if $j>r$, then $\varphi_j(V_s)\subseteq
W_{sm}\subseteq V_m^*$, so ${\rm cl}_s(X_s)$ is contained in the
closure; similarly, ${\rm cl}_{\ell}(X_{\ell})$ is contained in the
closure, so we always have ${\rm cl}_s(X_s)\oplus V_m\oplus {\rm
  cl}_{\ell}(X_{\ell})\subseteq (X_s\oplus V_m\oplus X_{\ell})^{**}$. 

Let $\mathbf{v}=\mathbf{v}_s+ \mathbf{v}_m+\mathbf{v}_{\ell}\in
  (X_s\oplus V_m\oplus X_{\ell})^{**}$, with $\mathbf{v}_s\in V_s$,
  $\mathbf{v}_{\ell}\in V_{\ell}$, and $\mathbf{v}_m\in V_{m}$. Since
  $V_m$ is contained in the closure, $\mathbf{v}$ is in the closure if
  and only if $\mathbf{v}_s+\mathbf{v}_{\ell}$ is in the closure. We
  further claim that $\mathbf{v_s}+\mathbf{v}_{\ell}$ is in the closure if and only if
  each of $\mathbf{v}_s$ and~$\mathbf{v}_{\ell}$ are in the
  closure. One implication is immediate. For the converse, suppose
  that $\mathbf{v}_s+\mathbf{v}_{\ell}$ is in the closure, and $i\leq
  r$. Then by (ii) we have:
\[\varphi_i(\mathbf{v}_s) + \varphi_i(\mathbf{v}_{\ell})\in \langle \varphi_j(X_s)\,|\, j\leq
  r\rangle \oplus W_{ms} \oplus W_{m\ell} \oplus
  \langle\varphi_j(X_{\ell})\,|\, r<j\leq n\rangle.\]
In particular, $\varphi_i(\mathbf{v}_s)\in\langle
  \varphi_j(X_s)\,|\,1\leq j\leq r\rangle$. Since $V_s$ is contained
  in $\varphi_{j}^{-1}(W_{ms})$ for all $j>r$, we conclude that
  $\mathbf{v}_s$ lies in the closure of $X_s\oplus V_{m}\oplus
  X_{\ell}$, and hence so does $\mathbf{v}_{\ell}$. This proves the
  claim.

Finally, observe as above that $\mathbf{v}_s$ lies in the closure if
and only if $\varphi_i(\mathbf{v}_s)$ lies in $\langle\varphi_j(X_s)\,|\,
1\leq j\leq r\rangle$ for $i=1,\ldots,r$, if and only if $\mathbf{v}_s$
lies in ${\rm cl}_s(X_s)$; and similarly that
$\mathbf{v}_{\ell}$ lies in the closure if and only if
it lies in ${\rm cl}_{\ell}(X_{\ell})$. 
Thus, the closure of $X_s\oplus V_m\oplus X_{\ell}$ is equal to
${\rm cl}_s(X_s)\oplus V_m\oplus {\rm cl}_{\ell}(X_{\ell})$. This
proves the theorem.
\end{proof}

The theorem gives the following two interesting corollaries:

\begin{corollary} Let $G_1$ and $G_2$ be any two nontrivial groups of class at
  most two and exponent an odd prime~$p$. Then $G=G_1\amalg^{{\germ
  N}_2} G_2$ is capable.
\label{cor:coprodalways}
\end{corollary}

\begin{proof} If $G_1$ is minimally $r$-generated, and $G_2$ is
  minimally $s$-generated, then $G$ is
  minimally $n=r+s$ generated. If we number the generators of $G_1$ as
  $g_1,\ldots,g_r$, and those of $G_2$ as $g_{r+1},\ldots,g_n$, then
  the subspace of $V$ corresponding to $G$ will be of the form
  $X_s\oplus X_{\ell}$, where $X_s\subseteq V_s$, $X_{\ell}\subseteq
  V_{\ell}$; namely, $X_s$ corresponds to $G_1$, and $X_{\ell}$
  corresponds to $G_2$. By Theorem~\ref{th:bigandsmall}(iii), this
  subspace is always closed. 
\end{proof}

\begin{corollary} Let $G_1$ and $G_2$ be two finite $p$-groups of
  class at most two and exponent~$p$. Then $G_1\oplus G_2$ is capable
  if and only if each $G_i$ is either nontrivial cyclic or capable.
\label{cor:directsum}
\end{corollary}

\begin{proof} Proceeding as above, note that the subspace of $V$
  corresponding to $G_1\oplus G_2$ is equal to $X_s\oplus V_m\oplus
  X_{\ell}$, so by Theorem~\ref{th:bigandsmall}(iv), this subspace is
  closed if and only if $X_s$ is $\{\varphi_i\}_{i=1}^r$ closed and
  $X_{\ell}$ is $\{\varphi_i\}_{i={r+1}}^n$-closed. For noncyclic
  $G_i$ this is equivalent to being capable, while for cyclic $G_i$
  the closure conditions are trivially met.
\end{proof}

In turn, this yields the following important consequences:

\begin{theorem} Let $G$ be a $p$-group of class at most two and
  exponent $p$. Then $G\oplus C_p$ is capable if and only if $G$ is
  cyclic of order~$p$ or capable.
\label{th:cancelcentralsummand}
\end{theorem}

\begin{corollary} Let $G$ be a $p$-group of class exactly two and
  exponent~$p$. If we write $G=K\oplus C_p^r$, where $r\geq 0$ is an integer
  and $K$ is a group of class two satisfying $Z(K)=[K,K]$, then $G$ is
  capable if and only if $K$ is capable.
\label{cor:generalcancelcentralsummand}
\end{corollary}

Note that any group of class exactly two and exponent~$p$ can be
written in the form specified by this corollary. 

\subsection*{Amalgamated direct products and amalgamated coproducts.}

We saw in Corollary~\ref{cor:coprodalways} that if we take two
nontrivial groups of class two and exponent~$p$, then their coproduct
(in this variety) will always be capable, while the capability of a
direct sum depends on the factors. 

We will now deal with two similar constructions, the direct product
with amalgamation and the coproduct with amalgamation. The first
construction includes central products (see for
example~\cite{leedgreen}*{Section~2.2}) but is more general.

\begin{definition} Let $G$ and $K$ be two groups, and let $H$ be a
  subgroup of $Z(G)$. Let $\phi\colon H\to Z(K)$ be an
  embedding. The \textit{amalgamated direct product of $G$ and $K$
  (along $\phi$)} is the group $G\times_{\phi} K$ given by
\[ G\times_{\phi} K = \frac{G\times K}{\{ (h,\phi(h)^{-1})\,|\,
  h\in H\}}.\]
\end{definition}

The maps sending $g\mapsto \overline{(g,e)}$ and
$k\mapsto\overline{(e,k)}$ embed copies of $G$
and of~$K$ into $G\times_{\phi} K$, respectively, and the
intersection of these images is exactly
$H$ (identified with $\phi(H)$).
When $H=Z(G)$ and $\phi$ is an
isomorphism, the construction is called the \textit{central product}
of~$G$ and~$K$
in~\cite{leedgreen}, where it is denoted by $G\circ K$. All extra-special
$p$-groups other than those of order $p^3$ may be constructed as
central products of smaller extra-special groups.

The following result was inspired by doing an automated brute force
search for non-closed subspaces $X$ of dimensions seven and~eight when
$n=5$. It was performed with the computer algebra system GAP~\cite{GAP}. 
I was able to find many examples, and by examining them was
led to the result below.
The statement of the linear algebra theorem is
somewhat complicated, but it leads to a straightforward
group-theoretic corollary: if $G$ and $K$ are groups of class two and
exponent~$p$, $H$ is a nontrivial subgroup of $[G,G]$, and $\phi$
embeds $H$ into $[K,K]$, then the amalgamated direct product
$G\times_{\phi}K$ is not capable.

\begin{theorem} Let $n>3$, and let $r$ be an integer, $2\leq r \leq 
  n-2$. Let $X_s$ and~$X_{\ell}$ be
  subspaces of $V_s$ and~$V_{\ell}$, respectively, and let $H$ be a
  nontrivial subspace of $V_s$ such that $H\cap X_s = \{\mathbf{0}\}$. Let
  $\phi\colon H\to V_{\ell}$ be an embedding such that
  $\phi(H)\cap X_{\ell}=\{\mathbf{0}\}$. Finally, let $X$ be the
  subspace
$X = X_s \oplus X_{\ell} \oplus V_m \oplus \{
  h-\phi(h)\,|\, h\in H\}$. 
Then the closure of $X^{**}$ is
the direct sum of the $\{\varphi_i\}_{i=1}^r$-closure of $X_s\oplus H$,
the $\{\varphi_i\}_{i=r+1}^n$-closure of $X_{\ell}\oplus \phi(H)$, and
$V_m$. In particular, $X$ is not closed.
\label{th:centralamalgnotclosed}
\end{theorem}

\begin{proof} Note that by Lemma~\ref{lemma:imsofmixed}(iii), we have
  $W_{ms}\oplus W_{m\ell}=V_m^*\subseteq X^*$. Next, note that
  $X\cap H=X\cap \varphi(H) = \{\mathbf{0}\}$.

We claim that $H^*\subseteq X^*$, and therefore that
$H\subseteq H^{**}\subset X^{**}$. 
Indeed, let $h\in H$, and let $k$ be an integer, $1\leq k\leq n$. If
$k\leq r$, then $\varphi_k(\phi(h))\in W_{m\ell}$ (since $\phi(h)\in
V_{\ell}$), so 
$\varphi_k(h) = \varphi_k(h - \phi(h)) + \varphi_k(\phi(h))\in
X^*+W_{m\ell}=X^*$.
And if $r<k\leq n$, then $\varphi_k(h)\in W_{ms}\subseteq X^*$. Thus,
$\varphi_k(h)\in X^*$ for $k=1,\ldots,n$, hence $h\in X^{**}$. This
proves that $H^*\subseteq X^*$, hence $H\subseteq H^{**}\subseteq
X^{**}$. 

Thus, the closure of $X$ contains $X_s\oplus H \oplus X_{\ell} \oplus
\phi(H)\oplus V_m$. The description of the closure of~$X$ now follows
as in the proof of Theorem~\ref{th:bigandsmall}(iv). We conclude that
$X$ is not closed, because $H$ is nontrivial, $H\cap X=\{\mathbf{0}\}$,
yet $H\subseteq X^{**}$. 
\end{proof}

\begin{corollary}
Let $G_1$ and $G_2$ be two nonabelian groups of class two and
exponent~$p$, let $H$ be a subgroup of
$[G_1,G_1]$, and let $\phi\colon H\to [G_2,G_2]$ be an
embedding. If $G$ is the amalgamated direct product
$G=G_1\times_{\phi}G_2$, then $G$ is capable if and only if
$H=\{e\}$ and both $G_1$ and~$G_2$ are capable.
\label{cor:centralamalggroupnotclosed}
\end{corollary}

\begin{proof} Let $r$ be the rank of~$G_1^{\rm ab}$, $s$ the rank of
  $G_2^{\rm ab}$, and $n=r+s$. Since $G_1$ and $G_2$ are nonabelian,
  we must have $2\leq r \leq n-2$. The subspace $X$ corresponding to
  $G_1\times G_2$ is of the form $X_s\oplus V_m \oplus X_{\ell}$, with
  $X_s$ and $X_{\ell}$ determined by $G_1$ and~$G_2$,
  respectively. Abusing notation, the subgroup $H$ can be made to
  correspond to a subspace $H$ of~$V_s$ with $H\cap X_s=\mathbf{0}$,
  and $\phi$ induces a linear transformation $\phi\colon H\to
  V_{\ell}$ which can also be chosen to have $\phi(H)\cap
  X_{\ell}=\{\mathbf{0}\}$. The subspace of $V$ corresponding to
  $G_1\times_{\mathbf{\phi}}G_2$ is then equal to $X = X_s \oplus
  X_{\ell} \oplus V_m \oplus \{ h-\varphi(h)\,|\, h\in H\}$.  If
  $H=\{\mathbf{0}\}$, then we are in the situation of
  Corollary~\ref{cor:directsum}. And if $H\neq\{\mathbf{0}\}$, then
  $X$ is not closed by Theorem~\ref{th:centralamalgnotclosed}. This
  proves the result.
\end{proof}

The following is of course well-known, and can be proven using other methods:

\begin{corollary} Let $G$ be an extra-special $p$-group. Then $G$ is
  capable if and only if it is of order $p^3$ and exponent~$p$.
\label{cor:extraspecialcase}
\end{corollary}

\begin{proof} If $G$ is not of exponent~$p$, then it is generated by
  elements of order~$p$ and one element of order $p^2$ (see for example
  \cite{leedgreen}*{Theorem 2.2.10}) and therefore
  is not capable by~\cite{capable}*{Theorem~3.12}. So we may assume $G$ is of
  exponent~$p$.  If $G$ is of order $p^{2n+1}$ with $n>1$, then it is
  isomorphic to a direct product with amalgamation of the
  extra-special $p$-group of order $p^3$ and exponent~$p$, and the
  extra-special $p$-group of order $p^{2n-1}$ and exponent~$p$,
  identifying their commutator subgroups; as such, it is not capable
  by Corollary~\ref{cor:centralamalggroupnotclosed} above. The
  extra-special group of order $p^3$ and exponent~$p$ is closed the
  coproduct of two cyclic groups of order~$p$, and thus is capable by
  Corollary~\ref{cor:coprodalways}. 
\end{proof}

We move now to the case of the coproduct with amalgamation.

\begin{definition}
Let $G$ and $K$ be two groups of class at most two and
exponent~$p$. Let $H$ be a subgroup of~$[G,G]$, and let $\phi\colon
H\to[K,K]$ be an embedding. The \textit{amalgamated coproduct of $G$
  and $K$ (along $\phi$)} is the group
$G\amalg_{\phi}^{\mathfrak{N}_2} K$ given by:
\[ G \amalg_{\phi}^{\mathfrak{N}_2} K = \frac{G\amalg^{\mathfrak{N}_2}
  K}{\{h\phi(h)^{-1}\,|\, h\in H\}}.\]
\end{definition}

Note that the elements $h$ and~$\phi(h)^{-1}$ are central, so the
subset given above is in fact a normal subgroup. Again, it is easy to

In
general, if $G,K\in\mathfrak{N}_2$, $H$ is an arbitrary subgroup of~$G$, and
$\phi\colon H\to K$ an embedding, then the coproduct with amalgamation
$G\amalg_\phi^{\mathfrak{N}_2} K$ may or may not contain copies of~$G$
and~$K$; and even if it does contain copies of~$G$ and~$K$, their
intersection may be strictly larger than~$H$. There are necessary and
sufficient conditions for each of the situations, given in
\cites{amalgone,amalgtwo,amalgams}. When $G$ and~$K$ are of
exponent~$p$ and the identified subgroups are contained in the
corresponding commutator subgroups, however, 
$G\amalg_{\phi}^{\mathfrak{N}_2} K$ always contains copies of $G$ and~$K$,
and these copies intersect exactly at~$H$.

As before, the statement of the linear algebra result is somewhat
complex; the group-theoretic version is unfortunately not as simple as
it was in the case above, so we will present instead an easy-to-state
consequence. 

\begin{theorem} Let $n>3$ and let $r$ be an integer, $2\leq r\leq
  n-2$.   Let $X_s$ and $X_{\ell}$ be
  subspaces of $V_s$ and~$V_{\ell}$, respectively, and let $H$ be a
  subspace of~$V_s$ such that $H\cap X_s=\{\mathbf{0}\}$. Let
  $\phi\colon H\to V_{\ell}$ be an embedding such that $\phi(H)\cap
  X_{\ell}=\{\mathbf{0}\}$. Finally, let $X$ be the subspace of~$V$
  given by
$X = X_s\oplus X_{\ell} \oplus \{ h - \phi(h)\,|\, h\in
  H\}$.
If ${\rm cl}_s(X_s)$ is the $\{\varphi_i\}_{i=1}^r$-closure of $X_s$ and
${\rm cl}_{\ell}(X_{\ell})$ is the $\{\varphi_i\}_{i=r+1}^n$-closure of $X_{\ell}$,
then the closure of $X$ is given by:
\[ X^{**} = X \oplus \bigl\{ h \in H\,\bigm|\, h\in {\rm cl}_s(X_s)\mbox{\ and\
}\phi(h)\in {\rm cl}_{\ell}(X_{\ell})\bigr\}.\]
In particular, $X$ is closed if and only if
\[\bigl\{ h \in H\,\bigm|\, h\in {\rm cl}_s(X_s)\mbox{\ and\ }\phi(h)\in
{\rm cl}_{\ell}(X_{\ell})\bigr\} = \{\mathbf{0}\}.\]
\label{thm:closureamalgcoprod}
\end{theorem}

\begin{proof} Note that $X\subset (X_s\oplus H) \oplus (X_{\ell}\oplus
  \phi(H))$; the latter subspace is closed by
  Theorem~\ref{th:bigandsmall}(iii), so it contains $X^{**}$.
  Thus, to describe the closure of~$X$ it is enough
  to determine exactly which $h\in H$ lie in the closure.

Suppose that $h\in H\cap X^{**}$. Then $\varphi_i(h)\in X^*$ for
$i=1,\ldots,n$; fix $i\leq r$. Then we know that there exist
$x_1,\ldots,x_n\in X_s$, $y_1,\ldots,y_n\in X_{\ell}$, and
$h_1,\ldots,h_n\in H$ such that
\[
\varphi_i(h) = \varphi_1(x_1+y_1+h_1-\phi(h_1)) + \cdots +
\varphi_n(x_n+y_n+h_n-\phi(h_n)).\]
By looking at the $W_s$, $W_{ms}$, $W_{m\ell}$, and~$W_{\ell}$
components, we deduce that:
\begin{eqnarray*}
\varphi_i(h) & = & \varphi_1(x_1+h_1) + \cdots + \varphi_r(x_r+h_r),\\
\mathbf{0} & = & \varphi_{r+1}(x_{r+1}+h_{r+1}) + \cdots +\varphi_n(x_n+h_n),\\
\mathbf{0}& = & \varphi_1(y_1-\phi(h_1))+\cdots + \varphi_r(y_r -\phi(h_r)),\\
\mathbf{0} & = & \varphi_{r+1}(y_{r+1}-\phi(h_{r+1}))+\cdots+\varphi_n(y_n-\phi(h_n)). 
\end{eqnarray*}
Now, $\varphi_{r+1}(x_{r+1}+h_{r+1})$ is the only term in the
expression that lies in $W_{ms}$ and involves generators $w_{jik}$
with one of $j$ or~$k$ (in fact, $k$) equal to $r+1$. Thus, we must
have $\varphi_{r+1}(x_{r+1}+h_{r+1})=\mathbf{0}$, which in turn gives
$x_{r+1}=h_{r+1}=\phi(h_{r+1})=\mathbf{0}$, since
$\varphi_{r+1}$ and $\phi$ are embeddings and $H\cap
X_s=\mathbf{0}$. Similarly, we deduce that
$x_{r+1}=x_{r+2}=\cdots = x_{n} = h_{r+1} = h_{r+2}=\cdots = h_n =
\mathbf{0}$.
We are then left with
$\varphi_{r+1}(y_{r+1})+\cdots+\varphi_n(y_n)=\mathbf{0}$ as the only
equation involving $y_{r+1},\ldots,y_n$, and so we may also assume
$y_{r+1}=\cdots=y_n=\mathbf{0}$. 

Consider now $\varphi_1(y_1-\phi(h_1))+\cdots +
\varphi_r(y_r-\phi(h_r))$. Again, $\varphi_1(y_1-\phi(h_1))$ is the
only term in the expression that lies in $W_{m\ell}$ and involves
generators $w_{jik}$ with $i=1$. Thus, we must have
$\varphi_1(y_1-\phi(h_1))=\mathbf{0}$, and as above we deduce from
this that $y_1=\phi(h_1)=\mathbf{0}$ since $\varphi_1$ and $\phi$
are embeddings and $\phi(H)\cap X_{\ell}=\mathbf{0}$. Similarly, we
obtain
$y_1=y_2 = \cdots = y_r = h_1=\cdots = h_r = \mathbf{0}$.
And so we obtain $\varphi_i(h) = \varphi_1(x_1)+\cdots
+\varphi_r(x_r)$ for some vectors $x_1,\ldots,x_r\in X_s$. That is, if
$h\in H\cap X^{**}$, then $h$ is ${\rm cl}_s(X_s)$.

A symmetric argument, considering $\varphi_i(\phi(h))$ with $i>r$
yields that if $\varphi(h)$ lies in $X^{**}$, then $\varphi(h)$
must lie in ${\rm cl}_{\ell}(X_{\ell})$. Since
$h-\phi(h)\in X^{**}$ for all $h\in H$, we obtain that a necessary
condition for $h\in H$ to lie in the closure is that $h\in {\rm
  cl}_s(X_s)$ and $\phi(h)\in{\rm cl}_{\ell}(X_{\ell})$.
The theorem will be
proven if we can show that this condition is also sufficient. 

Suppose that $h\in H\cap\, {\rm cl}_{s}(X_s)$ is 
such that $\phi(h)$ lies
in ${\rm cl}_{\ell}(X_{\ell})$. Then each of
$\varphi_1(h),\ldots,\varphi_r(h)$,
$\varphi_{r+1}(\phi(h)),\ldots,\varphi_n(\phi(h))$ lie in $X^*$.
Since $\varphi_i(h - \phi(h))\in X^*$ for all~$i$, we deduce that
$\varphi_i(h)\in X^*$ for all~$i$, so $h\in X^{**}$.
This proves that the condition given is also sufficient, and so
proves the theorem.
\end{proof}

\begin{corollary} Let $G$ and $K$ be two nonabelian groups of class
  two and exponent~$p$. Let $H$ be a nontrivial subgroup of $[G,G]$,
  and let $\phi\colon H\to [K,K]$ be an embedding. If either $G$ or
  $K$ are capable, then so is the amalgamated coproduct
  $G\amalg_\phi^{\mathfrak{N}_2} K$. 
\label{cor:sufficientamalgcoprod}
\end{corollary}

\begin{remark} It is perhaps interesting to note that
  when we passed from coproducts and direct products to their
  amalgamated counterparts, a kind of reversal took place. The
  coproduct of two nontrivial groups in our class is always
  capable, while the capability of the direct product depends on the
  capability of the two factors. However, when we amalgamate
  nontrivial subgroups of the commutators, then the amalgamated direct
  product which is \textit{never} capable, while it is in the
  amalgamated coproduct that capability depends on the capability of
  the two groups (and the precise choice of~$H$).
\end{remark}

\section{Dimension Counting.}\label{sec:dimcounting}

In this section we will establish a sufficient condition for
the capability of a $p$-group $G$ of exponent~$p$ and class at most
two that depends only on the ranks of $G/Z(G)$ and $[G,G]$. The idea
is the following: given a subspace $X$ of~$V$, we will find a lower
bound for the dimension of~$X^*$ in terms of $n$ and the
dimension of~$X$. If all subspaces $X'$ of $V$ that properly contain~$X$
yield subspaces $X'^*$ of dimension strictly larger
than $\dim(X^*)$, then it will follow that $X$ must be closed since
$X^* = (X^{**})^*$. In order to establish these bounds, we will
consider the images $\varphi_1(X)$,
$\varphi_2(X),\ldots,\varphi_n(X)$; since each $\varphi_i$ is
one-to-one, the dimension of $X^*$ will depend on how much ``overlap''
there can be among these subspaces of~$W$. 

\begin{lemma} Fix $n>1$, and let $i$ and~$j$ be integers, $1\leq
  i<j\leq n$. Then $\varphi_i(V)\cap\varphi_j(V)=\{\mathbf{0}\}$.
\end{lemma}

\begin{proof} Let $\varphi_i(\mathbf{v})\in \varphi_j(V)$, and assume
  that $\pi_{sr}(\mathbf{v})\neq\mathbf{0}$, $1\leq r<s\leq n$. If
  $r\leq i$, then $\pi_{sri}(\varphi_i(\mathbf{v}))\neq\mathbf{0}$,
  and since $\varphi_i(\mathbf{v})\in\varphi_j(V)$,
  Lemma~\ref{lemma:indexinequalities} implies
  $r\leq j\leq i$, contradicting the choice of $i$ and~$j$. 
  If $i<r$, then $\pi_{sir}(\varphi_i(\mathbf{v}))\neq
  \mathbf{0}$. By Lemma~\ref{lemma:indexinequalities}, we must have $i<
  j=r$. We also have $\pi_{ris}(\varphi_i(\mathbf{v}))\neq\mathbf{0}$,
  and since $\varphi_i(\mathbf{v})\in\varphi_j(V)$, this time we
  deduce $i<j=s$. But then we have $j=r=s$, and this is
  impossible. This contradiction arises from assuming
  $\pi_{sr}(\mathbf{v})\neq\mathbf{0}$ for some $1\leq r<s\leq n$,
  hence $\mathbf{v}=\mathbf{0}$.
\end{proof}

\begin{lemma} Fix $n>1$ and $r\leq n$. Let $i_1,\ldots,i_r$ be
 pairwise distinct integers, $1\leq i_1,\ldots, i_r\leq n$. Then
 $\varphi_{i_1}^{-1}\bigl(\bigl\langle
 \varphi_{i_2}(V),\ldots,\varphi_{i_r}(V)\bigr\rangle\bigr)$ is of
 dimension $\binom{r-1}{2}$, with basis given by the vectors $v_{ab}$,
 with $a,b\in \{i_2,\ldots,i_r\}$, $b<a$. In particular, the
 intersection
 $\varphi_{i_1}(V)\cap\bigl\langle\varphi_{i_2}(V),\ldots,\varphi_{i_r}(V)\bigr\rangle$
 has a basis made up of vectors of the form $w_{abi_1}$ with
 $a,b\in\{i_2,\ldots,i_r\}$, $b<a$ and $b<i_1$; and vectors of the
 form $w_{ai_1b}-w_{bi_1a}$, with $a,b\in\{i_2,\ldots,i_r\}$,
 $i_1<b<a$.
\label{lemma:pullback}
\end{lemma}

\begin{proof} By Proposition~\ref{prop:symmetry}, it is enough to
  consider the case where $i_1=1$.  Let $A$ denote the
  pullback described in the statement.

Given $a,b\in\{i_2,\ldots,i_r\}$, $a>b$, we have $v_{ab}\in A$:
\[ \varphi_{i_1}(v_{ab}) = w_{ai_1b}-w_{bi_1a} =
\varphi_b(v_{ai_1})-\varphi_a(v_{bi_1})\in \langle
\varphi_{i_2}(V),\ldots,\varphi_{i_r}(V)\rangle.\]
Conversely, let $\mathbf{v}\in A$, and let $a,b$ be integers, $1\leq
b<a\leq n$, such that $\pi_{ab}(\mathbf{v})\neq \mathbf{0}$. We can
write
\[ \varphi_{i_1}(\mathbf{v}) =
\varphi_{i_2}(\mathbf{v}_2)+\cdots+\varphi_{i_r}(\mathbf{v}_r).\]
Since $i_1=1$,
$\pi_{a1b}(\varphi_{i_1}(\mathbf{v}))=-\pi_{b1a}(\varphi_{i_1}(\mathbf{v}))\neq
\mathbf{0}$, and therefore we must have
$\pi_{a1b}(\varphi_{i_j}(\mathbf{v}_j))\neq\mathbf{0}$ for some $j\geq
2$. This implies $1\leq i_j\leq b$, with at most one inequality strict
by Lemma~\ref{lemma:indexinequalities}. Since $1=i_1\neq i_j$, we have
$i_j=b$. Considering $\pi_{b1a}$ instead,
we deduce that $a=i_k$ for some $k\geq 2$, so
$a,b\in\{i_2,\ldots,i_r\}$. 
Therefore, $A\subseteq\langle v_{ab}\,|\,
a,b\in\{i_2,\ldots,i_r\}, a>b\rangle$. This proves equality. 

Since the vectors described are linearly independent, they form a
basis. Mapping them via $\varphi_{i_1}$, which is one-to-one, proves
the final clause. 
\end{proof}

\begin{corollary}
Let $n>1$, $r\leq n$, and let $1\leq i_1< i_2 <\cdots < i_r\leq n$ be
integers. Then
\[ \dim\left(\langle
\varphi_{i_1}(V),\ldots,\varphi_{i_r}(V)\rangle\right) = r\binom{n}{2}
- \binom{r}{3}.\]
\end{corollary}

\begin{proof} For simplicitly, let
  $Y=\langle\varphi_{i_1}(V),\ldots,\varphi_{i_r}(V)\rangle$. We have:
\begin{eqnarray*}
\dim(Y) & = & \left(\sum_{k=1}^r \dim(\varphi_{i_k}(V))\right) -
\left(\sum_{k=2}^{r}\dim\Bigl( \varphi_{i_{k}}(V)\cap \bigr\langle
\varphi_{i_{1}}(V),\ldots,\varphi_{i_{k-1}}(V)\bigr\rangle \Bigr)\right)\\
& = & r\binom{n}{2} - \left(\sum_{k=2}^{r}\binom{k-1}{2}\right) = r\binom{n}{2}-\binom{r}{3},
\end{eqnarray*}
as claimed.
\end{proof}

\begin{definition} Fix $n>1$. We define $\Phi\colon V^n\to W$ to be
\[ \Phi(\mathbf{v}_1,\ldots,\mathbf{v}_n) =
\varphi_1(\mathbf{v}_1)+\cdots + \varphi_n(\mathbf{v}_n).\]
If there is danger of ambiguity, we use $\Phi_n$ to denote the map
associated to the spaces corresponding to the particular choice
of~$n$.
\label{defn:defofPhi}
\end{definition}

Note that if $X$ is a subspace of $V$, then $\Phi(X^n) = X^*$. 

\begin{prop} The kernel of $\Phi$ is of dimension
  $\binom{n}{3}$. A basis for ${\rm ker}(\Phi)$ can be determined as
  follows: each choice of integers $a,b,c$, $1\leq a<b<c\leq n$,
  gives an element $(\mathbf{v}_1,\ldots,\mathbf{v}_n)\in V^n$ of the
  basis, with:
\[ \mathbf{v}_i = \left\{ \begin{array}{ll}
v_{cb}&\mbox{if $i=a$,}\\
-v_{ca}&\mbox{if $i=b$,}\\
v_{ba}&\mbox{if $i=c$,}\\
\mathbf{0}&\mbox{otherwise.}
\end{array}\right.\]
\label{prop:kerofPhi}
\end{prop}

\begin{proof}
Denote the element corresponding to $a<b<c$ by
$\mathbf{v}_{(abc)}$. Note that $\mathbf{v}_{(abc)}$ is in ${\rm
ker}(\Phi)$:
\[ \Phi(\mathbf{v}_{(abc)}) = \varphi_a(v_{cb}) +
\varphi_b(-v_{ca})+\varphi_c(v_{ba}) = w_{cab}-w_{bac} - w_{cab} +
w_{bac} = \mathbf{0}.\]
Since $\Phi$ is surjective, $\dim(W) = n\dim(V) - \dim({\rm
  ker}(\Phi))$, hence
\[ \dim({\rm ker}(\Phi)) = n\binom{n}{2} -
2\binom{n+1}{3}=\binom{n}{3},\]
so the proposition will be established in full if we prove that the
elements $\mathbf{v}_{(abc)}$ of $V^n$ are linearly independent.

Let $\sum\beta_{abc}\mathbf{v}_{abc} = (\mathbf{0},\ldots,\mathbf{0})$ be a linear
combination equal to~zero. If we look at the $i$th coordinate of these
$n$-tuples, we have:
\[ \sum_{1\leq r<s<i\leq n}\beta_{rsi}v_{sr} - \sum_{1\leq r<i<s\leq
  n}\beta_{ris}v_{sr} + \sum_{1\leq i<r<s\leq n}\beta_{irs}v_{sr} =
  \mathbf{0}.\]
Each basis vector $v_{sr}$ occurs only once. Thus, if $i\in\{a,b,c\}$,
  then $\beta_{abc}=0$. This holds for each choice of $i$,
 hence $\beta_{abc}=0$ for all choices of $a,b,c$. This proves the
  $\mathbf{v}_{(abc)}$ are linearly independent.
\end{proof}

\begin{theorem}
Let $(\mathbf{v}_1,\ldots,\mathbf{v}_n)\in\ker(\Phi)$. Write
\[ \mathbf{v}_k = \sum_{1\leq i<j\leq n} \!\!\!\!\alpha_{ji}^{(k)}v_{ji},\]
\begin{itemize}
\item[(i)] If $i=k$ or $j=k$, then $\alpha_{ji}^{(k)}=0$; i.e.,
  $\Pi_k(\mathbf{v}_k) = \mathbf{0}$.
\item[(ii)] If $1\leq a<b<c\leq n$, then $\alpha_{ba}^{(c)} =
  \alpha_{cb}^{(a)} = -\alpha_{ca}^{(b)}$.
\item[(iii)] Fix $i,j$, $1\leq i<j\leq n$. Then
\begin{eqnarray*}
\Pi_{i}(\mathbf{v}_{j})& = &
\sum_{r=1}^{i-1}\left(-\alpha_{jr}^{(i)}\right)v_{ir}    +
 \sum_{r=i+1}^{j-1} \alpha_{jr}^{(i)}v_{ri} +
\sum_{r=j+1}^n \left(-\alpha_{rj}^{(i)}\right)v_{ri},\\
\Pi_{j}(\mathbf{v}_i) & = &
\sum_{r=1}^{i-1}\left(-\alpha_{ir}^{(j)}\right)v_{jr} +
\sum_{r=i+1}^{j-1}\alpha_{ri}^{(j)}v_{jr} +
\sum_{r=j+1}^n\left(-\alpha_{ri}^{(j)}\right)v_{rj}.
\end{eqnarray*}
\end{itemize}
\label{th:descriptionkernelPhi}
\end{theorem}

\begin{proof} 
Part (i) holds for the basis elements described in
Proposition~\ref{prop:kerofPhi}, hence holds for all vectors in the
kernel. For part (ii), note that if $1\leq a<b<c\leq n$, then
\begin{eqnarray*}
\pi_{bac}\bigl(\varphi_1(\mathbf{v}_1) + \cdots
+\varphi_{n}(\mathbf{v}_n)\bigr) &=&
\left(\alpha_{ba}^{(c)}-\alpha_{cb}^{(a)}\right)w_{bac},\\
\pi_{cab}\bigl(\varphi_1(\mathbf{v}_1) + \cdots +
\varphi_n(\mathbf{v}_n)\bigr) & = & \left(\alpha_{ca}^{(b)} +
\alpha_{cb}^{(a)}\right)w_{cab}.
\end{eqnarray*}
Since both are equal to zero, we deduce that
$\alpha_{ba}^{(c)}=\alpha_{cb}^{(a)}$ and
$\alpha_{ca}^{(b)}=-\alpha_{cb}^{(a)}$. Finally, for (iii), we know
that
$\Pi_i(\mathbf{v}_i) = \Pi_j(\mathbf{v}_j)=\mathbf{0}$ from (i), so we
can write:
\begin{eqnarray*}
\Pi_i(\mathbf{v}_j) & = & \sum_{r=1}^{i-1}\alpha_{ir}^{(j)}v_{ir} +
\sum_{r=i+1}^{j-1}\!\!\alpha_{ri}^{(j)}v_{ri} +
\sum_{r=j+1}^{n}\!\!\!\alpha_{ri}^{(j)}v_{ri},\\
\Pi_j(\mathbf{v}_i) & = & \sum_{r=1}^{i-1}\alpha_{jr}^{(i)}v_{jr} +
\sum_{r=i+1}^{j-1}\!\!\alpha_{jr}^{(i)}v_{jr} + \sum_{r=j+1}^n
\!\!\!\alpha_{rj}^{(i)}v_{rj},
\end{eqnarray*}
and applying (ii) gives the desired identities.
\end{proof}

\begin{corollary} Let $\mathbf{v}\in\ker(\Phi)$. If
  $\Pi_j(\mathbf{v}_i)=\mathbf{0}$, then
  $\Pi_i(\mathbf{v}_j)=\mathbf{0}$. In particular, if
  $\mathbf{v}_i=\mathbf{0}$, then $\Pi_i(\mathbf{v}_j)=\mathbf{0}$ for
  all~$j$.
\label{cor:swaptogetzero}
\end{corollary}

\begin{proof} The second assertion follows immediately from the
  first. The first assertion is trivial if $i=j$; for $i\neq j$, then
  $\alpha_{jr}^{(i)}=0$ for all $r<j$ and $\alpha_{rj}^{(i)}=0$ for $j<r$, so by
  Theorem~\ref{th:descriptionkernelPhi}(iii) it follows that
  $\Pi_i(\mathbf{v}_j)=\mathbf{0}$. 
\end{proof}

\begin{corollary} Let $\mathbf{v}\in\ker(\Phi)$, $\mathbf{v}\neq
  (\mathbf{0},\ldots,\mathbf{0})$. If
$\mathbf{v} = (\mathbf{v}_1,\ldots,\mathbf{v}_n)$
then the dimension of $\langle
\mathbf{v}_1,\ldots,\mathbf{v}_n\rangle$ is at least $3$.
\label{cor:atleastthree}
\end{corollary}

\begin{proof} Write
\[ \mathbf{v} = \!\!\!\!\!\sum_{1\leq a<b<c\leq n}\!\!\!\!\!\!\!
\beta_{abc}\mathbf{v}_{(abc)}.\]
Fix $a,b,c$ such that $1\leq a<b<c\leq n$, $\beta_{abc}\neq 0$. We
claim that $\mathbf{v}_a$, $\mathbf{v}_b$, and $\mathbf{v}_c$ are
linearly independent. Indeed, note that
$\Pi_a(\mathbf{v}_a)=\Pi_b(\mathbf{v}_b)=\Pi_c(\mathbf{v}_c) =
\mathbf{0}$, and $\pi_{cb}(\mathbf{v}_a) \neq\mathbf{0}$. Therefore,
if $\alpha_a\mathbf{v}_a + \alpha_b\mathbf{v}_b + \alpha_c\mathbf{v}_c
= \mathbf{0}$, then we must have $\alpha_a=0$. A symmetric argument
looking at $\pi_{ca}$ shows that $\alpha_b=0$, and considering
$\pi_{ba}$ shows that $\alpha_c=0$.
\end{proof}

\begin{corollary}[Prop.~4.6 in~\cite{capablep}]
Fix $n>1$, and let $X$ be a subspace of~$V$. If $\dim(X)=1$, then
$\dim(X^*)=n$; if $\dim(X)=2$, then $\dim(X^*)=2n$.
\label{cor:maincountingarg}
\end{corollary}

\begin{proof} We prove the contrapositive.
Since $\dim(X^*) = n\dim(X) - \dim(X^n\cap\ker(\Phi))$, if
$\dim(X^*)<n\dim(X)$, then $X^n\cap\ker(\Phi)\neq\{\mathbf{0}\}$.

Let $\mathbf{v}=(\mathbf{v}_1,\ldots,\mathbf{v}_n)\in
X^n\cap\ker(\Phi)$, $\mathbf{v}\neq \mathbf{0}$. Then $\mathbf{v}_i\in
X$ for $i=1,\ldots,n$, so by Corollary~\ref{cor:atleastthree},
$\dim(X)\geq 3$, as claimed.
\end{proof}

We now proceed along the lines outlined at the beginning of the
section. We first formalize the observation made there.

\begin{prop}
Let $X<V$. Assume that for all subspaces $Y$ of\/ $V$, if\/ $Y$
properly contains $X$
then $Y^*$ properly contains $X^*$. Then $X=X^{**}$.
\label{prop:strictlylarger}
\end{prop}

\begin{proof} If $X^{**}$ properly contains $X$, then $X^{***}$ would
  properly contain $X^*$. But $X^{***}=X^*$, a contradiction.
\end{proof}

We are therefore searching for a function $f(k,n)$, defined for $k$
with $1\leq k\leq \binom{n}{2}$, such that for all
$X<V(n)$, if $\dim(X)=k$ then $\dim(X^*)\geq nk - f(k,n)$. 
This is given by:
\[ f(k,n) = \max\bigl\{ \dim(X^n\cap\ker(\Phi_n))\,\bigm|\,
X<V, \dim(X)=k\}.\] Our objective in this section is to find an
expression for $f(k,n)$ in terms of~$k$ and~$n$; in fact, it turns out that
the value is independent of $n$. 
The main workhorse in our calculations will
Lemma~\ref{lemma:basicoverlapbound} below.  The idea is to find
$\dim(X^n\cap\ker(\Phi_n))$ by examining the ``partial
intersections''; namely, the intersections of the form
\[ \Bigl\langle
(\mathbf{0},\ldots\mathbf{0},\mathbf{v}_i,
\mathbf{v}_{i+1},\ldots,\mathbf{v}_n) \,\Bigm|\, \mathbf{v}_j\in
X\Bigr\rangle \bigcap \ker(\Phi_n),\]
as $i$ ranges from $1$ to $n-2$ (when $i=n-1$ or $i=n$, the
intersection is trivial by Corollary~\ref{cor:atleastthree}). For a
fixed $i$, we can consider the
subspace of $X$ consisting of all vectors $\mathbf{v}_i$ which can be
``completed'' to an element of $\ker(\Phi)$ by taking and $n$-tuple
with $i-1$ copies of $\mathbf{0}$, followed by $\mathbf{v}_i$,
followed by some vectors in~$X$; this is the same as
considering the pullbacks
$ X \cap
\varphi_i^{-1}\left(\langle\varphi_{i+1}(X),\ldots,\varphi_{n}(X)\rangle\right)$.
It is easy to verify that the sum of the dimensions of these pullbacks
is equal to the dimension of $X^n\cap\ker(\Phi_n)$.  We will first use
the dimension of these pullbacks to establish a lower bound for the
dimension of~$X$; then we will turn around and use these calculations
to give an upper bound for the dimension of the pullbacks in terms of
the dimension of~$X$.

Making the bounds as precise as possible, however, requires one to
keep track of a lot of information; this in turn requires the use of
multiple indices and subindices in the proof, for which I apologize in
advance. To illustrate the ideas and help the reader navigate through
the proof, we will first present an illustration. This is not an example in
the sense of a specific $X$, but rather a run-through the main part of
the analysis we will perform below, but with specific values for some
of the indices and some of the variables to make it more concrete.

\begin{example} Set $n=6$, and let $X$ be a subspace of~$V$. We will
be interested in bounding above the dimension of $Z_i$ in terms of
$\dim(X)$, where
\[Z_i = X \cap \varphi_i^{-1}\Bigl(\langle
\varphi_{i+1}(X),\ldots,\varphi_6(X)\rangle\Bigr);\]
i.e., $Z_i$ consists of all $\mathbf{v}\in X$ for which there exist
$\mathbf{v}_{i+1},\ldots,\mathbf{v}_6$ in~$X$ such that
\[
(\mathbf{0},\ldots,\mathbf{0},\mathbf{v},\mathbf{v}_{i+1},\ldots,\mathbf{v}_6)
  \in X^6\cap\ker(\Phi_6).\] To do this, we will obtain a lower bound
  for $\dim(X)$ in terms of $\dim(Z_i)$.  To further fix ideas, set
  $i=2$.  Note that by Theorem~\ref{th:descriptionkernelPhi}(i)
  and~(ii), we must have $\Pi_1(Z_2)=\Pi_2(Z_2)=\mathbf{0}$. 
  Order all pairs $(j,i)$
  lexicographically from right to left, so $(j,i)<(b,a)$ if and only
  if $i<a$, or $i=a$ and $j<b$. 
Doing row reduction, we can find a basis $\mathbf{v}_{1,2}$,
$\mathbf{v}_{2,2},\ldots,\mathbf{v}_{k,2}$
for $Z_2$ (the second index refers to the fact that
these vectors are in the second component of an element of
$\ker(\Phi_6)$), satisfying that the ``leading pair'' (smallest nonzero
component) of each is strictly smaller than that of its successors,
and all other vectors have zero component for that pair. For example,
suppose that $\dim(Z_2)=4$, and that the basis has the form:
\begin{alignat*}{2}
\mathbf{v}_{1,2} & =  v_{43} + \alpha_1 v_{53} + \alpha_2
v_{64},&\qquad
\mathbf{v}_{3,2} & =  v_{54} + \gamma v_{64},\\
\mathbf{v}_{2,2} & =  v_{63} + \beta v_{64},&
\mathbf{v}_{4,2} & =  v_{65},
\end{alignat*}
for some coefficients $\alpha_1,\alpha_2,\beta,\gamma\in\mathbb{F}_p$.
We know there exist vectors $\mathbf{v}_{i,3}$, $\mathbf{v}_{i,4}$,
$\mathbf{v}_{i,5}$, $\mathbf{v}_{i,6}$ such that
$
(\mathbf{0},\mathbf{v}_{i,2},\mathbf{v}_{i,3},\mathbf{v}_{i,4},\mathbf{v}_{i,5},\mathbf{v}_{i,6})\in
X^6\cap\ker(\Phi_6)$ for $i=1,2,3,4$. Naturally, $X$ contains all
twenty vectors, but there will normally be some linear
dependencies between them: some may even be equal to~$\mathbf{0}$. We
want to extract, in some systematic manner, a subset that we can
guarantee is linearly independent. First let us consider the information we can
obtain about
these vectors from our knowledge of the vectors~$\mathbf{v}_{i,2}$.

Since
$(\mathbf{0},\mathbf{v}_{i,2},\mathbf{v}_{i,3},\mathbf{v}_{i,4},\mathbf{v}_{i,5},\mathbf{v}_{i,6})$
lies in $\ker(\Phi)$, we can use
Theorem~\ref{th:descriptionkernelPhi}(iii) to describe the
$\Pi_i$-image of each
vector $\mathbf{v}_{i,j}$, where $i\leq 2$ and $j>2$.
The $\Pi_1$-image must be
trivial, and for the $\Pi_2$ image we obtain
the following:
\begin{alignat*}{2}
\Pi_2(\mathbf{v}_{1,3}) & = v_{42} + \alpha_1 v_{52},&{\qquad}
\Pi_2(\mathbf{v}_{2,3}) & = v_{62},\\
\Pi_2(\mathbf{v}_{1,4}) & = -v_{32} + \alpha_2 v_{62},&
\Pi_2(\mathbf{v}_{2,4}) & = \beta v_{62},\\
\Pi_2(\mathbf{v}_{1,5}) & = -\alpha_1 v_{32},&
\Pi_2(\mathbf{v}_{2,5}) & = \mathbf{0},\\
\Pi_2(\mathbf{v}_{1,6}) & = -\alpha_2 v_{42}.&
\Pi_2(\mathbf{v}_{2,6}) & = -v_{32} - \beta v_{42}.\\
\\
\Pi_2(\mathbf{v}_{3,3}) & = \mathbf{0},&\qquad
\Pi_2(\mathbf{v}_{4,3}) & = \mathbf{0},\\
\Pi_2(\mathbf{v}_{3,4}) & = v_{52} + \gamma v_{62},&
\Pi_2(\mathbf{v}_{4,4}) & = \mathbf{0},\\
\Pi_2(\mathbf{v}_{3,5}) & = - v_{42}, &
\Pi_2(\mathbf{v}_{4,5}) & = v_{62},\\
\Pi_2(\mathbf{v}_{3,6}) & = -\gamma v_{42}.&
\Pi_2(\mathbf{v}_{4,6}) & = -v_{52}.
\end{alignat*}
One way to obtain these without too much confusion is as follows: to
find $\Pi_2\left(\mathbf{v}_{j,k}\right)$, go through the expression
for $\mathbf{v}_{j,2}$ replacing all indices $k$ by $2$, remembering
that
$v_{ab}=-v_{ba}$. Any $v_{ba}$ in which neither $a$ nor~$b$ are equal
to $k$ are simply removed.

To extract systematically a set of linearly independent vectors, we
proceed in the following manner: consider all the pairs which are
leading components of the basis vectors $\mathbf{v}_{i,2}$; in this
case, $(4,3)$, $(6,3)$, $(5,4)$, and~$(6,5)$.  The individual indices
that occur are $3$, $4$, $5$, and $6$. For each of them, we identify
the smallest pair in which it occurs. Thus, $3$ first occurs in pair
number one, as does $4$. The index $5$ first occurs in pair number
three, and $6$ first occurs in pair number two.

Since the first pair in which $3$ appears is the \textit{first} pair
(corresponding to the first basis vectors $\mathbf{v}_{1,2}$,
$(4,3)$, where it is paired with~$4$, we will select the vector
$\mathbf{v}_{1,4}$; this vector has first nontrivial component $(3,2)$.
The next index is~$4$, again in the
\textit{first} pair, paired with~$3$; so this time we select
$\mathbf{v}_{1,3}$. This has nontrivial $(4,2)$ copmonent, and trivial
$(j,i)$ component for all $(j,i)<(4,2)$. 

The next index is~$5$, which first occurs in the third pair
(corresponding to $\mathbf{v}_{3,2}$) paired with~$4$. So we select
$\mathbf{v}_{3,4}$, a vector with trivial $(j,i)$ component for all
$(j,i)<(5,2)$, and nontrivial $(5,2)$ component. Next we go to the
index~$6$, that first occurs in second pair together with~$3$; so we
select the vector $\mathbf{v}_{2,3}$, a vector with nontrivial $(6,2)$
component, and trivial $(j,i)$ component for all $(j,i)<(6,2)$.

In summary, we want to consider our original basis vectors
$\mathbf{v}_{1,2}$, $\mathbf{v}_{2,2}$, $\mathbf{v}_{2,3}$,
and~$\mathbf{v}_{2,4}$, plus the vectors we have selected based on the
location of the indices, to wit the vectors $\mathbf{v}_{1,4},
\mathbf{v}_{1,3}, \mathbf{v}_{3,4}, \mathbf{v}_{2,3}$ corresponding,
respectively, to the indices $3$, $4$, $5$, and~$6$. The choices we
have made ensure that the $\Pi_2$-images of these latter four vectors
are linearly independent, and so the vectors themselves must be
linearly independent.  Since $\Pi_2(Z_2)=\mathbf{0}$, the full
collection of eight vectors is linearly independent, and so we can
conclude that~$X$ must have dimension at least $8$.

What is more, note that none of the four vectors $\mathbf{v}_{1,4}$,
$\mathbf{v}_{1,3}$, $\mathbf{v}_{3,4}$, and~$\mathbf{v}_{2,3}$ will
occur in a similar analysis involving $Z_3$ (or more generally $Z_i$
with $i>2$): when performing a similar analysis, all vectors will have
trivial $\Pi_i$-image when $i<3$, and these vectors have nontrivial
$\Pi_2$-image.  Note as well that the number of indices, in this case
$4$, must satisfy $\dim(Z_2)\leq\binom{4}{2}$, since we need to be
able to obtain at least $\dim(Z_2)$ pairs out of the indices that
occur.

Thus we have seen that if $\dim(Z_2)=4$, then $\dim(X)\geq 8$. If we
move on to $Z_3$, we will obtain new vectors that must lie in~$X$;
while the vectors in the basis for $Z_2$ may again occur in that
analysis, the vectors $\mathbf{v}_{1,4}$, $\mathbf{v}_{1,3}$,
$\mathbf{v}_{3,4}$, and~$\mathbf{v}_{2,3}$ will not, and so by keeping
track of them we can give an even better lower bound for $\dim(X)$.~$\Box$
\end{example}

What ensures that this process will work the way we want is how
we choose the vectors of the basis and the vectors that ``correspond''
to each index. The former count towards the value of
$\dim(X^n\cap\ker(\Phi_n))$, while the latter may be removed from
consideration when we move on to $Z_{i+1}$.
This is all done in generality in the proof
of the following promised lemma:

\begin{lemma} Fix $n>1$, and let $X$ be a subspace of $V$. For each $i$, $1\leq i\leq
  n$, let
\[Z_i = X \cap \varphi_i^{-1}\Bigl(\langle
\varphi_{i+1}(X),\ldots,\varphi_n(X)\rangle\Bigr);\]
i.e., $Z_i$ consists of all $\mathbf{v}\in X$ for which there exist
$\mathbf{v}_{i+1},\ldots,\mathbf{v}_n$ in~$X$ such that
\[
(\mathbf{0},\ldots,\mathbf{0},\mathbf{v},\mathbf{v}_{i+1},\ldots,\mathbf{v}_n)
  \in X^n\cap\ker(\Phi).\]
If
$\dim\bigl(X\cap\langle v_{sr}\,|\,i\leq r<s\leq
  n\rangle\bigr)=d_i$ and
$\dim(Z_i)=r_i$, then $r_i\leq \binom{d_i-r_i}{2}$. Morevoer, if
$s_i$ is the smallest positive integer such that
$r_i\leq\binom{s_i}{2}$, then we must have $d_{i+1}\leq d_i-s_i$.
\label{lemma:basicoverlapbound}
\end{lemma}

\begin{proof} Fix $i_0$, $1\leq i_0\leq n$. For simplicity, write
  $r=r_{i_0}$.
By Theorem~\ref{th:descriptionkernelPhi}, if $\mathbf{v}\in Z_{i_0}$
then
$\Pi_i(\mathbf{v})=\mathbf{0}$ for all $i\leq i_0$.

Let $\mathbf{v}_{1i_0},\ldots,\mathbf{v}_{ri_0}$ be a basis for
 $Z_{i_0}$. We will modify it as follows:

Order all pairs $(j,i)$, $i_0<i<j\leq n$ by letting $(j,i)<(b,a)$ if
 and only if $i<a$ or $i=a$ and $j<b$ (lexicographically from right to
 left).  Let $(j_1,i_1)$ be the smallest pair for which
 $\pi_{j_1i_1}(\mathbf{v}_{ki_0})\neq\mathbf{0}$ for some $k$, $1\leq
 k\leq r$. Reordering if necessary we may assume $k=1$. Replacing
 $\mathbf{v}_{1i_0}$ with a scalar multiple of itself and adding
 adequate multiples to the remaining $\mathbf{v}_{ki_0}$ if necessary
 we may also assume that
\[ \pi_{j_1i_1}\left(\mathbf{v}_{ki_0}\right) =
 \left\{\begin{array}{ll}
v_{j_1i_1}&\mbox{if $k=1$;}\\
\mathbf{0}&\mbox{if $k\neq 1$.}
\end{array}\right.\]

Let $(j_2,i_2)$ be the smallest pair for which
$\pi_{j_2i_2}(\mathbf{v}_{ki_0})\neq\mathbf{0}$ for some $k$, $2\leq
k\leq r$. Again we may assume $k=2$, and that
\[\pi_{j_2i_2}\left(\mathbf{v}_{ki_0}\right) =
\left\{\begin{array}{ll}
v_{j_2i_2}&\mbox{if $k=2$;}\\
\mathbf{0}&\mbox{if $k\neq 2$.}
\end{array}\right.\]

Proceeding in the same way for $k=3,\ldots,r$, we obtain an ordered
list of pairs
$(j_1,i_1)<(j_2,i_2)<\ldots<(j_r,i_r)$ and a basis
$\mathbf{v}_{1i_0},\ldots,\mathbf{v}_{ri_0}$ such that
\[ \pi_{j_{\ell}i_{\ell}}\left(\mathbf{v}_{ki_0}\right) =
  \left\{\begin{array}{ll}
v_{j_{\ell}i_{\ell}}&\mbox{if $\ell=k$,}\\
\mathbf{0}&\mbox{if $\ell\neq k$;}
\end{array}\right.\]
and such that
$\pi_{ba}\left(\mathbf{v}_{ki_0}\right) = \mathbf{0}\quad\mbox{for
  all $(b,a)<(j_k,i_k)$}$.
Write $\displaystyle \mathbf{v}_{ki_0} =\!\!\!\!\! \sum_{i_0<i<j\leq
  n}
\!\!\!\!\!\alpha_{ji}^{(k,i_0)}v_{ji}$.  From the above we have:
\[ \alpha_{ji}^{(k,i_0)} = \left\{ \begin{array}{ll}
1 &\mbox{if $(j,i)=(j_k,i_k)$,}\\
0 &\mbox{if $(j,i)<(j_k,i_k)$.}
\end{array}\right.\]
For $k=1,\ldots,r$ and $i=i_0+1,\ldots,n$, let $\mathbf{v}_{ki}$ be
vectors in~$X$ such that
\[
\left(\mathbf{0},\ldots,\mathbf{0},\mathbf{v}_{ki_0},\mathbf{v}_{ki_0+1},\ldots,\mathbf{v}_{kn}\right)\in
\ker(\Phi)\cap X^n.\]
By Theorem~\ref{th:descriptionkernelPhi}(iii) we have
\begin{equation*}
\Pi_{i_0}\left(\mathbf{v}_{kj}\right) = \sum_{m=i_0+1}^{j-1}
\!\!\!\alpha_{jm}^{(k,i_0)}v_{mi_0} - \sum_{m=j+1}^n
\!\!\!\alpha_{mj}^{(k,i_0)}v_{mi_0}.
\end{equation*}
For simplicity, set $\alpha_{ji}^{(k,i_0)} = -\alpha_{ij}^{(k,i_0)}$,
and $\alpha_{jj}^{(k,i_0)}=0$;
then we can rewrite the above expression as:
\begin{equation}
\Pi_{i_0}\left(\mathbf{v}_{kj}\right) = \sum_{m=i_0+1}^{n}
\!\!\!\alpha_{jm}^{(k,i_0)}v_{mi_0}.
\label{eq:generalformula}
\end{equation}

Let $s$ be the cardinality of the set $\{i_1,j_1,\ldots,i_r,j_r\}$;
that is, $s$ is the number of distinct indices that occur in the list
$(j_1,i_1),\ldots,(j_r,i_r)$. Note that $r\leq \binom{s}{2}$. Let
$a_1<a_2<\cdots < a_s$ be the list of these distinct indices. For each
$\ell$ with $1\leq \ell\leq s$, let $(j_{k(\ell)},i_{k(\ell)})$ be the
smallest pair among $(j_1,i_1),\ldots,(j_r,i_r)$ that has
$a_{\ell}\in\{ i_{k(\ell)}, j_{k(\ell)}\}$. If
$a_{\ell}=i_{k(\ell)}$, let $b_{\ell} = j_{k(\ell)}$; if
$a_{\ell}=j_{k(\ell)}$, let $b_{\ell}=i_{k(\ell)}$. Consider the
following list of vectors from~$X$:
\[\mathbf{v}_{1i_0},\mathbf{v}_{2i_0},\ldots,
\mathbf{v}_{ri_0},\mathbf{v}_{k(1)b_1},
\mathbf{v}_{k(2)b_2},\ldots,\mathbf{v}_{k(s)b_s}.\]
Note that all of these vectors lie in $X\cap\langle v_{ji} \,|\,
i_0\leq i<j\leq n\rangle$.
We will show that these vectors are linearly independent.
Since $\mathbf{v}_{1i_0},\ldots,\mathbf{v}_{ri_0}$ are
linearly independent and $\Pi_{i_0}(\mathbf{v}_{ki_0})=\mathbf{0}$ for
$k=1,\ldots,r$, it suffices to show that
$\Pi_{i_0}(\mathbf{v}_{k(1)b_1}),\ldots,\Pi_{i_0}(\mathbf{v}_{k(s)b_s})$
are linearly independent.

First, from $(\ref{eq:generalformula})$ we have
$ \pi_{a_{\ell}i_0}\left(\mathbf{v}_{k(m)b_m}\right) =
\alpha_{b_ma_{\ell}}^{(k(m),i_0)}$. We claim that if $\ell<m$, then
$\alpha_{b_m a_{\ell}}^{(k(m),i_0)}=0$. By construction, this claim
will follow if we can show that either $a_{\ell}=b_m$, or else the
pair made up of $b_m$ and $a_{\ell}$ is strictly smaller than the pair
made up of $a_m$ and $b_m$ (which is equal to $(j_{k(m)},i_{k(m)})$);
the claim will then follow because $\alpha_{ba}^{(k,i_0)}=0$ whenever
$(b,a)<(j_k,i_k)$.  Indeed, we know that $a_{\ell}<a_m$. If
$a_m=i_{k(m)}$ and $b_m=j_{k(m)}$, then replacing $a_m$ in the pair
$(b_m,a_m)$ with something smaller (namely $a_{\ell}$) gives a smaller
pair: $(b_m,a_{\ell})<(b_m,a_m)$. If, on the other hand, we have
$a_m=j_{k(m)}$ and $b_m=i_{k(m)}$, then if $a_{\ell}>b_m$ we have
$(a_{\ell},b_m)<(a_m,b_m)$, and if $a_{\ell}<b_m$ then we also have
$(b_m,a_{\ell})<(a_m,b_m)$. The only remaining possibility is
$a_{\ell}=b_m$, which is of course no trouble.

Thus, we conclude that
$\alpha_{b_ma_{\ell}}^{k(m),i_0}=0$ whenever $\ell<m$. To
see that the vectors
$\Pi_2(\mathbf{v}_{k(1)b_1}),\ldots,\Pi_2(\mathbf{v}_{k(s)b_s})$ are
linearly independent, note that
\[ \pi_{a_{\ell}i_0}\left(\mathbf{v}_{k(m)b_m}\right) =
\alpha_{b_ma_{\ell}}^{(k(m),i_0)} = \left\{\begin{array}{ll}
\mathbf{0}&\mbox{if $\ell<m$,}\\
v_{b_{\ell}a_{\ell}}&\mbox{if $m=\ell$.}
\end{array}\right.\]
Therefore, if $\beta_1\Pi_{i_0}(\mathbf{v}_{k(1)b_1}) + \cdots +
\beta_s\Pi_{i_0}(\mathbf{v}_{k(s)b_s})=\mathbf{0}$, then $\beta_1=0$
since
the only vector with nontrivial $(a_1,i_0)$-component is
$\Pi_{i_0}(\mathbf{v}_{k(1)b_1})$. Hence $\beta_2=0$, because the only
remaining vector with nontrivial $(a_2,i_0)$-component is
$\Pi_{i_0}(\mathbf{v}_{k(2)b_2})$; and continuing this way we conclude
$\beta_j=0$ for all $j$. So
the vectors are indeed linearly independent. Thus we have established
that
\[
\mathbf{v}_{1i_0},\mathbf{v}_{2i_0},\ldots,
\mathbf{v}_{ri_0},\mathbf{v}_{k(1)b_1}, \mathbf{v}_{k(2)b_2},\ldots,\mathbf{v}_{k(s)b_s}\] 
is a collection of linearly independent vectors in~$X\cap\langle
v_{sr}\,|\,i_0\leq r<s\leq n\rangle$.

Thus we conclude that $d_{i_0}\geq r+s$. Since $r\leq
\binom{s}{2}$, it follows that
\[r\leq \binom{s}{2} \leq \binom{d_{i_0}-r}{2},\]
as claimed.

To complete the proof, it only remains to establish the upper bound on
$d_{i_0+1}$. We have
$d_{i_0} = d_{i_0+1} + \dim\bigl( \langle v_{ji}\,|\, i_0\leq
i<j\leq n\rangle\cap\{\mathbf{v}\in X\,|\, \Pi_{i_0}(\mathbf{v})\neq
\mathbf{0}\}\bigr)$.
Since the vectors $\mathbf{v}_{k(1)b_1},\ldots,\mathbf{v}_{k(s)b_s}$
are linearly
independent, have nontrivial $\Pi_{i_0}$ projection, and lie in
$X \cap \bigl\langle
v_{ji}\,\bigm|\,i_0\leq i<j\leq n\bigr\rangle$,
we have $d_{i_0} \geq d_{i_0+1}+s$. Moreover, since $r\leq
\binom{s}{2}$, we also have $s_{i_0}\leq s$; therefore,
$d_{i_0+1} \leq d_{i_0}-s \leq d_{i_0}-s_{i_0}$,
as desired.
\end{proof}

Note that $Z_{n-1}$ and $Z_n$ are always trivial.

\begin{definition} Let $d$ be a nonnegative integer. We define $r(d)$
  to be the largest integer such that $r(d)\leq d$ and $r(d)\leq
  \binom{d-r(d)}{2}$.
\end{definition}

\begin{theorem}
Fix $n>1$ and let $X<V$. Fix $i_0$, $1\leq i_0\leq n-2$, and let
\[ Z_{i_0} = X\cap \varphi_{i_0}^{-1}\left(\Bigl\langle
\varphi_{i_0+1}(X),\ldots,\varphi_n(X)\Bigr\rangle\right).\]
If $\dim(X\cap \langle v_{ji}\,|\, i_0\leq i<j\leq n\rangle) = d$,
then $\dim(Z_{i_0})\leq r(d)$.
Equivalently,
\begin{equation}
\dim(Z_{i_0})\leq d -
\left\lceil\frac{\sqrt{8d+1}-1}{2}\right\rceil
\label{eq:maxoverlap}
\end{equation}
where $\lceil x\rceil$ is the smallest integer greater than or equal
to $x$.
\label{th:overlaplevelktermsoft}
\end{theorem}

\begin{proof} Let $\dim(Z_{i_0})=r$. By
  Lemma~\ref{lemma:basicoverlapbound}, $r\leq\binom{d-r}{2}$, so
  $r\leq r(d)$, as claimed. From $r(d)\leq \binom{d-r(d)}{2}$ we
  easily obtain $(\ref{eq:maxoverlap})$.
\end{proof}

We have two other ways of describing the function $r(d)$, which will
prove useful below:

\begin{corollary}
Let $d$ be a positive integer. Then $r(d)$ is the number of
nontriangular numbers strictly less than $d$.
 Equivalently, if we write
$d = \binom{t}{2} + s$, with $0< s \leq t$,
then $r(d) = \binom{t-1}{2} + (s-1)$.
\label{cor:overlapleveltriangular}
\end{corollary}

\begin{proof} Since $r(d)\leq \binom{d-r(d)}{2} \leq
  \binom{(d+1)-r(d)}{2}$, it follows that $r(d+1)\geq r(d)$.
We also have
\[ r(d)+2 > r(d)+1 > \binom{d-(r(d)+1)}{2} =
\binom{(d+1)-(r(d)+2)}{2},\]
so $r(d+1)<r(d)+2$.
If $r(d)<\binom{d-r(d)}{2}$, then
\[ r(d)+1\leq \binom{d-r(d)}{2} = \binom{(d+1)-(r(d)+1)}{2},\]
so $r(d+1)\geq r(d)+1$ and in this case we have $r(d+1)=r(d)+1$. If
$r(d)=\binom{d-r(d)}{2}$, then $r(d)+1>\binom{(d+1)-(r(d)+1)}{2}$,
hence $r(d+1)<r(d)+1$ and we conclude that $r(d+1)=r(d)$.  In summary,
we have:
\[ r(d+1) = \left\{\begin{array}{ll}
r(d)+1&\mbox{if $r(d)<\binom{d-r(d)}{2}$,}\\
r(d)&\mbox{if $r(d)=\binom{d-r(d)}{2}$.}
\end{array}\right.\]
We claim that $r(d)=\binom{d-r(d)}{2}$ if and only if $d$ is a
triangular number: when $d=\binom{t+1}{2}$ for some $t\geq 0$, we have
\[\binom{t}{2} = \binom{\binom{t+1}{2}-\binom{t}{2}}{2} =
\binom{d-\binom{t}{2}}{2},\]
so $r(d) = \binom{t}{2} = \binom{d-r(d)}{2}$. Conversely,
if $r(d)=\binom{d-r(d)}{2}$, then solving for $d$ we obtain $d =
\binom{d-r(d)+1}{2}$, proving that $d$ is a triangular number.
Therefore, we have:
\[ r(d+1) = \left\{\begin{array}{ll}
r(d) + 1&\mbox{if $d$ is not a triangular number,}\\
r(d) &\mbox{if $d$ is a triangular number.}
\end{array}\right.\]
Since $r(1)=0$, we conclude that $r(d)$ is the number of nontriangular
numbers strictly smaller than~$d$, as claimed. To establish the
formula, note that the value of $r$ at $\binom{t}{2}$ is
$\binom{t-1}{2}$, and therefore $r\left(\binom{t}{2}+s\right) =
\binom{t-1}{2}+(s-1)$ for $0<s<t$, since there are exactly $s-1$ more
nontriangular numbers strictly less than $\binom{t}{2}+s$ than there
are strictly less than $\binom{t}{2}$. And
$\binom{t}{2}+t=\binom{t+1}{2}$, so we also get equality when $s=t$.
\end{proof}

\begin{remark} These alternate descriptions can also be obtained by
  examining sequence A083920 in~\cite{onlineintseq}; for example,
  compare the closed formula there with $(\ref{eq:maxoverlap})$. I
  first realized these alternate descriptions hold by calculating the
  first few values of $r(d)$ directly, and then
  consulting~\cite{onlineintseq}.
\end{remark}

We can now obtain an upper bound for $\sum\dim(Z_k)$ in terms of
$\dim(X)$, which in turn gives a lower bound for $\dim(X^*)$ in terms
of~$\dim(X)$.

\begin{definition} For $n>0$ and integer $m$, $0\leq m\leq \binom{n}{2}$, we let $f(m,n)$ denote the largest
  possible value of $\sum\dim(Z_k)$ for a subspace $X$ of $V$ with
  $\dim(X)=m$; equivalently,
\[ f(m,n) = \max\Bigl\{ \dim\bigl(X^n\cap \ker(\Phi_n)\bigr)\,\Bigm|\,
  X<V(n),\
  \dim(X)=m\Bigr\}.\]
\end{definition}

\begin{remark} As we will see below, the value of $f(m,n)$ does not
  depend on $n$; meaning that if $m\leq \binom{n}{2}$ and $n\leq N$, then
  $f(m,n)=f(m,N)$. It is easy to verify that $f(m,n)\leq f(m,N)$: if
  $X$ is a subspace of $V(n)$ of dimension~$m$, we can also consider
  it as a subspace of $V(N)$. If the dimension of $X^*$ with respect
  to $\{\varphi_i\}_{i=1}^{n}$ is $nm-r$, then the dimension of $X^*$
  with respect to $\{\varphi_i\}_{i=1}^{N}$ is $Nm-r$; so we have 
\[\dim(X^n\cap {\rm ker}(\Phi_n)) = \dim(X^N\cap {\rm ker}(\Phi_N)).\]
Intuitively, the reason the reverse
inequality also holds is that the largest value of $f(m,n)$ occurs
when the vectors in $Z_i$ use fewer indices rather than more. Because
more indices means a larger value of $s$ in the proof of
Lemma~\ref{lemma:basicoverlapbound}, 
which means more vectors are ``taken out of circulation'' for
$Z_{i+1}$, which gives a smaller possible value for $X\cap\langle
v_{rs}\,|\, i<r<s<n\rangle$. So the ``best'' strategy for larger
intersection with $\ker(\Phi_n)$ is to keep $X$ confined to as small a
number of indices as possible.  The proof
below will formalize this intuition, and show that indeed the value of
$f$ depends only on~$m$.
\label{rem:largerndoesnotmatter}
\end{remark}

\begin{theorem} Let $m>0$, and write $m = \binom{T}{2}+s$, $0\leq
  s\leq T$. If $m\leq\binom{n}{2}$, then
\[ f(m,n) =\binom{T}{3} + \binom{s}{2}.\]
\label{th:maxtotaloverlappossible}
\end{theorem}

\begin{remark} Although there is some ambiguity in the expression
  for~$m$, since $\binom{T}{2} + T = \binom{T+1}{2}$, note that the
  values $\binom{T}{3}+\binom{T}{2}$ and $\binom{T+1}{3}+\binom{0}{2}$
  are equal, so the given value of $f(m)$ is well-defined.
\end{remark}

\begin{proof} By replacing $\binom{T+1}{2}$ with $\binom{T}{2}+T$
  if necessary, we may assume $s>0$. Note that we must have $T<n$ in
  this situation.  First we show that $f(m,n)\geq
  \binom{T}{3}+\binom{s}{2}$.

Let $X$ be the $m$-dimensional coordinate subspace of $V(n)$ generated
by all $v_{ji}$ with $1\leq i<j\leq T$, and the vectors
$v_{T+1,1},\ldots,v_{T+1,s}$. Then $X^*$ is the coordinate subspace of
$W(n)$ generated by all vectors of the form $w_{jik}$ with $1\leq
i<j\leq T$, $i\leq k\leq n$; plus the vectors of the form
$w_{T+1,i,k}$ with $1\leq i\leq s$, $i\leq k\leq n$. There are
$2\binom{T+1}{2}+(n-T)\binom{n}{2}$ vectors of the first kind, and 
\[ n + (n-1) + (n-2) + \cdots + n-(s-1) = sn - \binom{s}{2}\]
of the second kind. Thus $\dim(X^*) = 2\binom{T+1}{2} +
(n-T)\binom{T}{2}+sn - \binom{s}{2}$;  and we have:
\begin{eqnarray*}
n\dim(X) - \dim(X^*)& = & T\binom{T}{2} - 2\binom{T+1}{3} + \binom{s}{2}\\
& = & (T-2)\binom{T}{2} - 2\binom{T}{3} + \binom{s}{2}\\
& = & \binom{T}{3} + \binom{s}{2}.
\end{eqnarray*}
Therefore, $f(m,n)\geq \binom{T}{3}+\binom{s}{2}$. 

For the reverse inequality, we will apply induction. Assume the
for any $X'$ space of $V(n)$ with $\dim(X')<m$. 
Write $m=\binom{T}{2}+s$ with $0<s\leq T$, and $T<n$, and let
$X$ be a subspace of $V$ of dimension~$m$.
We want to show that $\sum
 \dim(Z_i)$ is bounded above by $\binom{T}{3}+\binom{s}{2}$. If all
 $Z_i$ are trivial, this follows. Otherwise, assume $i$ is the
 smallest index with nontrivial $Z_i$, and that $\dim(Z_i)=k>0$. Then
 $k\leq r(m)$, and if $\ell$ is the smallest positive integer such
 that $k\leq
 \binom{\ell}{2}$ then
\[ \dim\bigl(X \cap \langle v_{sr}\,|\, i<r<s\leq n\rangle\bigr)\leq
 m-\ell.\]
So the sum of the dimensions of the $Z_j$ with $j>i$ is at most
$f(m-\ell,n)$; that is, the sum over all $k$ is bounded:
\[\sum\dim(Z_k) \leq k+ f(m-\ell,n).\]
We want to show that $k+f(m-\ell,n) \leq \binom{T}{3}+\binom{s}{2}$ for
all $k$ and $\ell$ that satisfy the relevant conditions.
It is easy to show that for $m=1,2,3,4$, and~$5$, all values of the
form $k+f(m-\ell,n)$, $k\leq r(m)$ and $\ell$ as above are less than or
equal to $\binom{T}{3}+\binom{s}{2}$. 

If $\ell = T = m-r(m)$, then since $k\leq r(m)$ we have
\begin{eqnarray*}
k+f(m-\ell,n) &\leq &r(m) + f(r(m),n)\\
 &=& \binom{T-1}{2} + (s-1) +
f\left(\binom{T-1}{2}+(s-1),n\right)\\
&=& \binom{T-1}{2} + (s-1) + \binom{T-1}{3} + \binom{s-1}{2}\\
& = & \binom{T}{3} + \binom{s}{2};
\end{eqnarray*}
If $\ell<T$, since $k\leq \binom{\ell}{2}$, it is enough to 
to show that for $1<\ell<T$,
\[ \binom{\ell}{2} + f(m-\ell,n) \leq \binom{T}{3} + \binom{s}{2}.\]
If $2\leq \ell \leq s$, then:
\begin{eqnarray*}
\binom{\ell}{2} + f(m-\ell,n) & = & \binom{\ell}{2} +
f\left(\binom{T}{2} + (s-\ell),n\right)\\
& = & \binom{\ell}{2} + \binom{T}{3} + \binom{s-\ell}{2}\\
&\leq&\binom{T}{3} + \binom{s}{2}.
\end{eqnarray*}
The last inequality follows since $\binom{\ell}{2}+\binom{s-\ell}{2}$
is the number of two element subsets of $\{1,\ldots,s\}$, where either
both elements are less than or equal to $\ell$, or both strictly
larger than $\ell$.

If $s<\ell<T$, then write $\ell = s+a$, $a>0$. We then have
\[ m- \ell = \binom{T}{2} + s - (s+a) = \binom{T-1}{2} + (T-1-a),\]
so
\[ \binom{\ell}{2} + f(m-\ell,n) = \binom{\ell}{2} + \binom{T-1}{3} +
\binom{T-1-a}{2}.\]
Since $\ell+1-T\leq 0$ and $a>0$, we must have
\[ 6a(s+a+1-T)\leq 0.\]
Rewriting and introducing suitable terms we have:
\[ 6as +3a^2 - 3a - 3T^2 + 9T - 6 + 3T^2 - 9T - 6aT + 9a + 3a^2 +
6\leq 0\]
In turn, this can be rewritten as
\[ 6as + 3a^2 - 3a - 3(T-1)(T-2) + 3(T-a-1)(T-a-2)\leq 0.\]
This gives:
\[ 3(s^2 + 2as + a^2 - s - a) - 3(T-1)(T-2) + 3(T-a-1)(T-a-2)\leq
3(s^2 - s),\]
and so
\[ 3((s+a)^2 - (s+a)) -3(T-1)(T-2) + 3(T-a-1)(T-a-2) \leq 3(s^2-s).\]
Substituting $\ell$ for $s+a$ and adding $T(T-1)(T-2)$ to both sides
we have
\[ 3(\ell^2-\ell) + (T-3)(T-2)(T-1) + 3(T-a-1)(T-a-2) \leq T(T-1)(T-2)
+ 3(s^2-s);\]
dividing through by $6$ yields the desired inequality:
\[\binom{\ell}{2}+f(m-\ell,n) \leq \binom{\ell}{2} + \binom{T-1}{3} + \binom{T-1-a}{2} \leq
\binom{T}{3} + \binom{s}{2}.
\]

We therefore conclude that $f(m,n)\leq \binom{T}{3}+\binom{s}{2}$, which
completes the proof. Note that indeed, the value of $n$ is not
relevant to the value of $f(m,n)$, so long as $n$ is large enough to
satisfy $m\leq \binom{n}{2}$.
\end{proof}

Since the value of $f(m,n)$ does not depend on $n$, we will drop the
second argument and simly call this function $f(m)$.

\begin{theorem}
Fix $n>1$ and let $X$ be a subspace of $V$. Write $\dim(X) =
\binom{T}{2}+s$, $0\leq s\leq T$. Then
\[ n\dim(X) - \binom{T}{3} - \binom{s}{2} \leq \dim(X^*) \leq
\min\left\{n\dim(X),\ 2\binom{n+1}{3}\right\}.\]
\label{th:upperandlowerbounds}
\end{theorem}

\begin{proof} The lower bound follows from
$\dim(X^*)\geq n\dim(X) - f(\dim(X))$, and the upper  bound is
  immediate.
\end{proof}

\begin{corollary} Fix $n>1$ and let $X$ be a subspace of $V$ with
  $\dim(X)=m$. If $\dim(X^*) = nm - k$ and $n+k > f(m+1)$, then $X$ is
  closed.
\label{cor:Xpreciselysmallenough}
\end{corollary}

\begin{proof} Suppose $X$ is as in the statement, and let $Y$ be any
  subspace of~$V$ of dimension $m+1$. From the definition
  of $f$ we know that
\[ \dim(Y^*) \geq n(m+1) - f(m+1),\]
so $\dim(Y^*)-\dim(X^*) \geq n+k - f(m+1)>0$.
Therefore every $Y$ strictly larger than $X$ must have
$\dim(X^*)<\dim(Y^*)$, which shows that $X$ is closed by
Proposition~\ref{prop:strictlylarger}.
\end{proof}

\begin{corollary} Fix $n>1$ and let $X$ be a subspace of $V$ with
  $\dim(X)=m$. Write $m = \binom{T}{2} + s$, $0\leq s< T$. If
  $\binom{T}{3}+\binom{s+1}{2}<n$, then $X$ is closed.
\label{cor:Xsmallenough}
\end{corollary}

\begin{proof} This follows from the previous corollary and the formula
  for $f(m+1)$ in Theorem~\ref{th:maxtotaloverlappossible}.
\end{proof}

For reference, Table~\ref{table:valuesoff} contains the values of
$f(m)$, $3\leq m \leq 50$. Note that $f(1)=f(2)=0$ by
Corollary~\ref{cor:maincountingarg}. The sequence of values of~$f(m)$
appears as sequence A111138 in~\cite{onlineintseq}. 

\begin{table}[t]
\begin{tabular}{||c|c|||c|c|||c|c||}
\hline
$\quad m\quad$&$\quad f(m)\quad$&$\quad m\quad$&$\quad
f(m)\quad$&$\quad m\quad$&$\quad f(m)\quad$\\
\hline\hline
$3$&$1$&$19$&$26$&$35$&$77$\\
\hline
$4$&$1$&$20$&$30$&$36$&$84$\\
\hline
$5$&$2$&$21$&$35$&$37$&$84$\\
\hline
$6$&$4$&$22$&$35$&$38$&$85$\\
\hline
$7$&$4$&$23$&$36$&$39$&$87$ \\
\hline
$8$&$5$&$24$&$38$&$40$&$90$ \\
\hline
$9$&$7$&$25$&$41$&$41$&$94$\\
\hline
$10$&$10$&$26$&$45$&$42$&$99$\\
\hline
$11$&$10$&$27$&$50$&$43$&$105$\\
\hline
$12$&$11$&$28$&$56$&$44$&$112$\\
\hline
$13$&$13$&$29$&$56$&$45$&$120$\\
\hline
$14$&$16$&$30$&$57$&$46$&$120$\\
\hline
$15$&$20$&$31$&$59$&$47$&$121$\\
\hline
$16$&$20$&$32$&$62$&$48$&$123$\\
\hline
$17$&$21$&$33$&$66$&$49$&$126$\\
\hline
$18$&$23$&$34$&$71$&$50$&$130$\\
\hline
\end{tabular}
\caption{Explicit values of $f(m)$, $3\leq m\leq 50$}
\label{table:valuesoff}
\end{table}

Translating back into group theory, we obtain the following:

\begin{theorem} Let $G$ be a group of class at most two and exponent
  $p$, where $p$ is an odd prime. Let ${\rm rank}(G^{\rm ab}) = n$,
  and let ${\rm
  rank}([G,G])=m$. If $f\bigl(\binom{n}{2}-m+1\bigr)<n$, where $f(k)$
  is the
  function in Theorem~\ref{th:maxtotaloverlappossible}, then $G$ is
  capable.
\label{th:largeenoughcomm}
\end{theorem}

\begin{proof} The subspace $X$ of $V(n)$ corresponding to $G$ has
  dimension $\binom{n}{2}-m$; so the result follows directly from
  Corollary~\ref{cor:Xsmallenough}.
\end{proof}

\section{A Necessary Condition.}\label{sec:neccond}

In this section, we use our set-up to give a proof of a slight
strengthening of the necessary condition proven by Heineken and
Nikolova in~\cite{heinnikolova}.  The proof is essentially that given
in ~\cite{heinnikolova} ``translated'' into our notation and
set-up. We do gain two improvements on their result: a necessary
condition for equality to hold, and a weakening of the hypothesis by
dropping an assumption. In~\cite{heinnikolova} the authors assume
throughout that the capable group $G$ they investigate satisfies the
condition $Z(G)=[G,G]$, and so their result is restricted to that
situation. We will be able to obtain their result with this assumption
dropped.

The object of this section is to prove that if $G$ is capable, of
class at most two and exponent~$p$, and $[G,G]$ is of order $p^k$, then
$G/Z(G)$ is of order at most $p^{2k+\binom{k}{2}}$. 

It is interesting to note that while the results from the previous
sections, leading to sufficient conditions, have focused on the
closure operator on the subspaces of~$V$, the proof here will proceed
by placing considerable emphasis on the interior operator on the
subspaces of~$W$. I do not know if this is simple happenstance, or if
we can indeed expect that considerations of the interior operator
on~$W$ will generally point towards necessary conditions while the
closure operator on~$V$ will give sufficient ones.

In addition to the linear transformations $\varphi_{\mathbf u}$,
an important role in the proof is 
played by elements $g\in Z(G)$ which have nontrivial image in $G^{\rm
ab}$. In order to account for these elements in our setting, we will
use another family of linear transformations which we introduce now:

\begin{definition} Let $n>1$. We embed $U$ into $\mathcal{L}(U,V)$
  as follows:
  given $\mathbf{u}\in U$, we define
  $\psi_{\mathbf{u}}(\mathbf{a})=\mathbf{a}\wedge \mathbf{u}$ for all
  $\mathbf{a}\in U$. If $u_1,\ldots,u_n$ is a given basis for $U$ and
  $i$ is an integer, $1\leq i\leq n$, then we let $\psi_i$ denote the
  transformation $\psi_{u_i}$. Note that for any
  $\mathbf{a},\mathbf{b}\in U$, $\psi_{\mathbf{a}}(\mathbf{b}) =
  -\psi_{\mathbf{b}}(\mathbf{a})$. 
\end{definition}

Fix an isomorphism between $G^{\rm ab}$ and~$U$. Let $g\in G$ be an
element whose image in $G^{\rm ab}$ is nontrivial, and let
$\mathbf{u}_g$ be the corresponding element of~$U$.
Then $g\in Z(G)$ if and only if 
$\psi_{\mathbf{u}_g}(U)$ is contained in $X$.
Note also that for any $\mathbf{u}\in
U$, $\psi_{\mathbf{u}}(U)=\langle \mathbf{u}\rangle^* = {\rm
span}(\psi_1(\mathbf{u}),\ldots,\psi_n(\mathbf{u}))$. This is how we
will use the maps above to address the central elements of~$G$ that
are not in~$[G,G]$

An explicit description of the maps $\psi$ in
terms of a basis for $U$ is easy:

\begin{lemma}
Fix $n>1$, let $u_1,\ldots,u_n$ be a basis for $U$, and let
$v_{ji}$ be the corresponding basis for $V$.
For all integers $i$ and~$j$, $1\leq i,j\leq n$, the image of $u_j$
under $\psi_i$ in terms of the basis $v_{ji}$ is given by:
\[ \psi_i(u_j) = \left\{\begin{array}{ll}
v_{ji}&\mbox{if $i<j$,}\\
\mathbf{0}&\mbox{if $i=j$,}\\
-v_{ij} & \mbox{if $i>j$.}
\end{array}\right.\]
\label{lemma:explicitformulaforpsi}
\end{lemma}

Let $Y$ be a subspace of~$W$. If we let $X=Y^{*}$ then $X$ is closed
by Theorem~\ref{th:interiorop}; moreover, any closed subspace $X$ of~$V$ can
be realized this way, by letting $Y=X^{*}$. Given such an $X$ and $Y$,
we define two subsets of~$U$ as follows:
\begin{eqnarray*}
Z & = & \bigl\{ \mathbf{u} \in U\,\bigm|\,
\psi_{\mathbf{u}}(U)\subseteq X\bigr\},\\
C & = & \bigl\{ \mathbf{u} \in U\,\bigm|\,
\varphi_{\mathbf{u}}(V)\subseteq Y\bigr\}.
\end{eqnarray*}
Let $F$ be the $3$-nilpotent product of $n$ cyclic groups of
order~$p$, and let $U$, $V$, $W$ correspond to $F^{\rm ab}$, $\langle
[x_j,x_i]\,|\, 1\leq i<j\leq n\rangle$,
and~$F_3$ respectively, as in Section~\ref{sec:linalg}, Let $N$ correspond to~$X$, and
$G=F/(X\oplus F_3)$, $H=F/Y$. Then $G$ is capable (since $X$ is
closed), $H$ is a witness for the capability of~$G$, and it is not
hard to see that $Z$ will correspond to the image of $Z_2(H)$ in
$H^{\rm ab}$ (this is the same as the image of $Z(G)$ in $G^{\rm ab}$,
i.e., those elements that are central in $G$ but do not come from
commutators), while $C$ will correspond to the image of the
centralizer $C([H,H])$ in $H^{\rm ab}$. These two sets (in fact,
subspaces as we will prove below) play a key role in our analysis.

\begin{lemma}
Let $Y$ be a subspace of~$W$, and let $X=Y^{*}$. If
\[Z =  \bigl\{ \mathbf{u}\in U\,\bigm|\,
\psi_{\mathbf{u}}(U)\subseteq X\bigr\}\qquad\mbox{and}\qquad 
C  =  \bigl\{ \mathbf{u}\in U\,\bigm|\,
\varphi_{\mathbf{u}}(V)\subseteq Y\bigr\},\]
then both $Z$ and $C$ are subspaces of~$U$, and $Z\subseteq C$.
\label{lemma:ZandCsubspaces}
\end{lemma}

\begin{proof}
The map $\mathbf{u}\mapsto\psi_{\mathbf{u}}$ is a linear embedding from
$U$ to $\mathcal{L}(U,V)$. The canonical projection $V\to V/X$ induces a map
$U\to\mathcal{L}(U,V/X)$. The kernel of this map is $Z$, so~$Z$ is a subspace.

Similarly, the kernel of the composite map
$U\to\mathcal{L}(V,W)\to\mathcal{L}(V,W/Y)$, given by composing the
embedding $\mathbf{u}\mapsto\varphi_{\mathbf{u}}$ with the map induced
by the canonical projection $W\to W/Y$ has kernel $C$, so $C$ is a
subspace.

To prove that $Z\subseteq C$, let $\mathbf{z}\in Z$. If
$\mathbf{z}=\mathbf{0}$, then trivially $\mathbf{z}\in C$. If
$\mathbf{z}\neq \mathbf{0}$, then complete it to a basis
$\mathbf{z}=u_1,\ldots,u_n$ of $U$, and let $v_{ji}$, $w_{jik}$ be the
corresponding prefered bases for $V$ and~$W$. Since $\mathbf{z}=u_1\in
Z$, it follows that $v_{j1}\in X$ for $j=2,\ldots,n$. We want to show
that $\varphi_1(v_{ji})\in Y$ for all
$i,j$, $1\leq i<j\leq n$. If $i=1$, then $v_{ji}\in X$, so
$\varphi_{\mathbf{u}}(v_{ji})\in Y$ for all $\mathbf{u}\in U$ and
there is nothing to do. If $i>1$, then
$\varphi_1(v_{ji}) = w_{j1i} - w_{i1j} =
\varphi_i(v_{j1})-\varphi_j(v_{i1})$.
Since $v_{j1},v_{i1}\in X$, we have that both $\varphi_i(v_{j1})$ and
$\varphi_j(v_{i1})$ lie in~$Y$, hence $\varphi_1(v_{ji})\in Y$. This
proves that $\mathbf{z}\in C$, as claimed. 
\end{proof}

We continue by stating some results on the interactions between the maps
$\psi_{\mathbf{u}}$ and the maps $\varphi_{\mathbf{u'}}$.

\begin{lemma}
For $\mathbf{a}$, $\mathbf{b}$, $\mathbf{c}\in U$, 
$\varphi_{\mathbf{a}}(\psi_{\mathbf{b}}(\mathbf{c})) +
      \varphi_{\mathbf{c}}(\psi_{\mathbf{a}}(\mathbf{b})) +
      \varphi_{\mathbf{b}}(\psi_{\mathbf{c}}(\mathbf{a})) =
      \mathbf{0}$.
\label{lemma:jacobi}
\end{lemma}

\begin{proof} This is simply the Jacobi identity. Evaluating the left
      hand side, we obtain
\[ \overline{(\mathbf{c}\wedge \mathbf{b})\otimes \mathbf{a}} +
\overline{(\mathbf{b}\wedge\mathbf{a})\otimes \mathbf{c}} +
\overline{(\mathbf{a}\wedge\mathbf{c})\otimes\mathbf{b}},\]
where $\overline{(r\wedge s)\otimes t}$ represents the image of
$r\wedge s\otimes t$ in $(V\otimes U)/J$ (see
Definition~\ref{defn:defofUVW}). But since this element is one of the
generators of the subspace $J$, the left hand side is trivial in~$W$,
as claimed.
\end{proof}

\begin{lemma}
Let $Y$ be a subspace of~$W$, and let $X=Y^*$. Let 
\[ C  = \bigl\{\mathbf{u}\in U\,\bigm|\,
\varphi_{\mathbf{u}}(V)\subseteq Y\bigr\}.\]
If $\mathbf{c}\in C$, then for all $\mathbf{a},\mathbf{b}\in U$,
\[ \varphi_{\mathbf{b}}(\psi_\mathbf{c}(\mathbf{a})) \equiv
\varphi_{\mathbf{a}}(\psi_{\mathbf{c}}(\mathbf{b})) \pmod{Y}.\]
\label{lemma:swapwithCinbetween}
\end{lemma}

\begin{proof} From Lemma~\ref{lemma:jacobi}, we know that
$\varphi_{\mathbf{a}}(\psi_{\mathbf{b}}(\mathbf{c})) +
  \varphi_{\mathbf{b}}(\psi_{\mathbf{c}}(\mathbf{a})) =
  -\varphi_{\mathbf{c}}(\psi_{\mathbf{a}}(\mathbf{b}))$. Since
  $\mathbf{c}\in C$, we must have
  $\varphi_{\mathbf{c}}(\psi_{\mathbf{a}}(\mathbf{b}))\in
  Y$. Therefore,
\begin{alignat*}{2}
\varphi_{\mathbf{b}}(\psi_{\mathbf{c}}(\mathbf{a})) &\equiv
-\varphi_{\mathbf{a}}(\psi_{\mathbf{b}}(\mathbf{c})) &&\pmod{Y}\\
&\equiv
\varphi_{\mathbf{a}}(-\psi_{\mathbf{b}}(\mathbf{c})) &&\pmod{Y}\\
&\equiv \varphi_{\mathbf{a}}(\psi_{\mathbf{c}}(\mathbf{b}))
&&\pmod{Y}.
\end{alignat*}
This proves the lemma. 
\end{proof}

\begin{lemma}
Let $Y$ be a subspace of $W$, $X=Y^{*}$, and let
\[ C = \bigl\{ \mathbf{u}\in U\,\bigm|\,
\varphi_{\mathbf{u}}(V)\subseteq Y\bigr\}.\]
If $\codim_W(Y)=1$, then $\codim_V(X) = \codim_U(C)$.
\label{lemma:equalcodims}
\end{lemma}

\begin{proof}
The map $C\to \mathcal{L}(V,W/Y)$ factors through
$\mathcal{L}(V/X,W/Y)$, and the kernel of the induced map
$U\to\mathcal{L}(V/X,W/Y)$ is~$C$. Therefore,
\begin{eqnarray*}
\codim_U(C) = \dim(U/C) &\leq& \dim(\mathcal{L}(V/X,W/Y)) =
\dim(V/X)\dim(W/Y)\\
& = &\codim_U(X)\codim_W(Y) = \codim_U(X),
\end{eqnarray*}
proving one inequality. 

To prove the reverse inequality, let $\mathbf{w}\in W\backslash Y$,
$\codim_V(X)=k$, and pick elements
$\mathbf{v}_1,\ldots,\mathbf{v}_k$ of~$V$ whose images in
the quotient $V/X$ form a basis. Since $\mathbf{v}_1\notin X$, there
exists $\mathbf{u}_1\in U$ such that
$\varphi_{\mathbf{u}_1}(\mathbf{v}_1)\notin Y$. Note that
$\mathbf{u}_1\notin C$. Adjusting $\mathbf{v}_1$ by a scalar if
necessary, and adding multiples of $\mathbf{v}_1$ to
$\mathbf{v}_2,\ldots,\mathbf{v}_k$ if necessary, we may assume that
\[ \varphi_{\mathbf{u}_1}(\mathbf{v}_1)\equiv
\mathbf{w}\pmod{Y}\qquad\mbox{and}\qquad
\varphi_{\mathbf{u}_1}(\mathbf{v}_i)\equiv 0\pmod{Y}\quad\mbox{for
  $i>i$.}\]
Since $\mathbf{v}_2\notin X$, there exists $\mathbf{u}_2\in U$ such
that $\psi_{\mathbf{u}_2}(\mathbf{v}_2)\notin Y$. Multiplying
$\mathbf{v}_2$ by a scalar and adding multiples of $\mathbf{v}_2$ to
$\mathbf{v}_3,\ldots,\mathbf{v}_k$ if necessary, we may assume that
\[ \varphi_{\mathbf{u}_2}(\mathbf{v}_2) \equiv
\mathbf{w}\pmod{Y}\qquad\mbox{and}\qquad
\varphi_{\mathbf{u}_2}(\mathbf{v}_i)\equiv 0 \pmod{Y}\quad\mbox{for
  $i>2$.}\]
Proceeding in the same manner, we obtain elements
$\mathbf{u}_1,\ldots,\mathbf{u}_k\in U$ such that
$\varphi_{\mathbf{u}_i}(\mathbf{v}_i)\equiv \mathbf{w}\pmod{Y}$ for
$i=1,\ldots,k$, and $\varphi_{\mathbf{u}_i}(\mathbf{v}_j)\in Y$ for
$j>i$. Let $\overline{\varphi_{\mathbf{u}_1}},\ldots,
\overline{\varphi_{\mathbf{u}_k}}$ be the images of
$\mathbf{u}_1,\ldots,\mathbf{u}_k$ in $\mathcal{L}(V,W/Y)$. These
linear transformations are linearly independent, because if
$\alpha_1\overline{\varphi_{\mathbf{u}_1}}+\cdots +
\alpha_k\overline{\varphi_{\mathbf{u}_k}}$ is the zero transformation,
then evaluating at $\mathbf{v}_k$ we deduce that $\alpha_k=0$; then
evaluating at $\mathbf{v}_{k-1}$ we obtain $\alpha_{k-1}=0$; etc. Since the
images of $\mathbf{u}_1,\ldots,\mathbf{u}_k$ are linearly independent
under a map with kernel $C$, it follows that their images in $U/C$ are
also linearly independent, proving that $\codim_U(C)=\dim(U/C)\geq k$. This proves
the reverse inequality, and we are done. 
\end{proof}

From the proof above we also deduce the following technical corollary;
we will use it in argument below:

\begin{corollary}
Let $Y$ be a subspace of $W$, and let $X=Y^*$. Let
\[ C = \bigl\{ \mathbf{u}\in U\,\bigm|\,
\varphi_{\mathbf{u}}(V)\subseteq Y\bigr\},\]
and let $\mathbf{w}\in W\backslash Y$. If $\codim_W(Y)=1$ and
$\codim_V(X)=k$, then there exist
$\mathbf{v}_1,\ldots,\mathbf{v}_k\in V$ and
$\mathbf{u}_1,\ldots,\mathbf{u}_k\in U$ such that:
\begin{itemize}
\item[(i)] The images of $\mathbf{v}_1,\ldots,\mathbf{v}_k$ in $V/X$
  form a basis for $V/X$.
\item[(ii)] The images of $\mathbf{u}_1,\ldots,\mathbf{u}_k$ in $U/C$
  form a basis for $U/C$. In particular, $U=\langle
  \mathbf{u}_1,\ldots,\mathbf{u}_k\rangle\oplus C$. 
\item[(iii)] $\varphi_{\mathbf{u}_i}(\mathbf{v}_i)\equiv
  \mathbf{w}\pmod{Y}$ for $i=1,\ldots,k$.
\item[(iv)] $\varphi_{\mathbf{u}_i}(\mathbf{v}_j)\equiv
  \mathbf{0}\pmod{Y}$ for all $i,j$, $1\leq i<j\leq k$. 
\end{itemize}
\label{cor:baseswhencodimone}
\end{corollary}

\begin{lemma}
Let $Y$ be a subspace of $W$, and let $X=Y^*$. Let
\[
Z  =  \bigl\{ \mathbf{u}\in U\,\bigm|\, \psi_{\mathbf{u}}(U)\subseteq
X\bigr\}\quad\mbox{and}\quad 
C  =  \bigl\{ \mathbf{u}\in U\,\bigm|\,
\varphi_{\mathbf{u}}(V)\subseteq Y\bigr\}.\]
Let $\mathbf{u}_1,\ldots,\mathbf{u}_k\in U$ be elements such that
$U=\langle \mathbf{u}_1,\ldots,\mathbf{u}_k\rangle+C$. If
$\mathbf{c}\in C$ is such that $\psi_{\mathbf{c}}(\mathbf{u}_i)\in X$
for $i=1,\ldots,k$, then $\mathbf{c}\in Z$. 
\label{lemma:fromCtoZ}
\end{lemma}

\begin{proof} To prove that $\mathbf{v}\in V$ lies in $X$, it is
  enough to show that $\varphi_{\mathbf{u}_i}(\mathbf{v})\in Y$ for
  $i=1,\ldots,k$; this follows since $U$ is spanned by the vectors
  $\mathbf{u}_1,\ldots,\mathbf{u}_k$ and~$C$, the latter 
  always mapping into~$Y$, and 
$X = \bigl\{ \mathbf{v}\in V\,\bigm|\,
  \varphi_{\mathbf{u}}(\mathbf{v})\in Y\mbox{\ for all $\mathbf{u}\in
  U$}\bigr\}$. Thus, to prove that $\mathbf{c}\in Z$, it is enough to
  show that for every $\mathbf{u}\in U$ and $i=1,\ldots,k$,
  $\varphi_{\mathbf{u}_k}(\psi_{\mathbf{c}}(\mathbf{u}))\in Y$.
Since $\mathbf{c}\in C$, we know from
  Lemma~\ref{lemma:swapwithCinbetween} that
\[ \varphi_{\mathbf{u}_k}(\psi_{\mathbf{c}}(\mathbf{u}))\equiv
  \varphi_{\mathbf{u}}(\psi_{\mathbf{c}}(\mathbf{u}_k))\pmod{Y}.\]
By assumption, $\psi_{\mathbf{c}}(\mathbf{u}_k)\in X$, and therefore
$\varphi_{\mathbf{u}}(\psi_{\mathbf{c}}(\mathbf{u}_k)) \in Y$. Thus,
$\varphi_{\mathbf{u}_k}(\psi_{\mathbf{c}}(\mathbf{u}))$ lies in $Y$, and we
are done.
\end{proof}

The following counting argument will be needed a few times, and will be
the key tool used to establish the upper bounds.

\begin{lemma}
Let $A$ and $B$ be vector spaces over the same field, $\dim(A)=n$ and
$\dim(B)=1$. Let $\mathbf{a}_1,\ldots,\mathbf{a}_n$ be a basis
for~$A$, and let $\mathbf{b}\in B$ be a nonzero vector. Let
$f_1,\ldots,f_n\in\mathcal{L}(A,B)$ be linear transformations such
that
\[
f_i(\mathbf{a}_i) = \mathbf{b} \quad \mbox{for $i=1,\ldots,n$};
\qquad\mbox{and}\qquad 
f_i(\mathbf{a}_j) = \mathbf{0} \quad \mbox{if $1\leq i<j\leq n$}.
\]
Then the dimension of the subspace
\[ \Bigl\{ \bigl(\mathbf{v}_1,\ldots,\mathbf{v}_n\bigr)\in A^n\,\Bigm|\,
f_i(\mathbf{v}_j)=f_j(\mathbf{v}_i),\quad 1\leq i,j\leq n\Bigr\}\]
is $n+\binom{n}{2}$.
\label{lemma:countingargument}
\end{lemma}

\begin{proof}
Express $\mathbf{v}_i$ in terms of the basis for~$A$, 
$\mathbf{v}_i = \alpha_{i1}\mathbf{a}_1 + \cdots +
\alpha_{in}\mathbf{a}_n$.
We have $n$ degrees of freedom for choosing $\mathbf{v}_1$. Once
$\mathbf{v}_1$ is fixed, we must have
\[ \alpha_{21}\mathbf{b} = f_1(\mathbf{v}_2) = f_2(\mathbf{v}_1),\]
which fixes the value of $\alpha_{21}$, leaving $n-1$ degrees of
freedom for choosing $\mathbf{v}_2$. Once both $\mathbf{v}_1$ and
$\mathbf{v}_2$ are fixed, $\mathbf{v}_3$ must satisfy
\begin{alignat*}{2}
\alpha_{31}\mathbf{b} &= f_1(\mathbf{v}_3) &&=f_3(\mathbf{v}_1),\\
\alpha_{31}f_2(\mathbf{a}_1) + \alpha_{32}\mathbf{b} &=
f_2(\mathbf{v}_3) && = f_3(\mathbf{v}_2).
\end{alignat*}
The first equation completely determines $\alpha_{31}$, which together
with the second equation completely determines $\alpha_{32}$, leaving
$n-2$ degrees of freedom for choosing $\mathbf{v}_3$.

Continuing in this manner we have $n-3$ degrees of freedom for
$\mathbf{v}_4$, $n-4$ for $\mathbf{v}_5$, and so on, until we have one
degree of freedom left for $\mathbf{v}_n$. In total, we have 
\[ n + (n-1) + (n-2) + \cdots + 2 + 1 = n + \binom{n}{2}\]
degrees of freedom in choosing the $n$-tuple; this proves that the
subspace in question has dimension $n + \binom{n}{2}$, as claimed.
\end{proof}

Our proof that we can bound $\dim(U/Z)$ in terms of $\dim(V/X)$ will
proceed by induction on $\codim_W(Y)$. The basis of the induction is
contained in the following lemma:

\begin{lemma}[cf. \cite{heinnikolova}*{Lemma~1}]
Let $Y$ be a subspace of $W$ of codimension one. Let $X=Y^*$, and
let
$Z = \bigl\{ \mathbf{u}\in U\,\bigm|\, \psi_{\mathbf{u}}(U)\subseteq
X\bigr\}$.
If $\codim_V(X) = k$, then $\dim(U/Z)\leq 2k+\binom{k}{2}$. 
\label{lemma:basisforinduction}
\end{lemma}

\begin{proof} As before, let $C=\bigl\{\mathbf{u}\in U\,\bigm|\,
  \varphi_{\mathbf{u}}(V)\subseteq Y\bigr\}$. From
  Lemma~\ref{lemma:equalcodims} we know that $\dim(U/C)=k$, so we only
  need to prove that $\dim(C/Z)\leq k+\binom{k}{2}$.  Fix
  $\mathbf{w}\in W\backslash Y$, and let
  $\mathbf{v}_1,\ldots,\mathbf{v}_k$ and
  $\mathbf{u}_1,\ldots,\mathbf{u}_k$ be the vectors given by
  Corollary~\ref{cor:baseswhencodimone}. 

Consider the linear map $C\mapsto (V/X)^k$ defined by:
\[ \mathbf{c}\mapsto
\left(\overline{\psi_{\mathbf{c}}(\mathbf{u}_1)},\ldots,\overline{\psi_{\mathbf{c}}(\mathbf{u}_k)}\right),\]
where $\overline{\mathbf{v}}$ is the image of $\mathbf{v}\in V$ under
the canonical projection $V\to V/X$. By
Lemma~\ref{lemma:fromCtoZ}, the kernel of the map is $Z$, so we obtain
an embedding of $C/Z$ into $(V/X)^k$. Note also that the image of
$\mathbf{c}\in C$ is a vector
$(\overline{\psi_{\mathbf{c}}(\mathbf{u}_1)},\ldots,\overline{\psi_{\mathbf{c}}(\mathbf{u}_k)})$
that satisfies the congruence
$\varphi_{\mathbf{u}_j}(\psi_{\mathbf{c}}(\mathbf{u}_i)) \equiv
\varphi_{\mathbf{u}_i}(\psi_{\mathbf{c}}(\mathbf{u}_j))\pmod{Y}$ by
Lemma~\ref{lemma:swapwithCinbetween}. This is well defined, since
elements of $X$ always map into~$Y$ via any $\varphi_{\mathbf{u}}$. 

By Lemma~\ref{lemma:countingargument}, the image of~$\mathbf{c}$ lies in a subspace of
dimension $k+\binom{k}{2}$; therefore, $C/Z$ is of dimension at most $k+\binom{k}{2}$,
By Lemma~\ref{lemma:equalcodims}:
\[ \dim(U/Z) = \dim(U/C) + \dim(C/Z) \leq k+ k + \binom{k}{2} = 2k+\binom{k}{2},\]
proving the lemma.
\end{proof}

One final observation is needed:

\begin{lemma}
Let $Y$ be a subspace of $W$, and let $Y'$ be a
subspace of $W$ with the same interior as $Y$; that is, such that ${Y'}^{**}=Y^{**}$. Then 
${Y'}^*=Y^*$.
\label{lemma:onlyinteriormatters}
\end{lemma}

\begin{proof} This follows from Theorem~\ref{th:interiorop}, since
${Y'}^* = {Y'}^{***} = Y^{***}=Y$.
\end{proof}

We can now prove the main result of this section:

\begin{theorem}[cf. \cite{heinnikolova}*{Theorem~1}]
Let $Y$ be a subspace of~$W$, let $X=Y^{*}$, and let
$Z = \bigl\{ \mathbf{u}\in U\,\bigm|\, \psi_{\mathbf{u}}(U)\subseteq
X\bigr\}$.
If $\codim_V(X)=k$, then $\dim(U/Z)\leq 2k+\binom{k}{2}$, and 
equality holds only if there exists a subspace $Y'$ of $W$ such that
$\codim_W(Y')\leq 1$ and $Y'^{**}=Y^{**}$. 
\label{theorem:heinnikneccondition}
\end{theorem}

\begin{proof} We proceed by induction on $r=\codim_W(Y)$. If
  $\codim_W(Y)=0$, then $X=V$, $Z=U$, and the result holds
  trivially. If $\codim_W(Y)=1$, then the result is
  Lemma~\ref{lemma:basisforinduction}; the final clause also holds,
  since the consequent is trivially true with $Y'=Y$. Assume then that
  $\codim_W(Y)\geq 2$.  As before, let
\[ C = \bigl\{\mathbf{u}\in U\,\bigm|\,
  \varphi_{\mathbf{u}}(V)\subseteq Y\bigr\}.\]

Suppose inductively that the result holds for any subspace $Y'$ of $W$
such that $\codim_W(Y')<\codim_W(Y)=r$. If there
exists $Y'$ with $\codim_W(Y')<\codim_W(Y)$ and $Y'^{**}=Y^{**}$, then we
can replace $Y$ with $Y'$. Note that since $Y'^{*}=Y^*$, 
the subspaces $Z$ of $U$ and
$X$ of~$V$ are not affected by change, so the result holds by
induction. We may therefore assume that if $Y_2$ is any subspace of $W$
that properly contains $Y$, then $Y^*$ is properly contained in
$Y_2^*$. To prove the result for $Y$, we will need to establish that the
strict inequality holds in this situation.

Among all $Y_2$ such that $Y\subseteq Y_2$, and $\dim(Y_2)=\dim(Y)+1$,
we pick one for which $X_2 = Y_2^*$ is of minimal dimension. Note that
$X$ is properly contained in $X_2$; if we let $\omega=\dim(X_2/X)$,
then $0<\omega<k$ and $\codim_V(X_2) = k-\omega$. Let
$Z_2 = \bigl\{ \mathbf{u}\in U\,\bigm|\,
\psi_{\mathbf{u}}(U)\subseteq X_2\bigr\}$;
again we have that $Z\subseteq Z_2$. By the induction hypothesis, we know that
$\dim(U/Z_2) \leq 2(k-\omega) + \binom{k-\omega}{2}$. We now want to
estimate the dimension of $Z_2/Z$. We will do this in two steps, first
by giving an upper bound for the dimension of $Z_2/(Z_2\cap C)$, and
then giving an upper bound for the dimension of $(Z_2\cap C)/Z$. 

Let $Y_3$ be any subspace of $W$ that contains $Y$,
$\dim(Y_3)=\dim(Y)+1$, and $Y_3\neq Y_2$. This is possible because
$\dim(Y)<\dim(W)-1$, so there are at least $p+1$ subspaces of dimension
one more than $\dim(Y)$. Let $X_3 = Y_3^*$; by choice of $Y_2$,
$\dim(X_3/X)\geq \omega$, and so $\dim(V/X_3)\leq k-\omega$. 

If $\mathbf{v}\in X_3$ and $\mathbf{u}\in Z_2$, then
$\varphi_{\mathbf{u}}(\mathbf{v})\in Y_2\cap Y_3=Y$; 
so the map $Z_2\mapsto
\mathcal{L}(V,Y_2/Y)$ defined by $\mathbf{u}\mapsto
\varphi_{\mathbf{u}}$ factors through $\mathcal{L}(V/X_3,Y_2/Y)$. The
kernel of this map is $Z_2\cap C$, and therefore
\[\dim(Z_2/(Z_2\cap C)) \leq \dim(\mathcal{L}(V/X_3,Y_2/Y)) =
\dim(V/X_3)\leq k-\omega.\]

Finally, we want an upper bound for $\dim((Z_2\cap
C)/Z)$. This is the difficult part of the induction. 

By Corollary~\ref{cor:baseswhencodimone}, we can select elements
$\mathbf{x}_1,\ldots,\mathbf{x}_{\omega}$ in $X_2$,
and $\mathbf{u}_1,\ldots,\mathbf{u}_{\omega}$ in $U$, such that the
$\mathbf{x}_i$ projecto onto a basis for $X_2/X$, and
\begin{alignat*}{3}
\varphi_{\mathbf{u}_i}(\mathbf{x}_i) &\equiv \mathbf{w} &&\pmod{Y}
&\qquad&1\leq i\leq \omega,\\
\varphi_{\mathbf{u}_i}(\mathbf{x}_j) &\equiv \mathbf{0} &&\pmod{Y}
&\qquad&1\leq i<j\leq \omega.
\end{alignat*}
Since the images of $\mathbf{u}_1,\ldots,\mathbf{u}_{\omega}$ are
linearly independent when we project to $U/C$, if we project to
$U/(Z_2\cap C)$ the images are also linearly independent. Extend this
list to
$\mathbf{u}_1,\ldots,\mathbf{u}_{\omega},\mathbf{u}_{\omega+1},\ldots,\mathbf{u}_s$
such that the projections form a basis for $U/(Z_2\cap C)$.

Fix $j>\omega$. Adding suitable multiples of $\mathbf{u}_{\omega}$ to
$\mathbf{u}_j$, we may assume that
$\varphi_{\mathbf{u}_j}(\mathbf{x}_{\omega})$ is in~$Y$. Then, adding
multiples of $\mathbf{u}_{\omega-1}$ to $\mathbf{u}_j$, which will not
change the value of $\varphi_{\mathbf{u}_j}(\mathbf{x}_{\omega})$
modulo~$Y$, we may also assume that
$\varphi_{\mathbf{u}_j}(\mathbf{x}_{\omega-1})\in Y$. Continuing in
this manner, adding multiples of
$\mathbf{u}_{\omega-2},\ldots,\mathbf{u}_1$, we may assume that
$\varphi_{\mathbf{u}_j}(\mathbf{x}_i)\in Y$ for $i=1,\ldots,\omega$;
repeating this for each $j>\omega$, we obtain:
\[ \varphi_{\mathbf{u}_j}(\mathbf{x}_i)\in Y\quad\mbox{for all $i,j$,
  $1\leq i\leq \omega < j\leq s$.}\]
Since $X_2 = \langle
\mathbf{x}_1,\ldots,\mathbf{x}_{\omega}\rangle+X$, we thus have that
for $j>\omega$, $\varphi_{\mathbf{u}_j}(X_2)\subseteq Y$. 

Given $\mathbf{u}\in Z_2\cap C$, consider
\[\left(\overline{\psi_{\mathbf{u}}(\mathbf{u}_1)},\ldots,
\overline{\psi_{\mathbf{u}}(\mathbf{u}_{\omega})},
\overline{\psi_{\mathbf{u}}(\mathbf{u}_{\omega+1})},\ldots,
\overline{\psi_{\mathbf{u}}(\mathbf{u}_s)}\right)\in (X_2/X)^s.\]
Since $\mathbf{u}\in C$, if $1\leq i<j\leq s$ then by
Lemma~\ref{lemma:swapwithCinbetween} the coordinates satisfy
\[\varphi_{\mathbf{u}_j}(\psi_\mathbf{u}(\mathbf{u}_i)) \equiv
\varphi_{\mathbf{u}_i}(\psi_{\mathbf{u}}(\mathbf{u}_j)) \pmod{Y};\]
the values are well defined modulo $X$, since $X=Y^*$. In addition,
$\mathbf{u}\in Z_2$, and so $\psi_{\mathbf{u}}(\mathbf{u}_i)\in X_2$;
therefore, if $j>\omega$, then
$\varphi_{\mathbf{u}_j}(\psi_{\mathbf{u}}(\mathbf{u}_i))\in
\varphi_{\mathbf{u}_j}(X_2)\subseteq Y$. In particular, if $j>\omega$,
then $\varphi_{\mathbf{u}_j}(\psi_{\mathbf{u}}(\mathbf{u}_i))\equiv 0
\pmod{Y}$ for all~$i$, and therefore
\[ \varphi_{\mathbf{u}_1}(\psi_{\mathbf{u}}(\mathbf{u}_j))\equiv
\cdots \equiv
\varphi_{\mathbf{u}_s}(\psi_{\mathbf{u}}(\mathbf{u}_j))\equiv 0
\pmod{Y}.\]
Since $U=\langle \mathbf{u}_1,\ldots,\mathbf{u}_s\rangle + C$, it
follows that if $j>\omega$ then $\varphi_{\mathbf{a}}(\psi_{\mathbf{u}}(\mathbf{u}_j))\in
Y$ for all $\mathbf{a}\in U$. Therefore
$\psi_{\mathbf{u}}(\mathbf{u}_j)\in X$ for all $j>\omega$. Thus, we
conclude that for all $\mathbf{u}\in Z_2\cap C$ and all $j>\omega$,
$\psi_{\mathbf{u}}(\mathbf{u}_j)\in X$. Therefore, in the $s$-tuple
\[\left(\overline{\psi_{\mathbf{u}}(\mathbf{u}_1)},\ldots,
\overline{\psi_{\mathbf{u}}(\mathbf{u}_{\omega})},
\overline{\psi_{\mathbf{u}}(\mathbf{u}_{\omega+1})},\ldots,
\overline{\psi_{\mathbf{u}}(\mathbf{u}_s)}\right)\in (X_2/X)^s\]
only the first $\omega$ components may be nontrivial. 

Consider then the linear map $Z_2\cap C \longmapsto (X_2/X)^s$ given
by
\[ \mathbf{u} \mapsto \left(\overline{\psi_{\mathbf{u}}(\mathbf{u}_1)},\ldots,
\overline{\psi_{\mathbf{u}}(\mathbf{u}_{\omega})},
\overline{\psi_{\mathbf{u}}(\mathbf{u}_{\omega+1})},\ldots,
\overline{\psi_{\mathbf{u}}(\mathbf{u}_s)}\right).\]
We claim that the kernel of this map is $Z$.

Certainly, $Z$ is contained in the kernel. Conversely, let $\mathbf{u}\in Z_2\cap
C$ be such that $\psi_{\mathbf{u}}(\mathbf{u}_j)\in X$ for
$j=1,\ldots,s$. Since $U=\langle
\mathbf{u}_1,\ldots,\mathbf{u}_s\rangle + (Z_2\cap C)$, to prove that
$\mathbf{u}\in Z$ it is enough to show that $\psi_{\mathbf{u}}(Z_2\cap
C)\subseteq X$. In turn, to establish this it is enough to show that
if $\mathbf{z}\in Z_2\cap C$, then for all $\mathbf{a}\in U$,
$\varphi_{\mathbf{a}}(\psi_{\mathbf{u}}(\mathbf{z}))\in Y$. Since
$\mathbf{u}\in Z_2\cap C\subseteq C$, we know that
$\varphi_{\mathbf{a}}(\psi_{\mathbf{u}}(\mathbf{z}))\equiv
\varphi_{\mathbf{z}}(\psi_{\mathbf{u}}(\mathbf{a}))\pmod{Y}$; and
since $\mathbf{z}\in Z_2\cap C\subseteq C$, we also have that
$\varphi_{\mathbf{z}}(\psi_{\mathbf{u}}(\mathbf{a}))\in Y$. Therefore,
$\varphi_{\mathbf{a}}(\psi_{\mathbf{u}}(\varphi_{\mathbf{z}}))\in Y$
for all $\mathbf{a}\in U$ and we conclude that
$\psi_{\mathbf{u}}(\mathbf{z})\in X$ as desired. Therefore
$\mathbf{u}\in Z$, proving the claim. Putting this claim together
with the observation that only the first $\omega$ components can be
nontrivial in any case, we conclude that we have an embedding 
\begin{eqnarray*}
(Z_2\cap C)/Z &\hookrightarrow & (X_2/X)^{\omega}\\
\mathbf{u} &\longmapsto&
\left(\overline{\psi_{\mathbf{u}}(\mathbf{u}_1)},\ldots,\overline{\psi_{\mathbf{u}}(\mathbf{u}_{\omega})}\right),
\end{eqnarray*}
and that the $\omega$-tuples in the image satisfy
\[ \varphi_{\mathbf{u}_j}(\psi_{\mathbf{u}}(\mathbf{u}_i)) \equiv
\varphi_{\mathbf{u}_i}(\psi_{\mathbf{u}}(\mathbf{u}_j))\pmod{Y}\qquad\mbox{for
  all $i,j$, $1\leq i<j\leq \omega$.}\]
Applying Lemma~\ref{lemma:countingargument}, we have that $(Z_2\cap
  C)/Z$ embeds into a subspace of dimension
  $\omega+\binom{\omega}{2}$, and therefore $\dim((Z_2\cap C)/Z)\leq \omega+\binom{\omega}{2}$.

Thus we have shown that:
\begin{eqnarray*}
\dim(U/Z) & = & \dim(U/Z_2) + \dim(Z_2/(Z_2\cap C)) + \dim((Z_2\cap
C)/Z),\\
\dim(U/Z_2) & \leq & 2(k-\omega) + \binom{k-\omega}{2},\\
\dim(Z_2/(Z_2\cap C)) & \leq & k-\omega,\\
\dim((Z_2\cap C)/Z) & \leq & \omega + \binom{\omega}{2}.
\end{eqnarray*}
Putting it all together, we have
(cf.~Theorem~1 in~\cite{heinnikolova}):
\begin{eqnarray*}
\dim(U/Z) &\leq & 3k - 2\omega +
\binom{k-\omega}{2}+\binom{\omega}{2}\\
& = & 3k - 2\omega + \frac{k^2-k}{2} + \omega(\omega-k)\\
& = & 2k + \binom{k}{2} + (k-\omega)(1-\omega) - \omega\\
& = & 2k + \binom{k}{2} - \Bigl( (\omega-1)(k-\omega-1) + 1\Bigr)\\
& < & 2k + \binom{k}{2}.
\end{eqnarray*}
The last inequality holds since $0<\omega<k$, and therefore both
$\omega-1$ and $k-\omega-1$ are nonnegative. 

This strict inequality finishes the inductive step, and proves the theorem.
\end{proof}

Translating back to groups we obtain the promised improvement on the
necessary condition of Heineken and Nikolova:

\begin{theorem}
Let $G$ be a $p$-group of class at most~$2$ and exponent~$p$. If $G$
is capable, and $[G,G]$ is of rank~$k$, then $G/Z(G)$ is of rank at
most $2k+\binom{k}{2}$. Moreover, equality holds only if there exists
a witness $H$ to capability such that $H_3$ is cyclic (possibly
trivial).
\label{th:strengthenedheinnik}
\end{theorem}

\begin{proof} If we fix an isomorphism $G^{\rm ab}\cong U$ and let $X$
  be the corresponding subspace of~$V$, then $[G,G]\cong V/X$ and
  $G/Z(G)\cong U/Z$, where
$Z = \bigl\{ \mathbf{u}\in U\,\bigm|\, \psi_{\mathbf{u}}(U)\subseteq
  X\bigr\}$.
Thus, the inequality by Theorem~\ref{theorem:heinnikneccondition}. For the
``moreover'' clause, note that if we let $H$ be $F/M$, where $F$ is the
$3$-nilpotent product of $n$ groups of order~$p$ (with $n$ the
rank of $G^{\rm ab}$), and $M$ is the subspace of $H_3$ corresponding
to any $Y$ with $Y^*=X$, then $H$ will be a witness fot the capability
of~$X$ (as we are assuming that $X$ is closed). By picking the $Y'$ of
codimension at most~$1$ guaranteed by the theorem, we obtain a witness
with the desired property. 
\end{proof}

\section{The $5$-generated case.}\label{sec:applications}

In this section we combine our results so far to characterise the
capable groups among the $5$-generated groups of class at most two and
exponent~$p$.  

One way to interpret Corollary~\ref{cor:Xsmallenough} is that if $G$
is of class exactly two and exponent~$p$, and the commutator subgroup
of~$G$ is ``large enough,'' then $G$ will be capable. On the other
hand, Theorem~\ref{th:strengthenedheinnik} says that if $G$ is
capable, of class exactly two and exponent~$p$, then the commutator
subgroup of~$G$ cannot be ``too small''. Put together, the results
seem to indicate that a group of class exactly two and prime exponent
will be capable if and only if it is ``nonabelian enough.''  The
characterisation below seems to reinforce this intuition.

\subsection*{The $4$-generated case.}
It is of course well known that a nontrivial cyclic group cannot be
capable. It has also also been long known that an extra-special
$p$-group is capable if and only if it is of order $p^3$ and
exponent~$p$. The
following result shows that, at least for $4$-generated groups in the
class we are considering, these are the \textit{only} exceptions to
capability.

\begin{theorem}
Let $p$ be a prime, and let $G$ be a $4$-generated group of class at most~$2$ and
exponent~$p$. Then $G$ is one and only one of the following:
\begin{itemize}
\item[(i)] Cyclic and nontrivial;
\item[(ii)] Extra-special of order $p^5$ and exponent~$p$;
\item[(iii)] Capable.
\end{itemize}
\label{th:fourgeneratedcase}
\end{theorem}

\begin{proof} Following the notation of
  Theorem~\ref{th:largeenoughcomm}, let $n$ be the rank of $G^{\rm
    ab}$, and let $m$ be the rank of $[G,G]$. 

The case of $p=2$ is trivial, since $G$ is abelian in this
case. Assume then $p>2$. The three categories are of course disjoint,
so we only need to show that any such~$G$ is one of the three. If $G$
is trivial, then it is capable. If $G$ is minimally $1$-generated,
then it is nontrivial cyclic.

For $G$ minimally $2$-generated, Theorem~\ref{th:largeenoughcomm}
shows that $G$ is capable: we have $n=2$ and $m=0$ or~$1$, and in either
case $f(\binom{2}{2}-m+1)<2$. For $G$ minimally $3$-generated, again
Theorem~\ref{th:largeenoughcomm} settles the problem: here we have
$m=0$, $1$, $2$, or~$3$, and $f(\binom{3}{2}-m+1)<3$ in all cases.

Consider then the case of $G$ minimally $4$-generated; $m$ must
satisfy $0\leq m\leq 6$. If $m\geq 2$, then $f(\binom{4}{2}-m+1)<4$,
so $G$ will be capable. If $m=0$, then $G\cong C_p^4$, which is
capable. Thus, the only case not covered is when $m=1$, i.e., the
commutator subgroup is cyclic. 

If $Z(G)\neq [G,G]$, then we apply Theorem~\ref{th:cancelcentralsummand}
and the $n=3$ case to
deduce that $G$ is capable. Finally, if $Z(G)=[G,G]$ then we apply
Theorem~\ref{th:strengthenedheinnik}: the group cannot be capable,
since $4> 2(1)+\binom{1}{2}$. Alternatively, $G$ is of order
$p^5$, exponent~$p$, and extra-special, and so we apply
Corollary~\ref{cor:extraspecialcase}. 
\end{proof}

\subsection*{The minimally $5$-generated case.}

We next consider the case of $n=5$. Here,
Theorem~\ref{th:largeenoughcomm} settles the cases $m\geq 4$; 
and the case $m=0$ is of course trivial. We can finish the characterisation
applying some easy group theory, and finally by applying non-trivial
work of Brahana~\cite{brahanalines} to obtain a very satisfying result
similar to Theorem~\ref{th:fourgeneratedcase}.

If $m=1$ then our group $G$ has cyclic commutator subgroup. We
cannot then have $Z(G)=[G,G]$ since $G^{\rm ab}$ is of order $p^5$,
and so the group $G$ will be either of the form $E\oplus C_p$, where
$E$ is extra-special of order $p^5$ and exponent~$p$ (hence $G$ is not
capable), or else of the form $K\oplus C_p^3$ where $K$ is the
nonabelian (extra-special) group of order $p^3$ and exponent~$p$ (and
so $G$ will be capable). 

To discuss the cases of $m=2$ and~$m=3$, recall that if $V$ is a
vector space and $k$ is an integer, $0\leq k\leq \dim(V)$, then the
Grassmannian $Gr(k,V)$ is the set of all $k$-dimensional subspaces
of~$V$. This set has a rich geometric structure, though we will only
touch on it briefly.

To solve the cases of $m=2$ and~$m=3$, by
Proposition~\ref{prop:symmetry} we only need to consider one
representative from each orbit of the action of ${\rm GL}(5,p)$ in
$Gr(7,V)$ (for the case $m=3$) and $Gr(8,V)$ (for the case
$m=2$). In~\cite{brahanalines}, Brahana shows that there are $6$
orbits in $Gr(2,V)$ and $22$ orbits in $Gr(3,V)$. By taking the
orthogonal complement of each subspace (relative to our prefered basis
$v_{ji}$, $1\leq i<j\leq n$, with $\langle v_{ji},v_{rs}\rangle = 1$
if $(j,i)=(r,s)$ and $0$ otherwise) we obtain a well-known duality
that shows that the number of orbits in $Gr(k,V)$ is the same as the
number of orbits in $Gr(\binom{n}{2}-k,V)$ (see for example the
paragraphs leading to \cite{bush}*{Theorem 1}; the argument there is
for $n=4$ and $k=6$, but it trivially generalizes); thus, we can take
the lists from \cite{brahanalines} and by taking orthogonal complements,
obtain a complete list of orbit representatives for the cases we are
interested in. It is then an easy matter to check which ones
correspond to closed subspaces and which do not.

There are six orbits of $8$-dimensional subspaces under the action
of~${\rm GL}(5,p)$: we give representatives of the orbits as orthogonal complements
to the representatives found under the heading \textit{``the lines
of~$S$''} in~\cite{brahanalines}*{p.~547}:
\begin{itemize}
\item [1.] The coordinate subspace $X_1=\langle v_{41}, v_{51}, v_{32},
  v_{42}, v_{52}, v_{43}, v_{53}, v_{54}\rangle$; this is closed
  by Theorem~\ref{thm:coordclosed}. Alternatively, note that $u_5$ is central
  in the corresponding~$G$, so we can apply
  Corollary~\ref{cor:generalcancelcentralsummand} to reduce 
  to the $n=4$, $\dim(X)=4$ case.

\item[2.] The coordinate subspace $X_2=\langle v_{31}, v_{41}, v_{51},
  v_{32}, v_{42}, v_{52}, v_{53}, v_{54}\rangle$; again, this is
  closed either by appplying Theorem~\ref{thm:coordclosed} or
  Corollary~\ref{cor:generalcancelcentralsummand}. 

\item [3.] The subspace $X_3 = \langle v_{21}-v_{43}, v_{31}, v_{41},
  v_{51}, v_{42}, v_{52}, v_{53}, v_{54}\rangle$.  Again, note that
  $\psi_5(U)$ is contained in $X_3$, so by
  Corollary~\ref{cor:generalcancelcentralsummand} we conclude that
  $X_3$ is closed.

\item[4.] The subspace $X_4= \langle v_{21}-v_{43}, rv_{31}-v_{42},
  v_{41}, v_{51}, v_{32}, v_{52}, v_{53}, v_{54}\rangle$, with $r$ not
  a square in $\mathbb{F}_p$. Since $\psi_5(U)\subseteq X_4$, we
  conclude as before that $X_4$ is closed.

\item[5.] The subspace $X_5 = \langle v_{21}-v_{43}, v_{31}, v_{41},
  v_{32}, v_{42}, v_{52}, v_{53}, v_{54}\rangle$. In this case, $X_5$
  is \textit{not} closed: it corresponds to the amalgamated direct
  product of two groups: a $2$-nilpotent product of two cyclic groups
  of order~$p$, generated by $g_3$ and~$g_4$; and the $2$-nilpotent
  product of a cyclic group of order $p$ generated by $g_1$ and the
  direct sum of two cyclic groups of order~$p$, generated by $g_2$
  and~$g_5$. We amalgamate along the subgroup generated by
  $[g_4,g_3]$, identifying it with
  $[g_2,g_1]$. Theorem~\ref{th:centralamalgnotclosed} shows $X_5$ is
  therefore not closed.

\item[6.] The subspace $X_6 = \langle v_{21}-v_{43}, v_{31}-v_{52},
  v_{41}, v_{51}, v_{32}, v_{42}, v_{53}, v_{54}\rangle$. This
  subspace is closed, as can be verified with a simple computation
  in~GAP. Alternatively, if $X_6$ were not closed then the closure
  would contain either $v_{21}$ or $v_{31}$, but it is not hard 
  to verify that neither $w_{213}$ nor $w_{312}$ lie in
  $X_6^*$.
\end{itemize}

Moving on to the $7$-dimensional spaces, we obtain representatives of
the orbits as orthogonal complements of the twenty-two \textit{planes
  of $S$} listed in \cite{brahanalines}*{pp. 547\ndash 548}. We
present them in the same order as Brahana. The first six orbits
correspond to groups~$G$ with $Z(G)\neq [G,G]$;
this allows us reduce the problem to
a subspace with $n=4$ and codimension~$3$, all of which are
necessarily closed as already noted. In all six cases, $u_5$
corresponds to a central element:

\begin{itemize}
\item[1.] The subspace $X_1=\langle v_{41}, v_{51}, v_{42}, v_{52},
  v_{43}, v_{53}, v_{54}\rangle$. 

\item[2.] The subspace $X_2 = \langle v_{51}, v_{32}, v_{42}, v_{52},
  v_{43}, v_{53}, v_{54}\rangle$. 

\item[3.] The subspace $X_3 = \langle v_{41}, v_{51}, v_{32}, v_{42},
  v_{52}, v_{53}, v_{54}\rangle$. 

\item[4.] The subspace $X_4 = \langle v_{21}-v_{43}, v_{51}, v_{32},
  v_{42}, v_{52}, v_{53}, v_{54}\rangle$. 

\item[5.] The subspace $X_5 = \langle v_{21}-v_{43}, v_{41}, v_{51},
  v_{32}, v_{52}, v_{53}, v_{54}\rangle$.

\item[6.] The subspace $X_6 = \langle v_{21}-v_{43}, rv_{31}-v_{42},
  v_{51}, v_{32}, v_{52}, v_{53}, v_{54}\rangle$, with $r$ not a
  square in $\mathbb{F}_p$. 
\end{itemize}

The next fifteen orbits correspond to subspaces that are closed; this
is easy to determine using GAP, and not hard to verify by hand as
well (either by applying one of our theorems, or by explicit
computation). We list them without comment and leave routine (though
often tedious) verification that they are indeed closed to the interested
reader:

\begin{itemize}
\item[7.] The subspace $X_7 = \langle v_{21}-v_{43} + rv_{53},
  v_{31}-v_{52}, v_{41}, v_{51}, v_{32}, v_{42}-v_{53},
  v_{54}\rangle$, where $x^3 + rx - 1$ is irreducible
  over~$\mathbb{F}_p$. 

\item[8.] The subspace $X_8 = \langle v_{21}-v_{43}, v_{31}-v_{52},
  v_{41}+rv_{32}, v_{51}, v_{42}, v_{53}, v_{54}\rangle$, with $r$ not
  a square in~$\mathbb{F}_p$.

\item[9.] The subspace $X_9 = \langle v_{21}-v_{43}, v_{31}-v_{52},
  v_{41}, v_{51}, v_{32}-v_{54}, v_{42}, v_{53}\rangle$.

\item[10.] The subspace $X_{10} = \langle v_{21}-v_{43},
  v_{31}-v_{52}, v_{41}, v_{51}, v_{42}, v_{53}, v_{54}\rangle$.

\item[11.] The subspace $X_{11}=\langle v_{21}-v_{43}, v_{31}-v_{52},
  v_{41}, v_{51}, v_{32}, v_{53}, v_{54}\rangle$.

\item[12.] The subspace $X_{12} = \langle v_{21}-v_{43},
  v_{31}-v_{52}, v_{51}, v_{32}, v_{42}, v_{53}, v_{54}\rangle$.

\item[13.] The subspace $X_{13} = \langle v_{21}-v_{43},
  v_{31}-v_{52}-rv_{42}, v_{41}, v_{51}, v_{32}, v_{53},
  v_{54}\rangle$, with $r$ not a square in~$\mathbb{F}_p$.

\item[14.] The subspace $X_{14} = \langle v_{21}-v_{43},
  v_{31}-v_{42}, v_{41}-rv_{51}, v_{32}, v_{42}, v_{53},
  v_{54}\rangle$, $r\neq 0$.

\item[15.] The subspace $X_{15} = \langle v_{21}-v_{43},
  v_{31}-v_{52}, v_{41}, v_{51}, v_{32}, v_{42}, v_{53}\rangle$. 

\item[16.] The subspace $X_{16} = \langle v_{31}-v_{52}, v_{41},
  v_{51}, v_{32}, v_{42}, v_{53}, v_{54}\rangle$.

\item[17.] The subspace $X_{17} = \langle v_{21}-v_{41}-v_{43},
  v_{31}-v_{52}, v_{51}, v_{32}, v_{42}, v_{53}, v_{54}\rangle$.

\item[18.] The subspace $X_{18} = \langle v_{21}-v_{31}-v_{43} +
  v_{52}, v_{41}, v_{51}, v_{32}, v_{42}, v_{53}, v_{54}\rangle$.

\item[19.] The subspace $X_{19} =\langle v_{31}, v_{41}, v_{51},
  v_{32}, v_{42}, v_{52}, v_{43}\rangle$. 

\item[20.] The subspace $X_{20} = \langle v_{21}-v_{43}, v_{31},
  v_{41}, v_{51}, v_{42}, v_{52}, v_{54}\rangle$.

\item[21.] The subspace $X_{21} = \langle v_{21}-v_{43}, v_{31},
  v_{41}, v_{51}, v_{32}, v_{42}, v_{54}\rangle$.
\end{itemize}

The twenty-second and final orbit corresponds to an amalgamated direct
product of the $2$-nilpotent product of two cyclic groups of
order~$p$, generated by $g_1$ and $g_2$, with the $2$-nilpotent
product of three cyclic groups of order~$p$, generated by $g_3$,
$g_4$, and $g_5$, amalgamating by identifying the commutator
$[g_2,g_1]$ with $[g_4,g_3]$. Thus, by Theorem~\ref{th:centralamalgnotclosed}
it gives the only nonclosed
subspace of dimension~$7$ when $n=5$ (up to the action of $GL(5,p)$):
\begin{itemize}
\item[22.] The subspace $X_{22} = \langle v_{21}-v_{43}, v_{31},
  v_{41}, v_{51}, v_{32}, v_{42}, v_{52}\rangle$. 
\end{itemize}

We then obtain:

\begin{theorem} Let $G$ be a minimally $5$-generated $p$-group of
  class at most two and exponent~$p$. Then $G$ is one and only one of
  the following:
\begin{itemize}
\item[(i)] Isomorphic to a direct product $E\times C_p$, where $E$ is
  the extra-special $p$-group of order $p^5$ and exponent~$p$;
\item[(ii)] Isomorphic to the amalgamated direct product
\[ \Bigl( \langle x_1\rangle \amalg^{\mathfrak{N}_2}\langle
x_2\rangle\Bigr) \times_{\phi} \Bigl(\langle
 x_3\rangle\amalg^{\mathfrak{N}_2} \langle x_4\rangle
 \amalg^{\mathfrak{N}_2} \langle x_5\rangle\Bigr),\]
with each $x_i$ of order~$p$, and $\phi([x_2,x_1])=[x_4,x_3]$; 
\item[(iii)] Isomorphic to the amalgamated direct product
\[ \Bigl(\langle x_1\rangle \amalg^{\mathfrak{N}_2} \langle x_2\rangle
\Bigr) \times_{\phi} \Bigl(\bigl(\langle
x_3\rangle\amalg^{\mathfrak{N}_2}\langle x_4\rangle \amalg^{\mathfrak{N}_2} 
\langle x_5\rangle\bigr)/\langle [x_5,x_4]\rangle \Bigr),\]
with each $x_i$ of order~$p$ and $\phi([x_2,x_1]) = [x_4,x_3]$; 
\item[(iv)] Capable.
\end{itemize}
\label{thm:classifminfivegen}
\end{theorem}

If we recall that the extraspecial group of order $p^5$ and
exponent~$p$ is obtained by taking the central product 
of two nonabelian groups of order $p^3$ and exponent~$p$ (more precisely,
a central product) , we combine Theorems~\ref{th:fourgeneratedcase}
and~\ref{thm:classifminfivegen} into a single statement:

\begin{theorem}
Let $G$ be a $5$-generated group of class at most $2$ and
exponent~$p$. Then $G$ is one and only one of the following:
\begin{itemize}
\item[(i)] Nontrivial cyclic;
\item[(ii)] Isomorphic to an amalgamated direct product
  $G_1\times_{\phi} G_2$ of two nonabelian groups, amalgamating a
  nontrivial cyclic subgroup of the commutator subgroups.
\item[(iii)] Capable.
\end{itemize}
\label{thm:fullclassiffivegen}
\end{theorem}

\subsection*{An alternative geometrical proof.}

The only part of the proof of Theorem~\ref{th:fourgeneratedcase} that
does not follow by applying Theorem~\ref{th:largeenoughcomm} is the
case of $n=4$ and $\dim(X)=5$. I would like to present an alternative
proof for this case. The reason for doing so is that a key step in the
proof is geometric rather than algebraic. This highlights what I
believe to be one of the potential strengths of the approach through
linear algebra, namely that by casting the problem in terms of linear
algebra we have an array of tools that can be brought to bear on the
problem, most particularly geometric tools whose application may not
be so easy to discern when the problem is presented in terms of
commutators. This can also be seen in \cite{brahanalines}, though it
will not be apparent in our presentation above. The geometric part of
the argument is due to David McKinnon.

Fix $n=4$. Given a vector $\mathbf{u}\in U$, $\mathbf{u}\neq
\mathbf{0}$, we obtain a subspace $\psi_{\mathbf{u}}(U)$ of $V$;
it is easy to
verify that this subspace is $3$-dimensional. Moreover, any nontrivial
scalar multiple of $\mathbf{u}$ will yield the same subspace.  Thus
we obtain a map from the one dimensional subspaces
of~$V$ (which form projective $3$-space over $\mathbb{F}_p$) to
$Gr(3,V)$; that is, a map $\Psi\colon\mathbb{P}^3\to
Gr(3,V)$. Explicitly, given
$[\alpha_1\colon\alpha_2\colon\alpha_3\colon\alpha_4]\in\mathbb{P}^3$,
we associate to it the subspace
$U\wedge(\alpha_1u_1+\alpha_2u_2+\alpha_3u_3+\alpha_4u_4)$. 

Turning now to $\ker(\Phi)$, where $\Phi$ is the map from
Definition~\ref{defn:defofPhi}, it is easy to verify that if
$\mathbf{p}\in\mathbb{P}^3$, $\mathbf{v}\in V$ is an arbitrary vector,
and $X=\langle\Psi(\mathbf{p}),\mathbf{v}\rangle$, then $X^4\cap{\rm
  ker}\Phi$ is trivial if and only if
$\mathbf{v}\in\Psi(\mathbf{p})$. 

Let $(\mathbf{v}_1,\mathbf{v}_2,\mathbf{v}_3,\mathbf{v}_4)$ be a
nontrivial element of ${\rm ker}(\Phi)$. The subspace of~$V$ spanned by
$\mathbf{v}_1,\mathbf{v}_2,\mathbf{v}_3,\mathbf{v}_4$ is exactly
$3$-dimensional. This gives a mapping from the one-dimensional
subspaces of ${\rm ker}(\Phi)$ to the $3$-dimensional subspaces of~$V$,
\[\Upsilon\colon Gr(1,{\rm ker}(\Phi))\to Gr(3,V).\] 
We can identify
$Gr(1,{\rm ker}(\Phi))$ with $\mathbb{P}^3$ (or to be more
precise, with $\mathbb{P}^{\binom{4}{3}-1}$): we have a bijection
between a basis for ${\rm ker}(\Phi)$ and the choice of triples from
$\{1,2,3,4\}$, so a point
$[\alpha_1\colon\alpha_2\colon\alpha_3\colon\alpha_4]$ can be
identified, for example, with the element $\alpha_1\mathbf{v}_{(234)}
+ \alpha_2\mathbf{v}_{(134)} + \alpha_3\mathbf{v}_{(124)} +
\alpha_4\mathbf{v}_{(123)}$ (using the notation from the proof of
Proposition~\ref{prop:kerofPhi}). Thus we have two maps with domain
$\mathbb{P}^3$ and codomain $Gr(3,V)$.

Consider a $5$-dimensional subspace $X$ of~$V$. From
Theorem~\ref{th:upperandlowerbounds}, we know that $\dim(X^*)=18$,
$\dim(X^*)=19$, or $\dim(X^*)=20$. Since the only subspace of $V$ that
properly contains~$X$ is $V$ itself, we deduce that $X$ is closed if and
only if $X^4\cap\ker(\Phi_4)$ is nontrivial; that is, a $5$-dimensional
subspace of~$V$ is closed if and only if there exists
$\mathbf{q}\in\mathbb{P}^3$ such that $\Upsilon(\mathbf{q})\subseteq
X$. As noted above, if $X$ 
contains $\Psi(\mathbf{p})$ for some $\mathbf{p}\in\mathbb{P}^3$, then
$X$ will be closed. The result we want is the converse:
that if $X$ is closed, then there exists $\mathbf{p}\in\mathbb{P}^3$
such that $\Psi(\mathbf{p})\subseteq X$. This result can be
established by considering the maps $\Psi$, $\Upsilon$, and using a
little algebraic geometry.

Suppose first we are working over the algebraic closure
$\overline{\mathbb{F}_p}$ of $\mathbb{F}_p$ (so we can do algebraic
geometry). The maps $\Psi\colon\mathbb{P}^3\to Gr(3,V)$ and
$\Upsilon\colon\mathbb{P}^3\to Gr(3,V)$ are both regular maps, since
they are defined everywhere and are locally (relative to the Zariski
topology) determined by rational functions on the coordinates. We
define two subsets of the algebraic variety $Gr(4,V)\times
\mathbb{P}^3$, namely:
\[
A  =  \bigl\{ (X,\mathbf{p})\,\bigm|, \Psi(\mathbf{p})\subseteq
X\bigr\},\qquad\mbox{and}\qquad
B  =  \bigl\{ (X,\mathbf{q})\,\bigm|\, \Upsilon(\mathbf{q})\subseteq
X\bigr\}.
\]
Since both $\Psi$ and $\Upsilon$ are regular, both $A$ and~$B$ are
closed subvarieties of $Gr(4,V)\times \mathbb{P}^3$. If we now
consider the projections,
\begin{alignat*}{2}
p_1 &\colon Gr(4,V)\times \mathbb{P}^3 &&\to Gr(4,V)\\
p_2 &\colon Gr(4,V)\times \mathbb{P}^3 &&\to \mathbb{P}^3,
\end{alignat*}
we obtain maps from each of $A$ and~$B$ into $Gr(4,V)$
and~$\mathbb{P}^3$, respectively. The maps to $\mathbb{P}^3$ are
surjections, and the fibers all have dimension $2$ because the fiber
over $\mathbf{p}$ (resp.~over $\mathbf{q}$) is the set of all
$4$-dimensional subspaces of~$V$ that contain the $3$-dimensional
space $\Psi(\mathbf{p})$ (resp.~$\Upsilon(\mathbf{q})$); this set is
isomorphic to the set of lines in the quotient space
$V/\Psi(\mathbf{p})$ (resp.~$V/\Upsilon(\mathbf{q})$), which in turn
is isomorphic to the projective plane $\mathbb{P}^2$, hence
$2$-dimensional.

The maps are also smooth, so we have smooth maps of fiber
dimension~$2$ over a smooth $3$-dimensional variety; this means that
both $A$ and $B$ are of dimension $3+2=5$.

Consider now the projections to $Gr(4,V)$. We know that $p_1(A)$ and
$p_1(B)$ are irreducible subvarieties of $Gr(4,V)$ of dimension at
most $5$, and that $p_1(A)$ is contained in $p_1(B)$ (to see this last
assertion, note that if $(X,\mathbf{p})\in A$, then $X^4\cap{\rm ker}(\Phi)$ is
nontrivial, so there exists $\mathbf{q}$ such that $(X,\mathbf{q})\in B$).
If we can show that $p_1(A)$ is of dimension
exactly~$5$, then the irreducibility of $B$ will imply that
$p_1(A)=p_1(B)$. To show that $p_1(A)$ is of dimension exactly~$5$ it
is enough to show that it is generically finite; for this it is, in
turn, enough
to show there is at least one $X\in Gr(4,V)$ such that $p_1^{-1}(X)$
is nonempty and finite. But in fact $p_1^{-1}(X)$ has at most one
element, for if $\mathbf{p}\neq\mathbf{q}$, then
$\langle\Psi(\mathbf{p}),\Psi(\mathbf{q})\rangle$ contains $U\wedge
U'$ with $U'$ of dimension~$2$ (spanned by the lines corresponding to
$\mathbf{p}$ and~$\mathbf{q}$); and this subspace is of
dimension~$5$. So the conclusion that $p_1(A)=p_1(B)$ holds over
$\overline{\mathbb{F}_p}$.  

Thus, if $\mathbf{q}\in\mathbb{P}^3$ and $X\in Gr(4,V)$ are such that
$(X,\mathbf{q})\in B$, then there exists $\mathbf{p}\in\mathbb{P}^3$
such that $(X,\mathbf{p})\in A$. We want to show that if $\mathbf{q}$
and $X$ are defined over $\mathbb{F}_p$, then $\mathbf{p}$ is also
defined over~$\mathbb{F}_p$. If we apply a Galois automorphism to the
varieties over $\overline{\mathbb{F}_p}$, both $X$ and $\mathbf{q}$
are fixed, and every conjugate of $\mathbf{p}$ will also satisfy the
conclusion; however, we know that if
$(X,\mathbf{p}),(X,\mathbf{p}')\in A$, then $\mathbf{p}=\mathbf{p}'$,
by the argument above, so we conclude that $\mathbf{p}$ is fixed by
all Galois automorphisms of $\overline{\mathbb{F}_p}$, proving it is
indeed defined over $\mathbb{F}_p$. 

This proves what we want: if $X'$ is a $5$-dimensional subspace
of~$V$, and if there exists $(X,\mathbf{p})\in A$ such that
$X\subseteq X'$, then $X'$ is closed. And if $X'$ is closed, then
there exists $(X,\mathbf{q})\in B$ with $X\subseteq X'$, and this
implies the existence of $\mathbf{p}\in \mathbb{P}^3$ with
$(X,\mathbf{p})\in A$. Thus, $X'$ is closed if and only if it contains
$\Psi(\mathbf{p})$ for some $\mathbf{p}\in \mathbb{P}^3$. In terms of
the groups, it says that a group $G$ of class two, exponent~$p$, with
$G^{\rm ab}$ of rank~$4$ and $[G,G]$ of order~$p$ is capable if and
only if $[G,G]\neq Z(G)$. That is, a $5$-dimensional subspace of~$V$
is closed if and only if the corresponding group is not extra-special.

\begin{remark} The proof that there exist $\mathbf{p}\in\mathbb{P}^3$
  such that $\Psi(\mathbf{p})\subseteq X$ if and only if there exists
  $\mathbf{q}\in\mathbb{P}^3$ such that $\Upsilon(\mathbf{q})\subseteq
  X$ can be done purely at an algebraic level; see for
  example~\cite{capablep2}. However, I find the geometric argument
  more satisfying.
\end{remark}

\section{Final remarks and questions.}\label{sec:finalsec}

The gap between our necessary and sufficient condition, unfortunately,
grows with $n$. Thus, when $n=4$ the necessary condition allows us to
discard the case $\dim(X)=5$ (when $X$ does not contain
$\Psi(\mathbf{u})$ for some nontrivial $\mathbf{u}\in U$), while the
sufficient condition handles the remaining cases with $\dim(X)\leq
4$. When we move to $n=5$, however,
Theorem~\ref{th:strengthenedheinnik} deals only with $\dim(X)=9$
(where we are reduced to the case $n=4$ as above), while
Corollary~\ref{cor:Xsmallenough} dispatches $\dim(X)\leq 6$, leaving
us to deal with the cases of dimension $7$ and~$8$. Our success above
was achieved thanks to the careful geometric analysis of
Brahana. With $n=6$, Theorem~\ref{th:strengthenedheinnik} would handle
$\dim(X)=14$ and $15$ (we can either reduce to a smaller~$n$, or else
the subspace is not closed), and Corollary~\ref{cor:Xsmallenough}
deals with $\dim(X)\leq 7$, leaving now six potential dimensions
open. As $n$ increases, the gap between our numerical necessary and
sufficient conditions continues to widen, making them less and less
useful.

Heineken proved that the necessary condition is sharp, in that there
are examples of capable groups in which the inequality from
Theorem~\ref{th:strengthenedheinnik} is an equality. We might likewise
wonder if we can sharpen the sufficient condition. There is some hope
this might be possible, since for example
Corollary~\ref{cor:Xpreciselysmallenough} considers all subspaces of
dimension strictly larger than~$X$, while
Proposition~\ref{prop:strictlylarger} only requires us to look at
those subspaces that properly contain~$X$. So we ask:

\begin{question}
Is the sufficient condition in Corollary~\ref{cor:Xsmallenough} sharp?
That is, is it true that for all $n>1$, if $m$ is the smallest integer
such that $0<m<\binom{n}{2}$ and $f(m+1)\geq n$, then there exists $X<V$ such that
$\dim(X)=m$ and $X\neq X^{**}$? 
\label{quest:sufficientsharp}
\end{question}

Note that if we can find a non-closed subspace $X<V(n)$ with
$\dim(X)=k$, then we can find non-closed subspaces $X'<V(n)$ with
$\dim(X')=r$ for any $r$ satisfying $k\leq r < \binom{n}{2}$:
enlarge $X$ by adding vectors from $X^{**}$ not in $X$ until we obtain
a subspace of codimension one  in its closure; and then continue by
adding vectors that do not lie in $X^{**}$ until we obtain a subspace
of codimension~$1$ in $V(n)$. So it is enough to ask about the
smallest value of $m$ with $\dim(X)=m$ and $X\neq X^{**}$. 

For $m\leq 5$, the answer to Question~\ref{quest:sufficientsharp} is
affirmative. 
Consider then $n=6$; by taking an amalgamated central product of the
$2$-nilpotent product of two cyclic groups of order~$p$ and the
$2$-nilpotent product of $4$ cyclic groups of order~$p$ we can find a
non-closed subspace of dimension $9$; the least $m$, however, for
which $f(m+1)\geq 6$ is $m=8$. So we ask:

\begin{question}
Is there a subspace $X$ of $V(6)$ with $\dim(X)=8$ and $X\neq X^{**}$?
\end{question}

I do not know the answer to this question yet; I have done a brute
force search using GAP and have found no examples yet. However, though
the search has considered over one hundred million subspaces, the total
number of eight dimensional subspaces of the fifteen dimensional
space~$V(6)$ is approximately $9.3\times 10^{26}$ if we work over
$\mathbb{F}_3$, so the negative results in this search are hardly significant.

In general, given $n$, taking an amalgamated central product of two
relatively free groups, one of rank~$2$ and one of rank $n-2$, and
identifying a subgroup of order~$p$ from each, yields a non-closed
subspace of dimension $2n-3$ (we need $2(n-2)$ relations to state the
generators from one relatively free group commute with those of the
other, and one relation to identify one nontrivial commutator from
each factor with each other). This is the smallest nonclosed subspace
we can obtain with amalgamated direct products, but it is not
necessarily the smallest non-closed. For example, with $n=8$, the
amalgamated direct product yields a non-closed $X$ of dimension~$13$;
but if we take the amalgamated coproduct of two extra-special groups
of order $p^5$ and exponent~$p$, identifying the commutator subgroups,
we obtain a non-closed~$X$ of dimension~$11$ (we will need $5$
relations to describe each of the extra-special groups, plus one
relation to identify the two commutator subgroups). This eleven
dimensional subspace still falls two short of the $9$-dimensional
example we would need for $n=8$ if Corollary~\ref{cor:Xsmallenough} is
indeed sharp.

\section*{Acknowledgements}

In addition to the theorems from \cite{brahanalines}, the work of
Brahana helped to clarify many notions with which I had been
playing; I thank Prof.~Mike Newman very much for bringing the work of Brahana to my
attention and other helpful references. I also thank Michael Bush for his
help. I especially thank David McKinnon for many stimulating
conversations, most of the geometry that appears in this work, and for
his help in finding a formula for the function~$f(m)$. Part of this
work was conducted while the author was on a brief visit to the
University of~Waterloo at the invitation of Prof.~McKinnon; I am very
grateful to him for the invitation, and to the Department of Pure
Mathematics and the University of Waterloo for the great hospitality I
received there. The work was begun while the author was at the
University of Montana, and finished at the University of Louisiana
in~Lafayette.

\end{document}